\documentclass[11pt,a4paper,reqno]{amsart}
\pdfoutput=1

\usepackage[T1]{fontenc}
\usepackage[utf8]{inputenc}
\usepackage{lmodern}
\usepackage[english]{babel}
\usepackage[activate]{pdfcprot}

\usepackage{amssymb}
\usepackage{amsfonts}
\usepackage{mathtools}
\usepackage{enumitem}
\setlist{beginpenalty=100}

\usepackage[pdftex,colorlinks,linkcolor=black,citecolor=black,filecolor=black]{hyperref}

\newtheorem{thm}{Theorem}[section]
\newtheorem{lem}[thm]{Lemma}
\newtheorem{cor}[thm]{Corollary}
\newtheorem{prop}[thm]{Proposition}

\theoremstyle{definition}
\newtheorem{de}[thm]{Definition}
\newtheorem{ex}[thm]{Example}
\newtheorem*{qu}{Question}

\theoremstyle{remark}
\newtheorem{rem}[thm]{Remark}

\numberwithin{equation}{section}

\DeclarePairedDelimiter\abs{\lvert}{\rvert}
\DeclarePairedDelimiter\norm{\lVert}{\rVert}

\newcommand{\R}{\mathbb{R}}
\newcommand{\N}{\mathbb{N}}
\newcommand{\Z}{\mathbb{Z}}
\newcommand{\Q}{\mathbb{Q}}
\newcommand{\C}{\mathbb{C}}

\DeclareMathOperator{\SL}{SL}
\DeclareMathOperator{\PGL}{PGL}
\DeclareMathOperator{\GL}{GL}
\DeclareMathOperator{\SO}{SO}
\DeclareMathOperator{\Orth}{O}
\DeclareMathOperator{\Mat}{Mat}
\DeclareMathOperator{\Res}{Res}

\newcommand{\dd}{\mathop{}\!\mathrm{d}}
\newcommand{\calS}{\mathcal{S}}
\newcommand{\calT}{\mathcal{T}}
\newcommand{\leqs}{\leqslant}
\newcommand{\geqs}{\geqslant}
\newcommand{\df}{\coloneqq}
\DeclareMathOperator{\Lie}{Lie}
\DeclareMathOperator{\Aut}{Aut}
\DeclareMathOperator{\diag}{diag}
\DeclareMathOperator{\Ad}{Ad}
\DeclareMathOperator{\End}{End}
\DeclareMathOperator{\supp}{supp}
\DeclareMathOperator{\Stab}{Stab}
\DeclareMathOperator{\Zcl}{Zcl}
\newcommand{\acts}{{\boldsymbol{\cdot}}}

\providecommand\for{}
\newcommand\SetSymbol[1][]{%
\nonscript\:#1\vert
\allowbreak
\nonscript\:
\mathopen{}}
\DeclarePairedDelimiterX\set[1]{\lbrace}{\rbrace}{%
\renewcommand\for{\SetSymbol[\delimsize]}
#1
}
\allowdisplaybreaks

\begin{document}

\title[Expanding measures: Random walks and rigidity]{Expanding measures: Random walks and rigidity on homogeneous spaces}

\author{Roland Prohaska}
\address{Departement Mathematik, ETH Z\"{u}rich, R\"{a}mistrasse 101, 8092 Z\"{u}rich, Switzerland}
\email{roland.prohaska@math.ethz.ch}
%\thanks{Part of the author's PhD thesis at ETH Z\"{u}rich.}

\author{Cagri Sert}
\address{Institut f\"{u}r Mathematik, Universit\"{a}t Z\"{u}rich, Winterthurerstrasse 190, 8057 Z\"{u}rich, Switzerland}
\email{cagri.sert@math.uzh.ch}

\author{Ronggang Shi}
\address{Shanghai Center for Mathematical Sciences, Jiangwan Campus, Fudan University, No.2005 Songhu Road, Shanghai, 200438, China}
\email{ronggang@fudan.edu.cn}

% General info


\date{\today}

\begin{abstract}
Let $G$ be a real Lie group, $\Lambda<G$ a lattice and $H\leqs G$ a connected semisimple subgroup without compact factors and with finite center. We define the notion of $H$-expanding measures $\mu$ on $H$ and, applying recent work of Eskin--Lindenstrauss, prove that $\mu$-stationary probability measures on $G/\Lambda$ are homogeneous. Transferring a construction by Benoist--Quint and drawing on ideas of Eskin--Mirzakhani--Mohammadi, we construct  Lyapunov/Margulis functions to show that $H$-expanding random walks on $G/\Lambda$ satisfy a recurrence condition and that homogeneous subspaces are repelling. Combined with a countability result, this allows us to prove equidistribution of trajectories in $G/\Lambda$ for $H$-expanding random walks and to obtain orbit closure descriptions. Finally, elaborating on an idea of Simmons--Weiss, we deduce Birkhoff genericity of a class of measures with respect to some diagonal flows and extend their applications to Diophantine approximation on similarity fractals to a non-conformal and weighted setting.
\end{abstract}

\maketitle

\tableofcontents

\section{Introduction}\label{sec;introduction}
Originally motivated by applications to number theory, the rigidity properties of subgroup actions on a homogeneous space $X=G/\Lambda$, where $G$ is a real Lie group and $\Lambda<G$ a discrete subgroup, have been an active field of research over the last fifty years. Among the first striking results was Margulis' resolution of the Oppenheim conjecture~\cite{mar2,mar1} via a reformulation into an orbit closure problem for the action of $\SO(2,1)$ on $\SL_3(\R)/\SL_3(\Z)$ noticed by Raghunathan. Raghunathan had conjectured, more generally, that orbit closures for unipotent subgroups are closed orbits of larger subgroups. After more partial results by Dani, Margulis, and Shah, Raghunathan's conjecture was settled in full generality in celebrated work of Ratner~\cite{rat2,rat1,rat3,rat4}.

In absence of unipotent elements, the dynamics of subgroup actions are harder to understand---already the case of actions on a torus $\mathbb{T}^d=\R^d/\Z^d$ by non-amenable subgroups of $\SL_d(\Z)$ poses serious challenges. The very first difficulty arising in this setup is the potential lack of invariant measures. What has proved to be a fruitful approach for overcoming this issue is taking a probabilistic viewpoint of random walks and stationary measures, techniques mainly pioneered by Furstenberg starting in the sixties~\cite{fk60,furstenberg.positivity,furstenberg.poisson,furstenberg.stiffness}.
Using this random walks approach, Guivarc'h--Starkov~\cite{guivarch-starkov} made first contributions to understanding the action of $\Gamma<\SL_d(\Z)$ on $\mathbb{T}^d$, and  Bourgain--Furman--Lindenstrauss--Mozes~\cite{bflm} proved a quantitative result which answered many remaining questions.

For subgroup actions on a general homogeneous space $X=G/\Lambda$, a major breakthrough came with a series of papers by Benoist--Quint~\cite{bq11,bq12,bq13,bq132}. Applying several novel techniques, they were able to give a complete classification of stationary measures, descriptions of orbit closures, and prove equidistribution statements for random walks under the assumption of semisimplicity of the Zariski closure of the acting group $\Gamma$. One crucial new ingredient in the proof of their measure classification result is the so-called ``exponential drift'' argument (as compared to the ``polynomial drift'' argument of Ratner), which was further developed in the seminal work of Eskin--Mirzakhani~\cite{eskin-mirzakhani} on stationary measures for the $\SL_2(\R)$-action on moduli space. Bringing back to homogeneous dynamics ideas from the setting of random walks on moduli space, Eskin--Lindenstrauss~\cite{el} have recently obtained a theorem which generalizes the measure classification results of Benoist--Quint. 

The aim of this paper is to further advance the study of subgroup actions and random walks on homogeneous spaces, systematically dropping the assumption that the Zariski closure of the acting group $\Gamma$ is semisimple. 
We will introduce and study a new class of measures $\mu$ supported on a connected semisimple subgroup $H\leqs G$ without compact factors and with finite center that we call $H$-expanding measures. These are defined by an expansion condition in non-trivial irreducible finite-dimensional representations of $H$ resembling the conclusion of the fundamental result of Furstenberg on the positivity of the top Lyapunov exponent. In particular, this class contains the Zariski dense measures underlying the work of Benoist--Quint. After deducing a measure classification result based on the progress by Eskin--Lindenstrauss~\cite{el}, we will prove orbit closure descriptions, as well as recurrence and equidistribution results for the random walk on $G/\Lambda$ given by an $H$-expanding probability measure $\mu$. Finally, taking advantage of the generality of $H$-expanding measures, these main results will be used to also obtain new equidistribution statements for diagonalizable flows, which in turn have implications for Diophantine approximation problems on fractals.

To introduce the notion of $H$-expansion, we say that
a Borel probability measure $\mu$ on $\GL_d(\R)$ is  \emph{uniformly expanding} if for every nonzero $v \in \R^d$, we have
\begin{align*}
\liminf_{n \to \infty} \frac{1}{n}\log \norm{g_n \dotsm g_1 v}>0
\end{align*}
for $\mu^{\N}$-almost every (a.e.) sequence $(g_1,g_2,\dots)$. 
A probability measure $\mu$ on $H$ is said to be \emph{$H$-expanding} if for every finite-dimensional representation $(\rho,V)$ of $H$ without nonzero $H$-fixed vectors, the measure $\rho_*\mu$ is uniformly expanding, where $\rho_*\mu $ denotes the pushforward of $\mu$ by $\rho$. We are going to elaborate on this definition and give non-trivially equivalent formulations in  \S\ref{sec;def}.

Ranging over all finite-dimensional representations, the $H$-expansion property of a probability measure $\mu$ on $H$ is a universal condition and as such ensures validity of our results for an arbitrary embedding $H\hookrightarrow G$ and any lattice $\Lambda<G$.
This universality notwithstanding, the class of $H$-expanding measures contains an abundance of interesting examples:
\begin{itemize}
    \item Zariski dense measures (\S\ref{subsec;Z-dense}): If the closed subgroup $\Gamma_{\mu}$ of $H$ generated by the support of $\mu$ has Zariski dense image in $\Ad(H)$ and $\mu$ satisfies a moment condition, then $\mu$ is $H$-expanding as a consequence of Furstenberg's theorem on positivity of the top Lyapunov exponent.
    \item Measures on parabolic groups (\S\ref{sec;parabolic}): We give a general criterion for $H$-expansion of a measure $\mu$ on a parabolic subgroup of $H$ and, using the notion of expanding cone introduced by the third-named author~\cite{s15}, explicitly exhibit a class of examples of such measures. For the sake of concreteness, let us mention here that, for example, our results directly imply that any probability measure on $H=\SL_4(\R)$ with support consisting of the five matrices
\begin{align*}
\begin{pmatrix}
 2\\
   & 2\\
   & & 1\\
   & & & 1/4
\end{pmatrix},\,
\begin{pmatrix}
 2 & 1\\
 1 & 1\\
   & & 1\\
   & & & 1
\end{pmatrix},\,
\begin{pmatrix}
 1 & 1\\
 1 & 2\\
   & & 1\\
   & & & 1
\end{pmatrix},\,
\begin{pmatrix}
 1\\
   & 1 & 1\\
   & & 1\\
   & & & 1
\end{pmatrix},\,
\begin{pmatrix}
 1\\
   & 1\\
   & & 1 & 1\\
   & & & 1
\end{pmatrix}
\end{align*}
is $H$-expanding.
\item Epimorphic subgroups (\S\ref{sec;solv_epi}): The closed subgroup $\Gamma_\mu$ generated by the support of $\mu$ is necessarily an epimorphic subgroup of $H$ when $\mu$ is $H$-expanding. Conversely, thanks to the work of Bien--Borel~\cite{bien-borel1} and its subsequent developments, we will see that many epimorphic subgroups of $H$ support $H$-expanding measures. For example, any $\R$-split simple group $H$ admits distinguished three-dimensional epimorphic subgroups for which this is the case, showing that $H$-expanding measures may live on subgroups which are very small compared to $H$ itself. See also Corollary~\ref{cor;epi_class}.
\end{itemize}

Under various weaker assumptions than $H$-expansion, not all of our conclusions hold in full strength. For instance, requiring uniform expansion only in the adjoint representation, homogeneity of stationary measures can fail, as an example at the end of~\cite[\S1.2]{el} shows. For unipotent random walks, recurrence is not always guaranteed~\cite[\S10.2.1]{breuillard}. On the other hand, in the particular case of measures on parabolic subgroups, slightly weaker expansion properties were first used in the work of Simmons--Weiss~\cite{sw} and subsequently in~\cite{prohaska-sert} to prove measure rigidity and equidistribution results in a setting corresponding to the case $H=G$ in our framework. See also Remark~\ref{rmk;finitely_many}.

We next introduce the terminology necessary to state our main results. Given a continuous action of a locally compact second countable group $G$ on a locally compact second countable metrizable space $X$, a probability measure $\nu$ on $X$ is said to be \emph{$\mu$-stationary} if $\nu=\mu*\nu$, where the convolution is defined by
\begin{align*}
\int_X f\dd(\mu*\nu)=\int_X\int_Gf(gx)\dd\mu(g)\dd\nu(x)
\end{align*}
for non-negative Borel functions $f$ on $X$.
A $\mu$-stationary probability measure  $\nu$ is
said to be \emph{$\mu$-ergodic} if it is extremal in the convex set of $\mu$-stationary probability measures. 

Now let $G$ be a real Lie group, $\Lambda<G$ a discrete subgroup and $X=G/\Lambda$.
A probability measure  $\nu$ on $X$   is said to be \emph{homogeneous} if
there exists $x\in X$
and a closed subgroup $N$ of $G$ preserving $\nu$ such that  $\nu(Nx)=1$.
In this case, the orbit $Nx$ is automatically closed and is called a \emph{homogeneous subspace} of $X$. It is equivalent to require that $\nu$ assigns full measure to an orbit of its stabilizer group
\begin{align*}
\Stab_G(\nu)=\set{g\in G\for g_*\nu=\nu}.
\end{align*}
This gives a one-to-one correspondence between homogeneous measures on $X$ and homogeneous subspaces of $X$. 
For a closed subgroup $\Gamma$ of 
$G$, a homogeneous subspace $Y$ of $X$ is said to be \emph{$\Gamma$-ergodic} if $\Gamma$ preserves the corresponding homogeneous probability measure $\nu_Y$ and the action of $\Gamma$ on $(Y,\nu_Y)$ is ergodic.

Finally, for $g \in \GL_d(\R)$ we set 
$\operatorname{N}(g)=\max\set{\norm{g}, \norm{g^{-1}}}$ for some choice of operator norm on $\Mat_{d\times d}(\R)$. A probability measure $\mu$ on $\GL_d(\R)$ is said to have a \emph{finite first moment} if
\begin{align*}
\int \log \operatorname{N}(g) \dd\mu(g)<\infty,
\end{align*}
and to have \emph{finite exponential moments} if
\begin{align*}
\int \operatorname{N}(g)^\delta\dd\mu(g)<\infty
\end{align*}
for $\delta>0$ sufficiently small. These definitions are independent of the choice of operator norm. We say that a probability measure $\mu$ on a connected semisimple Lie group $H$ with finite center has a finite first moment or finite exponential moments if its image in a finite-dimensional representation of $H$ with finite kernel has the corresponding property. This does not depend on the choice of such a linear representation (see Lemma~\ref{lem;first_moment}).
Both moment conditions are automatically satisfied, for example, if $\mu$ has compact support. 
\subsection{Measure rigidity}\label{sec;intro.rigidity}
We start with the classification of stationary measures. Recall that given a measure $\mu$ on $H$, we denote by $\Gamma_\mu$ the closed subgroup generated by the support of $\mu$.
\begin{thm}
	\label{thm;rigidity}
	Let $\Lambda$ be  a discrete subgroup of a real Lie group  $G$. Let $H\leqs G$ be a connected semisimple subgroup without compact factors and with finite center. Let $\mu$ be a probability measure on $H$ that is $H$-expanding and has a finite first moment. Then any $\mu$-ergodic  $\mu$-stationary probability measure $\nu$ on $G/\Lambda$ is $\Gamma_\mu$-invariant and homogeneous. 
	Moreover, the connected component of $\Stab_G(\nu)$ is normalized by $H$. 
\end{thm}

Using the properties of $H$-expanding measures, the above theorem is deduced by an iterative application of the recent measure classification results of Eskin--Lindenstrauss~\cite{el};  see~\S\ref{sec;rigidity_proof}. The proof is similar to the argument Eskin--Lindenstrauss use to show that their result implies Benoist--Quint's measure classification.

In certain cases, the last conclusion of Theorem~\ref{thm;rigidity} allows us to show that $\nu$ is actually $H$-invariant; see Proposition~\ref{prop;more} and also the corollary below. For its statement, recall that a discrete subgroup $\Lambda$ is said to be a \emph{lattice} in $G$ if $X=G/\Lambda$ admits a $G$-invariant probability measure $m_X$. In this case, we refer to $m_X$ as the \emph{Haar measure} on $X$. A lattice $\Lambda$ in a connected semisimple Lie group $G$ without compact factors is said to be \emph{irreducible} if $\Lambda\cap S$ is not a lattice in $S$ for every non-trivial proper connected normal subgroup $S$ of $G$. Equivalently, $S\Lambda$ is dense in $G$ for every such $S$.

\begin{cor}\label{cor;rigidity}
Let $G$ be a connected semisimple Lie group without compact factors and with finite center and let  $\Lambda<G$ be an irreducible lattice. Let $H$ be a connected normal subgroup of $G$ of positive dimension and let $\mu$ be an $H$-expanding probability measure on $H$ with finite first moment. 
\begin{enumerate}[label=\textup{(\roman*)}]
    \item If $H\neq G$, then the Haar measure $m_X$ on $X=G/\Lambda$ is the unique $\mu$-stationary probability measure on $X$.
    \item If $H=G$, then the only $\mu$-ergodic $\mu$-stationary probability measures on $X$ are uniform measures on finite $\Gamma_\mu$-orbits and the Haar measure $m_X$ on $X$. Moreover, $m_X$ is the only non-atomic $\mu$-stationary probability measure on $X$.
\end{enumerate}
\end{cor}
We note that finite $\Gamma_\mu$-orbits do only occur when $\Gamma_\mu$ is virtually contained in a conjugate of $\Lambda$. The proof of part~(i) of the corollary above relies on Margulis' arithmeticity theorem and a careful analysis of stationary measures charging an orbit of the centralizer of $\Gamma_\mu$, which is carried out in~\S\ref{sec;centralizer}. The last statement in part~(ii) additionally requires countability of finite $\Gamma_\mu$-orbits, which follows from a general countability result for homogeneous subspaces in~\S\ref{subsec;countability}. 

\begin{rem}\label{rmk;finitely_many}
As mentioned before, the $H$-expansion condition is universal so that all our results hold for an arbitrary embedding $H\hookrightarrow G$. For a fixed Lie group $G$, it suffices to impose uniform expansion on $\rho_*\mu$ only for a finite collection of representations $(\rho,V)$ of $H$ (which depends on $G$), as the proofs show. In \S\ref{sec;exp_grass} we track which representations are needed in the case of measure classification; see Theorem~\ref{thm.exp.grass} for the precise statement. Our countability result (Proposition~\ref{prop;countability}) will also be phrased using only this finite collection of representations, allowing us to prove it without an assumption of compact generation (cf.~\cite[Proposition 2.1]{bq132}).
\end{rem}

\subsection{Recurrence and Lyapunov functions}\label{sec;rec_intro}
Now we assume in addition that $\Lambda$ is a lattice and that $\mu$ has finite exponential moments. 
Under certain assumptions including semisimplicity of the non-compact part of the Zariski closure of $\Gamma_\mu$, Eskin--Margulis~\cite{em} and later Benoist--Quint~\cite{bq12} have shown that the random walk on $X=G/\Lambda$ given by $\mu$ satisfies strong recurrence properties. If 
$\delta_x$ denotes the Dirac measure at $x\in X$ and $\mu^{* n} $  is the $n$-fold convolution power of $\mu$, these recurrence statements take the general form that $\mu^{*n}*\delta_x(M)$ is close to $1$ for large $n$, where $M\subset X$ is a certain compact set. We obtain analogous results for $H$-expanding measures.
\begin{thm}\label{thm;recurrence}
Let $\Lambda$ be  a lattice in a real Lie group  $G$. Let $H\leqs G$ be a connected semisimple subgroup without compact factors and with finite center. Let $\mu$ be an $H$-expanding probability measure with finite exponential moments on $H$. Let $Y$ be a $\Gamma_\mu$-ergodic homogeneous subspace of $X=G/\Lambda$ or the empty set. Finally, let $K_L$ be any compact subset of the centralizer $L$ of $\Gamma_\mu$ in $G$, and set $\mathcal{N}=K_LY$. Then for any compact subset $Z\subset X\setminus \mathcal{N}$ and $\delta>0$ there exists a compact subset $M_{Z,\delta}$ of $X\setminus \mathcal{N}$ such that
\begin{align*}
    \mu^{*n}*\delta_x(M_{Z,\delta})\ge 1-\delta
\end{align*}
for every $n\ge 0$ and $x\in Z$.
\end{thm}
Loosely speaking, the basic case (with $Y=\emptyset$) implies that the random walk does not spend too much time in the cusp. The general case ensures that the random walk also does not accumulate near lower-dimensional homogeneous subspaces.

This result will be proved in \S\ref{subsec;recurrence} using height functions 
on $X=G/\Lambda$ satisfying a contraction property with respect to the averaging operator $A_\mu$ defined by
\begin{align*}
A_\mu(f)(x)=\int_Gf(gx)\dd\mu(g)
\end{align*}
for non-negative Borel functions $f$ on $X$. Heuristically, if $\beta$ is a function on $X$ with values in $[0,\infty]$ such that
\begin{align}\label{contr.intro}
A_\mu(\beta)\le a\beta +b
\end{align}
for constants $a\in (0,1)$ and $b\ge 0$, then, with high probability, the dynamics of the random walk are driven towards the part of the space where $\beta$ takes values below a certain threshold, and $X_\infty=\beta^{-1}(\set{\infty})$ acts as a repeller. Putting this heuristic into quantitative terms yields strong recurrence properties of the random walk away from $X_\infty$, which play a key role not only in the proof of Theorem~\ref{thm;recurrence}, but also for orbit closure and equidistribution results to be described in what follows. 

Ideas of this kind have a rich history in the theory of stochastic processes and dynamical systems and trace back to the work of Foster~\cite{foster} and Lyapunov~\cite{lyapunov} (see also \cite[\S15]{meyn-tweedie}). In homogeneous dynamics, they first appear in Eskin--Margulis--Mozes' work on a quantitative version of the Oppenheim conjecture~\cite{eskin-margulis-mozes}. In the study of random walks on homogeneous spaces, height functions were first systematically used by Eskin--Margulis~\cite{em} to establish recurrence properties. Functions satisfying the contraction property \eqref{contr.intro} are therefore often referred to either as \emph{Lyapunov functions} or \emph{Margulis functions}.

To obtain our results, we will need to construct two types of Lyapunov functions. 
\begin{itemize}
\item Height functions with respect to the cusps (\S\ref{subsec;height}): First, corresponding to the case $Y=\emptyset$ in Theorem~\ref{thm;recurrence}, we require a Lyapunov function $\beta_\infty$ that stays bounded on a prescribed compact subset $Z$ of $X$ and tends to infinity when leaving compact parts of the space into the cusps of $X$. Its role is to rule out escape of mass, i.e.\ ensure that the random walk does not escape to infinity. For this case, we will show that we can use the height function constructed by Benoist--Quint~\cite{bq12}. Indeed, as it turns out, the algebraic condition that is imposed in their paper on the Zariski closure of $\Gamma_\mu$ is only crucially used to ensure an expansion property in representations of $H$, so that the proof also goes through under our $H$-expansion assumption. 
\item Height functions with respect to singular subspaces (\S\ref{subsec;singular}): Secondly, corresponding to the case of a lower-dimensional homogeneous subspace $Y$ in Theorem~\ref{thm;recurrence}, we also need Lyapunov functions which blow up near the singular subspace $Y$. These are used to ensure that random walk trajectories do not accumulate near $Y$ when starting outside of it. Here, we give a construction inspired by the work of Eskin--Mirzakhani--Mohammadi~\cite{emm} for random walks on moduli space. This will allow us to avoid the use of the first return cocycles and operators appearing in \cite{bq13,bq132}, and to obtain a height function $\beta_{\mathcal{N}}$ which satisfies the contraction property \eqref{contr.intro} with respect to $A_\mu$ itself.
\end{itemize}

\begin{rem}After finishing the first version of our article, B\'{e}nard--de Saxc\'{e} improved the Markov-chain theoretic ingredient of the proofs concerning the moment assumption. Namely, using their result \cite[Theorem D]{benard-desaxce}, one can now relax the exponential moment assumption in our work (in Theorems \ref{thm;recurrence} and \ref{thm;orbit}) to a finite first moment assumption. B\'{e}nard--de Saxc\'{e}  prove this in the particular (compared to $H$-expansion) setting of Benoist--Quint, using logarithmic versions of our height functions (\cite[Theorems A,B,C]{benard-desaxce}).
\end{rem}

\subsection{Orbit closures and equidistribution}\label{sec;orbit_intro}
Measure classification and recurrence properties at hand, the next step is the question of equidistribution of random walks with respect to a homogeneous probability measure, which, once established, yields orbit closure descriptions analogous to Ratner's theorems in unipotent dynamics.

Let $\Gamma^+_\mu$ be the closed semigroup generated by the support of $\mu$.
If $\Gamma_{\mu}$ has Zariski dense image in $\Ad (H)$, then 
it is proved in~\cite{bq132}  that  the orbit closure  $\overline{\Gamma _\mu^+x}$ is a homogeneous subspace of $X$ inside which the random walk equidistributes. 
Our next result is a generalization of this and other rigidity results for the random trajectory of points proved in~\cite{bq132,prohaska-sert,sw}.  
\begin{thm}
	\label{thm;orbit}
	Let $\Lambda$ be  a lattice in a real Lie group  $G$. Let $H\leqs G$ be a connected semisimple subgroup without compact factors and with finite center. Let $\mu$ be an $H$-expanding probability measure with finite exponential moments on $H$.
	Then for every $x\in X=G/\Lambda$ there is a $\Gamma_\mu$-ergodic homogeneous 
	subspace $Y_x\subset X$ with corresponding homogeneous probability measure $\nu_x$ such that the following hold:
	\begin{enumerate}[label=\textup{(\roman*)}]
		
		\item The orbit closure $\overline{\Gamma_\mu^+x}$ equals $Y_x$.
		
		\item One has
		\begin{align*}
		    \lim_{n\to \infty}\frac{1}{n}\sum_{k=0}^{n-1}\mu^{*k}* \delta_x=\nu_x.
		\end{align*}
		\item For $\mu^{\N}$-a.e.~$(g_1, g_2 , \dots)\in H^\N $ one has 
		\begin{align*}
		\lim_{n\to \infty}\frac{1}{n}\sum_{k=0}^{n-1}\delta_{g_k\dotsm g_1x}=\nu_x.
		\end{align*}
	\end{enumerate}
\end{thm}
In statements (ii) and (iii) of the theorem above, convergence is understood with respect to the weak* topology, where weak* convergence of a sequence of probability measures $\nu_n$ on $X$ to a finite measure $\nu$ on $X$ is defined to mean that
\begin{align}\label{weak*_conv}
\lim_{n\to\infty}\int_Xf\dd\nu_n=\int_Xf\dd\nu
\end{align}
for every compactly supported continuous test function $f$ on $X$. In case the limit measure $\nu$ is a probability measure, weak* convergence $\nu_n\to\nu$ implies that \eqref{weak*_conv} holds for any bounded continuous function $f$ on $X$.

Theorem~\ref{thm;orbit} will be proved in \S\ref{subsec;equidistribution}. 
It has the non-trivial topological consequence that any infinite $\Gamma_\mu^+$-orbit in $X$ is dense in a homogeneous subspace of positive dimension. In the $G$-expanding case with an irreducible lattice $\Lambda<G$, this means that every infinite $\Gamma_\mu^+$-orbit in $X=G/\Lambda$ is dense. 

\begin{rem}
Using auxiliary constructions, our results can be applied in certain cases where the connected semisimple group $H$ is invisible. For example, they cover random walks by automorphisms on a compact nilmanifold $N/\Lambda'$ by considering $G=\Zcl(\Aut(\Lambda'))\ltimes N$ and $\Lambda=\Aut(\Lambda')\ltimes \Lambda'$, where $\Zcl(\Aut(\Lambda'))$ denotes the Zariski closure of $\Aut(\Lambda')$ inside $\Aut(N)$; see~\S\ref{sec;nilmanifolds}.
\end{rem}

\subsection{The space of homogeneous measures}
Given a closed subgroup $\Gamma$ of the Lie group $G$, we consider
\begin{align*}
\calS(\Gamma)=\set{\Gamma\text{-invariant }\Gamma\text{-ergodic homogeneous subspaces }Y\subset X},
\end{align*}
where, as before, $X=G/\Lambda$ is the quotient of $G$ by a lattice $\Lambda$. By definition, associated to each $Y\in\calS(\Gamma)$ is a $\Gamma$-invariant and ergodic probability measure $\nu_Y$ with support $Y$. This defines an embedding of $\calS(\Gamma)$ into the space of probability measures on $X$, which we use to endow $\calS(\Gamma)$ with the weak* topology. In the unipotent case, Mozes--Shah~\cite{mozes-shah} proved that convergence of homogeneous subspaces in this topology behaves in a very rigid way. Benoist--Quint~\cite[\S1.3]{bq132} later obtained a version of this result for a subgroup $\Gamma$ that is Zariski dense in a semisimple group. Following their strategy, we obtain similar results in our setup.

Given a subset $Z$ of $X$, let us write $\calS_Z(\Gamma)=\set{Y\in\calS(\Gamma)\for Y\cap Z\neq\emptyset}$ and denote by $\delta_\infty$ the Dirac measure at $\infty$ in the one-point compactification $\overline{X}=X\cup\set{\infty}$ of $X$.
\begin{prop}\label{prop.mozes-shah}
Retain the notation and assumptions of Theorem~\textup{\ref{thm;orbit}}. Then we have:
\begin{enumerate}[label=\textup{(\roman*)}]
\item  For every compact subset $Z\subset X$, $\calS_Z(\Gamma_\mu)$ is compact, and $\calS_{HZ}(\Gamma_\mu)$ is relatively compact inside $\calS(\Gamma_\mu)$. Moreover, the set $\calS(\Gamma_\mu)\cup\set{\delta_\infty}$ is compact.
\item If $Y_n \to Y_\infty$ in $\calS(\Gamma_\mu)$, then there exists a sequence $l_n \in C_G(\Gamma_\mu)$ with $l_n \to e$ and $Y_n \subset l_n Y_\infty$ for every $n\in \N$ large enough.
\end{enumerate}
\end{prop}
%Proposition~\ref{prop.mozes-shah} is an analogue of~\cite[Theorem~1.5]{bq132}. Similar remarks as in~\cite[\S1.3]{bq132} apply: When there exists a compact subset $Z\subset X$ such that $HZ=X$, the set $\calS(\Gamma_\mu)$ is compact. Otherwise, the only additional thing that can happen is escape of mass to infinity.

This proposition is a manifestation of strong rigidity of the $\Gamma_\mu$-invariant and ergodic homogeneous subspaces. For example, given a compact subset $Z \subset X$ and $Y_\infty \in \calS(\Gamma_\mu)$ with $Z^\circ \cap Y_\infty \neq \emptyset$, if for a sequence $Y_n \in \calS(\Gamma_\mu)$ we have $Y_n \cap Z \to Y_\infty \cap Z$ in the Hausdorff metric, then one can conclude that $Y_n \to Y_\infty$ in $\calS(\Gamma_\mu)$. % Proof: [For large $n$, $Y_n\in\calS_Z(\Gamma_\mu)$, so by (a) of the prop one may assume $Y_n\to Y\in\calS_Z(\Gamma_\mu)$ weak*. using (b)  and $Y_\infty\cap Z^\circ\neq \emptyset$ we get $Y\cap Z^\circ\neq\emptyset$. if $Y\cap Z^\circ\not\subset Y_\infty\cap Z^\circ$, easy contradiction. now assume $Y\cap Z^\circ\subset Y_\infty\cap Z^\circ$. now $\dim(Y)>\dim(Y_\infty)$ cannot happen, since both sets see the interior of $Z$. If $\dim(Y)<\dim(Y_\infty)$ also a contradiction by (b) of prop. Thus $Y$ and $Y_\infty$ are intersecting ergodic homogeneous subspaces of same dim, hence equal]
In particular, the weak* topology on $\calS(\Gamma_\mu)$ coincides with the restriction to $\calS(\Gamma_\mu)$ of the Fell topology on closed subsets of $X$. 

Another consequence of Proposition~\ref{prop.mozes-shah} is the following equidistribution result for sequences of homogeneous subspaces in the case that $\Gamma_\mu$ has discrete centralizer in $G$.
% we cannot reproduce \cite[Theorem~1.7]{bq132} fully for the following reason: for this one would have to argue that with discrete centralizer, the Lyapunov function is finite everywhere. BQ do this using the uniform recurrence property on their $X'$ defined in \cite[(6.1)]{bq12}. If the Lyapunov takes value $\infty$, the measure is supported on some proper parabolic subgroup. As they assume semisimplicity, the measure must then live on the Levi factor, hence the centralizer cannot be discrete. In our setting, we cannot conclude the latter, since there are $H$-expanding measures on parabolics with discrete centralizer

\begin{cor}\label{cor;hom_subspace_equidistribution}
Retain the notation and assumptions of Theorem~\textup{\ref{thm;orbit}} and assume in addition that the centralizer $C_G(\Gamma_\mu)$ of $\Gamma_\mu$ in $G$ is discrete. Let $Y_\infty\in\calS(\Gamma_\mu)$ and consider the set
\begin{align*}
\calS(\Gamma_\mu,Y_\infty)=\set{Y\in\calS(\Gamma_\mu)\for Y\subset Y_\infty}
\end{align*}
of ergodic homogeneous subspaces of $Y_\infty$. Suppose that $(Y_n)_n$ is a sequence in $\calS(\Gamma_\mu,Y_\infty)$ such that for every fixed $Y\in \calS(\Gamma_\mu,Y_\infty)\setminus\set{Y_\infty}$ one has $Y_n\not\subset Y$ for all but finitely many $n$, and such that no subsequence of $(Y_n)_n$ escapes to infinity. Then $Y_n\to Y_\infty$ in $\calS(\Gamma_\mu)$.
\end{cor}
Here, by ``escape to infinity'' we mean weak* convergence towards the Dirac measure $\delta_\infty$ at infinity.

The proofs of both statements above will be given in \S\ref{subsec;topology}.

\subsection{Birkhoff genericity}
We still assume that $\Lambda$ is a lattice in the Lie group $G$. Let $(a(t))_{t\in \R }$ be a one-parameter $\Ad$-diagonalizable subgroup of $H$ and $\nu$ a probability measure on $X=G/\Lambda$ invariant under $a(t)$ for every $t\in\R$. We say that a Radon measure $\eta$ on $H$ is \emph{$a(t)$-Birkhoff generic at $x\in X$ with respect to $\nu$} if 
\begin{align*}
\frac1T\int_0^T\delta_{a(t)hx}\dd t\to\nu
\end{align*}
in the weak* topology as $T\to\infty$ for $\eta$-almost every $h\in H$.
It was first noticed by Simmons--Weiss~\cite{sw} that, in certain situations, pathwise equidistribution of random walks as in Theorem~\ref{thm;orbit}(iii) can be 
used to deduce Birkhoff genericity of fractal measures $\eta$ on unipotent subgroups of $H$ with respect to the Haar measure on $X$, which has consequences in Diophantine approximation thanks to the Dani correspondence principle. 
Recently, more results were obtained in this direction in~\cite{prohaska-sert}. Both of these papers only deal with cases corresponding to $H=G$ in our setup. We are going to extend the existing results by removing this restriction. Even in the case where $H=G$, we obtain Birkhoff genericity for more general one-parameter subgroups and fractal measures, which will also give new results on Diophantine approximation (see~\S\ref{sec;dioph.intro}).
	
The one-parameter subgroups to which our results apply are required to satisfy certain expansion condition with respect to a unipotent subgroup of $H$.
To phrase it, we use the concept of an $a$-expanding subgroup of $H$ introduced in~\cite{s}.
 Namely, given an $\Ad$-diagonalizable element $a\in H$, a connected  $\Ad$-unipotent subgroup $U$  of $H$ normalized by $a$ is said to be \emph{$a$-expanding} if for any non-trivial irreducible representation of $H$ on a finite-dimensional real vector space $V$, the subspace 
 $V^U$ of $U$-fixed vectors is expanded by $a$, i.e.~$\lim_{n\to\infty}a^{-n} v=0$
 for any  $v\in V^U$. 
If the projection of $a$ to each simple factor of $H$ is non-trivial, then certain horospherical subgroups of $H$ are $a$-expanding. For example, this holds for the
  unstable horospherical subgroup 
  \begin{align*}
 H^+_a\df \set{ h\in H\for\lim_{n\to\infty}a^{-n} h a^n=1_H },
 \end{align*}
of $a$; see \S\ref{sec;parabolic}.

Now let $U$ be an $a(1)$-expanding subgroup contained in the unstable horospherical subgroup $H^+_{a(1)}$ of $a(1)$.
We wish to introduce a family of measures on $U$ which are generated by random walks, in a sense to be made precise in what follows. Let $A'=\set{a(t)\for t\in\R}$, $K$ be a maximal compact subgroup of $H$, and  $K'=C_K(A')\cap N_H(U)$. Here and hereafter, $C_K(A')$ denotes the centralizer of $A'$ in $K$ and $N_H(U)$ the normalizer of $U$ in $H$.  We set $P\df K'A'U\subset H$ and denote by $\lambda$ the function which associates to $g\in P$ the real parameter of its $A'$ component in its $K'A'U$ factorization; that is, $\lambda(g)=t \in \R$ for $g=k a(t) u\in K'A'U$.  
Finally, given $\omega =(g_1,g_2, \dots)\in P^\N$ and $n\in \N$, let
$k_{\omega, n} \in K'$, $a_{\omega, n}\in A'$ and $u_{\omega, n}\in U$ be such that 
\begin{align*}
g_n \dotsm g_1 = k_{\omega, n} a_{\omega, n} u_{\omega, n}. 
\end{align*}
With this notation, we are ready to define the class of measures on $U$ we shall be interested in. 
\begin{de}\label{def.generated.by.expanding}
	Let $(a(t))_{t\in \R }\leqs H$ be a one-parameter $\Ad$-diagonalizable subgroup of  $H$ and $U$ an $a(1)$-expanding subgroup of $H$ contained in $H_{a(1)}^+$. A probability measure $\eta$ on $U$ is said to be \emph{generated by $a(1)$-expanding random walks} if there is a probability measure $\mu$ on $H$ with finite exponential moments satisfying the following properties:
	\begin{enumerate}
		\item $\mu(P)=1$ and $\int_P \lambda(g) \dd\mu(g)>0$,
		\item the Zariski closure of the image of  $\Gamma_\mu$ in $\Ad(H)$ contains $\Ad(U)$, and 
		\item $\eta$ is equivalent to the pushforward of $\mu^\N$ by the map $\omega\mapsto u_\omega\df\lim_{n\to\infty}u_{\omega,n}$. 
	\end{enumerate}
\end{de}
The existence of the limit in condition~(3) above will be proved in Lemma~\ref{lem;unipotent}. Moreover, we will see as part of our discussion in \S\ref{sec;birkhoff} that conditions (1) and (2) imply that $\mu$ is $H$-expanding which will allow us to employ our main measure classification and equidistribution results discussed above.

For the statement of our result on Birkhoff genericity, recall that by Ratner's theorems the orbit closure $\overline{Hx}$ is 
homogeneous for any $x\in X$. We denote the homogeneous probability measure corresponding to $\overline{Hx}$ by $\nu_{ \overline{Hx}}$.

\begin{thm}\label{thm;birkhoff}
Let $\Lambda$ be a lattice in a real Lie group $G$ and let $H\leqs G$ be a connected semisimple subgroup without compact factors and with finite center. Let $(a(t))_{t\in \R }$ be a one-parameter $\Ad$-diagonalizable subgroup of $H$ and $U$ an $a(1)$-expanding subgroup of $H$ contained in $H_{a(1)}^+$. 
Suppose that $\eta$ is a probability measure on $U$ generated by $a(1)$-expanding random walks. Then for every $x\in X$, $\eta$ is $a(t)$-Birkhoff generic at $x$ with respect to $\nu_{ \overline{Hx}}$.
\end{thm}

Theorem~\ref{thm;birkhoff} extends the main results of~\cite{s}, which used the method of Chaika--Eskin~\cite{chaika-eskin} developed for the Teichm\"uller geodesic flow to prove Birkhoff genericity for the Haar measure on $U$. The same method was employed in~\cite{fsu} to obtain Birkhoff genericity for volume measures on curves. The proof of Theorem~\ref{thm;birkhoff} will be given in~\S\ref{sec;birkhoff}, using the connection to random walks observed in~\cite{sw}.

Probability measures generated by expanding random walks include a piece of Haar measure on $U$ and, under irreducibility conditions, self-similar measures on $\R^m$ as well as natural self-affine measures on Bedford--McMullen carpets. The latter example is crucial for our application to Diophantine approximation problems on fractals described next. In \S\ref{subsec:sponges-aff-meas} we will also discuss a more general class of fractal measures covered by Definition~\ref{def.generated.by.expanding}.

\subsection{Diophantine approximation}\label{sec;dioph.intro}
By virtue of a correspondence principle going back to the work of Dani~\cite{dani.correspond} and Kleinbock~\cite{kleinbock.duke}, Theorem~\ref{thm;birkhoff} on Birkhoff genericity has consequences for problems in Diophantine approximation, which we shall now describe.

Let $m \in  \N$ be a positive integer, $\mathbf{v}=(v_1,\dots,v_m)^t$ a (column) vector in $\R^m$, and $\mathbf{r}=(r_1, \dots,r_m) \in (0,1]^m$ such that $\sum_{i=1}^m r_i=1$.
The vector $\mathbf{v}$ is called \emph{$\mathbf{r}$-badly approximable} if there exists a constant $C>0$ such that
\begin{align}\label{eq.dioph.form}
\max_{1\le i\le m}\abs{v_iq-p_i}^{1/r_i}\cdot \abs{q} \ge C
\end{align}
for every $\mathbf{p}=(p_1, \dots, p_m) \in \Z^m$ and $q \in \Z\setminus\set{0}$. When $r_i=1/m$ for every $i=1,\dots, m$, such a vector is simply called \emph{badly approximable}. In the case $m=1$, the latter corresponds to the classical definition of a badly approximable number. It is easily seen by Dirichlet's principle that for any vector $\mathbf{v} \in \R^m$, the left-hand side of \eqref{eq.dioph.form} is $\le 1$ for infinitely many pairs $(\mathbf{p}, q) \in \Z^m \times (\Z\setminus\set{0})$.

The existence of badly approximable vectors was observed by Perron~\cite{perron1921} a century ago. It follows from Schmidt's results~\cite{schmidt.badly} that such vectors constitute a subset of $\R^m$ of everywhere-full Hausdorff dimension. This was strengthened in more recent works~\cite{kleinbock-weiss.badly,kristensen-thorn-velani} to the statement that badly approximable vectors contained in a sufficiently regular fractal $\mathcal{K}$ form a subset of full Hausdorff dimension in $\mathcal{K}$. For a general weight $\mathbf{r}$, the results of~\cite{kw.modified.schmidt,kristensen-thorn-velani,pollington-velani} imply that $\mathbf{r}$-badly approximable vectors have everywhere-full Hausdorff dimension in $\R^m$. % When $m=2$ it is known that the $\mathbf{r}$-badly approximable set is winning for Schmidt's game; see doi.org/10.1215/00127094-3165862. For $m>2$ this question is open
For $\mathbf{r}$-badly approximable vectors on a fractal set $\mathcal{K}$, the full-dimension statement is known to hold when $\mathcal{K}$ has a certain product structure (see \cite[Theorem~8.4]{kleinbock-weiss.badly}, \cite[Theorems~11,13]{kristensen-thorn-velani}).

The results outlined above can be summarized by saying that ($\mathbf{r}$-)badly approximable vectors are abundant from the viewpoint of Hausdorff dimension. On the Lebesgue measure side, however, Khintchine's theorem~\cite{khintchine} implies that badly approximable vectors have zero Lebesgue measure. Using a generalization of Khintchine's theorem~\cite{schmidt.general.khintchine}, the same is seen to be true for $\mathbf{r}$-badly approximable vectors. The question whether badly approximable vectors on a given fractal $\mathcal{K}$ also form a null set with respect to a natural measure on the fractal proved to be rather more delicate. The first results in this direction are due to Einsiedler--Fishman--Shapira~\cite{einsiedler-fishman-shapira}, who proved that badly approximable vectors have zero Hausdorff measure on certain fractals invariant under toral endomorphisms (in case the dimension is $m=1$) or toral automorphisms (in case $m=2$). For example, their results apply to the middle-third Cantor set. This was vastly generalized by Simmons--Weiss~\cite{sw}, who established the same statement for general self-similar fractals satisfying a separation condition. 
To the best of our knowledge, for general weights $\mathbf{r}$ or on fractals which are not strictly self-similar, the question of the measure of badly approximable vectors is open. Our methods allow us to make an initial contribution in this direction. For simplicity, here in the introduction we will describe only the special case of Bedford--McMullen carpets; see \S\ref{sec;diophantine} for the discussion in full generality.

\emph{Bedford--McMullen carpets} are two-dimensional self-affine fractals, introduced independently by Bedford~\cite{bedford} and McMullen~\cite{mcmullen}, which admit a particularly simple construction. Let $a,b\ge 2$ be distinct integers and divide the unit square $[0,1]^2$ into an $a\times b$-grid parallel to the coordinate axes. Choose an arbitrary subcollection $S$ of the $ab$ rectangles created and discard the remaining ones. Iteratively proceed in the same way for each of the rectangles that remain, using the same pattern $S$. The points remaining after infinite iteration form a Bedford--McMullen carpet $\mathcal{K}$. If $(c_i,d_i)_{i=1}^k$ denote the coordinates of the bottom-left corners of the rectangles kept in the first construction step and we define the affine maps $\phi_i\colon\R^2\to\R^2$ by
\begin{align*}
    \phi_i(x,y)=\begin{pmatrix}\frac1a&\\&\frac1b\end{pmatrix}\begin{pmatrix}x\\y\end{pmatrix}+\begin{pmatrix}c_i\\d_i\end{pmatrix},
\end{align*}
then $\mathcal{K}$ is the unique non-empty compact subset of $\R^2$ satisfying $\bigcup_{i=1}^k\phi_i(\mathcal{K})=\mathcal{K}$. The Hausdorff dimension of fractals of this type was explicity calculated by Bedford and McMullen. Except for special cases, it turns out that their Hausdorff measure in the correct dimension is infinite~\cite{peres.inf-H-meas}. However, there exists another natural measure $\nu_{\mathcal{K}}$ on $\mathcal{K}$, known as the \emph{McMullen measure}: It is the unique $T$-invariant ergodic probability measure on $\mathcal{K}$ of full Hausdorff dimension, where $T$ is the toral endomorphism corresponding to $(\begin{smallmatrix}a&\\&b\end{smallmatrix})$~\cite{kenyon-peres,mcmullen}. For further background on the fractal geometry of Bedford--McMullen carpets, we refer to the survey article~\cite{fraser.survey}.

The following is a specialization of our Theorem~\ref{thm;dioph.insection} to the case of weighted badly approximable vectors on Bedford--McMullen carpets (see Corollary~\ref{cor.sierpinski}).
\begin{thm}\label{thm.dioph.intro}
Let $a,b$ be positive integers satisfying $\min\set{a^2,b^2} > \max\set{a,b}$ and let $\mathcal{K}\subset\R^2$ be a Bedford--McMullen carpet invariant under the toral endomorphism $T=(\begin{smallmatrix}a&\\&b\end{smallmatrix})$. Suppose that $\mathcal{K}$ is not contained in any straight line. Then for the choice of weights
\begin{align*}
\mathbf{r}=\biggl(\frac{2\log a-\log b}{\log a+\log b},\frac{2\log b-\log a}{\log a+\log b}\biggr),
\end{align*}
the set of $\mathbf{r}$-badly approximable vectors on $\mathcal{K}$ has measure zero with respect to the McMullen measure $\nu_{\mathcal{K}}$ on $\mathcal{K}$.
\end{thm}
The requirement above that $\mathcal{K}$ is not contained in any straight line plays the role of an irreducibility condition. It is satisfied when, in the construction of the Bedford--McMullen carpet described above, the kept rectangles in the pattern $S$ do not all belong to a single line or column in the $a\times b$-grid.

As mentioned before its statement, the above theorem will follow from a much more general result about Diophantine properties of ``$(\mathbf{r},\mathbf{s})$-matrix sponges'' (Theorem~\ref{thm;dioph.insection})---a class of fractals that we will introduce in \S\ref{subsub.sponges}. In fact, the latter result will imply a version of Theorem~\ref{thm.dioph.intro} for higher-dimensional analogues of Bedford--McMullen carpets, which are called ``self-affine Sierpi\'{n}ski sponges'' in~\cite{kenyon-peres}; see Corollary~\ref{cor.sierpinski}.

\subsection*{Acknowledgments}
The authors are thankful to Barak Weiss for encouraging discussions on an initial version of this article as well as for helpful bibliographical suggestions, and to Manfred Einsiedler for numerous useful remarks. 
R.\ S.\ is supported by National Key Research and Development Program of China 2021YFA1003204, NSFC 12161141014 and NSF Shanghai 22ZR1406200. 
C.\ S.\ is supported by SNF grant 182089 and SNF Ambizione 193481.

\section{\texorpdfstring{$H$}{H}-expansion: Definition and basic properties}\label{sec;def}
We start by properly stating the definition of uniform expansion and giving alternative formulations thereof. 
\begin{de}\label{de;unif.exp}
Let $\mu$ be a probability measure on $\GL_d(\R)$. A vector $v\in\R^d$ is said to be \emph{$\mu$-expanded} if
\begin{align}\label{eq.v.exp}
\liminf_{n\to\infty}\frac1n\log\norm{g_n\dotsm g_1v}>0
\end{align}
for $\mu^\N$-almost every sequence $(g_i)_i$ of elements of $\GL_d(\R)$. The measure $\mu$ is said to be \emph{uniformly expanding} if every nonzero $v\in\R^d$ is $\mu$-expanded. If \eqref{eq.v.exp} holds with $\ge$ in place of $>$ for every nonzero $v\in\R^d$, we call $\mu$ \emph{non-contracting}.
\end{de}

The above definition is the most general, but it can be hard to verify in practice. The characterization in the following proposition is often simpler to check. Moreover, it will also play an important role in the height function constructions in \S\ref{sec;height_functions}.
Recall that a probability measure $\mu$ on $\GL_d(\R)$ is said to have a finite first moment if $\int \log \operatorname{N}(g)\dd\mu(g)<\infty$, where $\operatorname{N}(g)=\max\set{\norm{g},\norm{g^{-1}}}$.

\begin{prop}[{\cite[Lemma~1.5]{el},~\cite[Proposition~2.4]{prohaska-sert}}]\label{prop;expansion_char}
Let $\mu$ be a probability measure on $\GL_d(\R)$
with finite first moment. 
Then $\mu$ is uniformly expanding if and only if there exists $N\in\N$ and a constant $C>0$ such that for every nonzero $v \in \R^d$
\begin{align*}
\int_{\GL_d(\R)}\log\frac{\norm{gv}}{\norm{v}}\dd\mu^{*N}(g)\ge C.
\end{align*}
%It is also equivalent to require that the latter holds, for possibly another constant $C>0$, for every $N \in \N$ larger than a constant $N_0 \in \N$.
\end{prop}
Uniform expansion can also be conveniently understood in light of the following theorem of Furstenberg--Kifer and Hennion. Recall that given a probability measure $\mu$ on a Lie group $G$, we denote by $\Gamma_\mu$ the closed subgroup generated by the support of $\mu$.

\begin{thm}[Furstenberg--Kifer~\cite{furstenberg-kifer}, Hennion~\cite{Hen}]\label{thm:furstenberg_kifer}
	Let $\mu$ be a probability measure on $\GL_d(\R)$ with finite first moment. Then there exists a partial flag $\R^d= F_1 \supset F_2 \supset \dots \supset F_k\supset F_{k+1}=\set{0}$ of $\Gamma_\mu$-invariant subspaces and a collection of real numbers $\beta_1(\mu)>\dots>\beta_k(\mu)$ such that for every $v \in F_i\setminus F_{i+1}$, we have $\mu^\N$-a.s.\ 
	\begin{align*}
		\lim_{n\to\infty} \frac{1}{n}\log \norm{g_n \dotsm g_1 v}=\beta_i(\mu).
	\end{align*}
    Moreover, the $\beta_i(\mu)$ are the values of
    \begin{align*}
        \alpha(\nu)\df \int_{\mathbb{P}(\R^d)}\int_{\GL_d(\R)}\log\frac{\norm{gv}}{\norm{v}}\dd\mu(g)\dd\nu(\R v)
    \end{align*}
    that occur when $\nu$ ranges over $\mu$-ergodic $\mu$-stationary probability measures on the projective space $\mathbb{P}(\R^d)$.
\end{thm}

In this result, the set of exponents $\set{\beta_1(\mu),\dots,\beta_k(\mu)}$ is contained in the set of Lyapunov exponents of $\mu$ and $\beta_1(\mu)$ coincides with the top Lyapunov exponent. 

Uniform expansion can now be rephrased as follows.
\begin{lem}
A probability measure $\mu$ on $\GL_d(\R)$ with finite first moment is uniformly expanding if and only if $\beta_k(\mu)>0$, where $\beta_k(\mu)$ is the smallest exponent appearing in Theorem~\textup{\ref{thm:furstenberg_kifer}}.\qed
\end{lem}

Furstenberg--Kifer's theorem can also be used to see that, in fact, almost sure divergence is enough to get uniform expansion. It will be useful to denote by $F^{\leqs 0}$ the largest subspace among $F_1,\dots,F_{k+1}$ with non-positive exponent.
\begin{prop}
Let $\mu$ be a probability measure on $\GL_d(\R)$ with finite first moment. Then $\mu$ is uniformly expanding if and only if for every nonzero $v\in\R^d$ we have
\begin{align}\label{eq.as.div}
\lim_{n\to\infty}\norm{g_n\dotsm g_1v}=\infty
\end{align}
for $\mu^\N$-a.e.\ sequence $(g_i)_i$ of elements of $\GL_d(\R)$.
\end{prop}
\begin{proof}
We only need to show that \eqref{eq.as.div} implies uniform expansion. 
We apply Theorem~\ref{thm:furstenberg_kifer} and consider the space $F^{\leqs 0}$ defined before the statement of the proposition. This space is $\Gamma_\mu$-invariant. If it is nonzero, its projectivization thus supports an ergodic $\mu$-stationary probability measure $\nu$. Using the assumed almost sure divergence and Atkinson/Kesten's lemma (see e.g.~\cite[Lemma~II.2.2]{bl}), it follows that $\alpha(\nu)>0$, where $\alpha(\nu)$ is as defined in Theorem~\ref{thm:furstenberg_kifer}, a contradiction.
\end{proof}

For later use, let us also record at this point an immediate restriction that the presence of expansion puts on $\mu$-stationary measures on finite-dimensional vector spaces. 
\begin{lem}
    \label{lem;stationary}
   	Let $\mu$ be a probability measure on $\GL_d(\R)$ and $E$ a measurable subset of $\R^d$ such that every $v\in E$ is $\mu$-expanded. Then every $\mu$-stationary probability measure $\nu$ on $\R^d$ satisfies $\nu(E)=0$.
\end{lem}
In particular, if $\mu$ has a finite first moment, then any $\mu$-stationary probability measure $\nu$ on $\R^d$ is supported on the Furstenberg--Kifer subspace $F^{\leqs 0}$ of subexponential expansion. With a similar argument for vectors that are contracted instead of expanded, one can more generally show that $\nu((F^{\leqs 0}\setminus F^{<0})\cup\set{0})=1$, where $F^{<0}$ is defined in a way analogous to $F^{\leqs 0}$.
\begin{proof}
Write $G=\GL_d(\R)$ and $V=\R^d$. 
By~\cite[Proposition~2.14]{bq}, 
 the forward dynamical system $(G^\N \times V, \mu^\N \times \nu, T^V)$ is measure preserving, where 
\begin{align*}
T^V((g_1, g_2, \dots), v)=((g_2, g_3, \dots), g_1 v).
\end{align*}
Let $K$ be a compact subset of  $V$. Then by Poincar\'e recurrence applied to $G^\N \times K$, we know that $\nu (K\cap E)=0$, and the conclusion follows. 
\end{proof}

Now we come to the central concept of this article: $H$-expansion.
\begin{de}\label{def;H.expand}
Let $H$ be a connected semisimple Lie group  with finite center and $\mu$ a probability measure on $H$. Given a representation $(\rho,V)$ of $H$ we say that $\mu$ is \emph{uniformly expanding in $(\rho,V)$} if $\rho_*\mu$ is uniformly expanding. We say that $\mu$ is \emph{$H$-expanding} if $\mu$ is uniformly expanding in every representation of $H$ without nonzero $H$-fixed vectors, or equivalently, in every non-trivial irreducible representation of $H$. 
\end{de}
Here and everywhere else, by a ``representation'' we always mean a continuous homomorphism into 
the group of invertible linear transformations of a
finite-dimensional real vector space. 
It is well known that such representations are automatically smooth.
For notational simplicity, we are going to simply write $h\acts v$ for $\rho(h)v$ for $h\in H$ and $v\in V$ when the representation $(\rho,V)$ is clear from context.
In this case, we also just say that $\mu$ is uniformly expanding on $V$ to mean that $\mu$ is uniformly expanding in $(\rho,V)$. 

We next recall what the moment conditions mean for a probability measure on a semisimple group that is not necessarily linear.
\begin{de}\label{de;moments}
Let $H$ be a connected semisimple Lie group with finite center. Let $\mu$ be a probability measure on $H$. Then $\mu$ is said to have a \emph{finite first moment} (resp.\ \emph{finite exponential moments}) if $\rho_*\mu$ has a finite first moment (resp.\ finite exponential moments) for some representation $\rho$ of $H$ with finite kernel.
\end{de}
Of course, these moment conditions are automatically satisfied when $\mu$ has compact support.
\begin{lem}[{\cite[Lemmas~10.6,~10.7]{bq}}]\label{lem;first_moment}
Let $H$ and $\mu$ be as in Definition~\textup{\ref{de;moments}} and suppose that $\mu$ has a finite first moment \textup{(}resp.\ finite exponential moments\textup{)}. Then $\rho_*\mu$ has a finite first moment \textup{(}resp.\ finite exponential moments\textup{)} for any representation $\rho$ of $H$.
\end{lem}
We remark that even though in~\cite{bq}, the above lemma is proved for algebraic groups, the given proof also works for a connected semisimple group $H$ with finite center. Indeed, the argument relies only on a reformulation of the moment condition into an integrability condition on the Cartan projection $\kappa\colon H\to\mathfrak{a}^+$, which is related to representations of $H$ by virtue of the formula $\norm{\rho(h)}=e^{\chi(\kappa(h))}$ for $h\in H$, where $(\rho,V)$ is an irreducible representation  of $H$ with highest weight $\chi$ and $\norm{\cdot}$ is the operator norm associated to a Euclidean norm on $V$ invariant under the maximal compact subgroup $K$ of $H$ used to define $\kappa$.

In the proposition below we collect some first facts about $H$-expansion.
\begin{prop}\label{prop;facts}
Let $H$ be a connected semisimple Lie group with finite center and $\mu$ a probability measure on $H$. Then:
\begin{enumerate}[label=\textup{(\roman*)}]
\item Given a representation $(\rho,V)$ of $H$, the following are equivalent:
\begin{itemize}
    \item Any vector $v\in V$ that is not $\rho_*\mu$-expanded is $H$-fixed.
    \item The measure $\mu$ is uniformly expanding on the quotient $V/V^H$.
\end{itemize}
\item If $\mu$ is $H$-expanding, then $H$ has no compact factors.
\item If $\mu$ is $H$-expanding and $\psi\colon H\to G'$ is a non-trivial continuous homomorphism into a real Lie group $G'$, then $H'=\psi(H)$ is a connected, closed, semisimple subgroup of $G'$ with finite center and $\psi_*\mu$ is $H'$-expanding.

\item  Suppose $H$ is an almost direct product of
    connected normal subgroups $H_1$ and $H_2$ and let $\mu_i$ be probability measures on $H_i$ with finite first moments, $i=1,2$. 
    If $\mu_i$ is $H_i$-expanding for $i=1,2$ and $\mu$ is the pushforward of $\mu_1 \times \mu_2$ by multiplication, then $\mu$ is $H$-expanding.
\end{enumerate}
\end{prop}

\begin{proof}
For (i), note that by semisimplicity of $H$, the quotient $V/V^H$ identifies with an $H$-invariant complement $V^+$ of $V^H$ in $V$. Thus we only need to prove that uniform expansion of $\mu$ on $V^+$ implies the statement in the first bullet point. Let $p_+\colon V\to V^+$ be the projection and take $v \in V$ which is not $\rho_*\mu$-expanded. Then also $p_+(v)$ is not $\rho_*\mu$-expanded, so that uniform expansion on $V^+$ implies $p_+(v)=0$. Hence, $v$ is $H$-fixed.

For (ii), suppose $H$ has a compact factor $K$. Then $\mu$ cannot be uniformly expanding in the representation of $H$ obtained by composing the projection on $K$ with the adjoint representation of $K$. Thus, $\mu$ is not $H$-expanding.

As $H$ is semisimple and has finite center, $H'$ is a connected and semisimple immersed Lie subgroup of $G'$ with finite center in the setting of (iii). As representations of $H'$ induce representations of $H$ by precomposition with $\psi$, the $H'$-expansion condition is immediate. It only remains to argue that $H'$ is closed in $G'$. As this is in fact a more general statement, we drop the accents and simply show that a semisimple immersed Lie subgroup $H$ of a Lie group $G$ must be closed when $H$ has finite center. For this, it suffices to show that if a sequence $(h_n)_n$ in $H$ converges to the identity $e$ in the topology of $G$, then this convergence holds also in the topology of $H$. Notice that $\Ad_G(h_n)$ considered as elements of $\Aut(\mathfrak h)$ converges to the identity map when $\Aut(\mathfrak h)$ is endowed with the subspace topology inherited from $\Aut(\mathfrak g)$. However, as linear semisimple Lie algebras are algebraic (see~\cite[Theorem~VIII.3.2]{hochschild}), 
this subspace topology coincides with the usual topology of $\Aut(\mathfrak h)$. Since near the identity, $\Ad_H$ is a local isomorphism from $H$ to $\Aut(\mathfrak h)$, we thus find a sequence $(h_n')_n$ converging to $e$ in $H$ such that $\Ad_H(h_n)=\Ad_H(h_n')$ for all $n$. This implies that $h_n^{-1}h_n'$ is contained in the center of $H$ and converges to $e$. As the center is finite, we have $h_n=h_n'$ for all $n$ large enough. We conclude that, indeed, $h_n\to e$ as $n\to\infty$ holds also in the topology of $H$.

Finally, to prove (iv), let $(\rho,V)$ be a non-trivial irreducible representation of $H$. 
Since $H_1$ and $H_2$ commute, for every $n \in \N$, $\mu^{*n}$ is the pushforward by multiplication of $\mu_1^{*n} \times \mu_2^{*n}$, and the subspaces $V^{H_i}$ of $H_i$-fixed vectors in $V$ are $H$-invariant. By irreducibility, they are trivial or all of $V$. It follows that one of $V^{H_1},V^{H_2}$ is zero. We assume without loss of generality that $V^{H_1}=\set{0}$.

Note that both $\rho_*\mu_1$ and $\rho_*\mu_2$ have a finite first moment by Lemma~\ref{lem;first_moment}. This readily implies that $\rho_*\mu$ has a finite first moment. By Proposition~\ref{prop;expansion_char}, it suffices to show that for $N$ large enough and $v\neq 0$, the quantity
\begin{align*}
&\int_{H_1\times H_2}\log\frac{\norm{h_1h_2\acts v}}{\norm{v}}\dd\mu_1^{*N}(h_1)\dd\mu_2^{*N}(h_2)\\&=\int_{H_2}\int_{H_1}\log\frac{\norm{h_1h_2\acts v}}{\norm{h_2\acts v}}\dd\mu_1^{*N}(h_1)\dd\mu_2^{*N}(h_2)+\int_{H_2}\log\frac{\norm{h_2\acts v}}{\norm{v}}\dd\mu_2^{*N}(h_2)
\end{align*}
is uniformly bounded from below by some $C>0$. As $\rho_*\mu_1$ is uniformly expanding, Proposition~\ref{prop;expansion_char} gives this lower bound for the first integral above for $N$ large enough. By the same argument, the second term is either equal to $0$ or also bounded below by some $C>0$, according to whether $V^{H_2}$ is $V$ or $\set{0}$, respectively.
\end{proof}

\begin{rem}\label{rem;alg}
We point out that in part (iii) of the previous proposition, if the target $G'$ of the homomorphism $\psi$ is a real algebraic group, then the conclusion can be strengthened to the statement that the semisimple group $H'=\psi(H)$ is \emph{almost algebraic}, meaning that it has finite index in a  real algebraic subgroup of $G'$. Indeed, as already exploited in the proof above, the point is that linear semisimple Lie algebras are algebraic. In particular, this applies when $\psi$ is a representation $(\rho,V)$ of $H$. This fact is useful to keep in mind.
\end{rem}

Combining Proposition~\ref{prop;facts}(i) with Lemma~\ref{lem;stationary}, we immediately obtain the following corollary about $\mu$-stationary measures on vector spaces.
\begin{cor}\label{cor;stationary}
Let $(\rho,V)$ be a representation of $H$ and suppose that $\mu$ is uniformly expanding on $V/V^H$. Then any $\mu$-stationary probability measure on $V$ is supported on the subspace $V^H$ of $H$-fixed vectors.\qed
\end{cor}

\section{Examples of \texorpdfstring{$H$}{H}-expanding measures}\label{sec;ex}
In this section, we exhibit classes of probability measures on semisimple Lie groups that satisfy the $H$-expansion property.

\subsection{Zariski dense measures}\label{subsec;Z-dense}
As already mentioned in \S\ref{sec;introduction}, the first class of examples of $H$-expanding measures consists of those whose support generates a Zariski dense subgroup of $H$. This is the class of measures considered by Benoist--Quint~\cite{bq11,bq13,bq132}.
\begin{prop}\label{prop;Z-dense}
Let $H$ be a connected semisimple Lie group without compact factors and with finite center. Let $\mu$
 be a  probability measure on $H$ with finite  first moment. 
 Suppose that $\Ad(\Gamma_\mu)$ is Zariski dense in $\Ad(H)$. 
 Then $\mu$ is $H$-expanding.
\end{prop}

For the proof we need the following lemma, which is used to extend the Zariski density assumption to arbitrary representations.
\begin{lem}\label{lem;useful2}
Let $\Gamma$ be a subsemigroup of $H$ and $S$ a connected subgroup of $H$. Suppose that the Zariski closure of $\Ad(\Gamma)$ contains $\Ad(S)$. Then for every representation $(\rho,V)$ of $H$, $\rho(S)$ is contained in $\Zcl(\rho(\Gamma))$.
\end{lem}
\begin{proof}
We consider the product representation $\rho'=\Ad\times\rho$. Let $\mathcal H'$ be the Zariski closure of $\rho'(H)$ inside $\GL(\mathfrak h)\times \GL(V)$. Then both $\Ad$ and $\rho$ factor through $\mathcal H'$. As noted in Remark~\ref{rem;alg}, $\rho'(H)$ has finite index in $\mathcal H'$. The same holds for the Zariski closure $\mathcal H$ of $\Ad(H)$, so that both $\mathcal H$ and $\mathcal H'$ are Zariski connected real algebraic groups of dimension $\dim(H)$. Thus, projection to the first factor of $\GL(\mathfrak h)\times \GL(V)$ gives an isogeny $p\colon\mathcal H'\to\mathcal H$, and we know that $\Zcl(\rho'(\Gamma))$ has finite index in $p^{-1}(\Zcl(\Ad(\Gamma)))$. Since $\rho'(S)$ is connected and $\Ad(S)$ is contained in $\Zcl(\Ad(\Gamma))$ by assumption, it follows that $\rho'(S)$ is contained in $\Zcl(\rho'(\Gamma))$. By projecting to the second factor, we conclude that $\rho(S)$ is contained in $\Zcl(\rho(\Gamma))$.
\end{proof}

\begin{proof}[Proof of Proposition~\textup{\ref{prop;Z-dense}}]
Let $(\rho,V)$ be a representation of $H$ without nonzero $H$-fixed vectors. By Lemma~\ref{lem;useful2}, $\rho(\Gamma_\mu)$ is Zariski dense in $\rho(H)$.
Now uniform expansion in $(\rho,V)$ follows directly from Furstenberg's theorem on positivity of the top Lyapunov exponent (see~\cite[Theorem~8.6]{furstenberg.positivity}). 
To see that the assumptions of Furstenberg's theorem are satisfied, note that by Lemma~\ref{lem;first_moment} we know that $\rho_*\mu$ has a finite first moment, and using Zariski density of $\rho(\Gamma_\mu)$ together with complete reducibility one may assume that $\rho(\Gamma_\mu)$ acts irreducibly, which implies strong irreducibility in view of Zariski connectedness of $\rho(H)$. Finally, since the ground field is $\R$, the fact that the Zariski closure of $\rho(\Gamma_\mu)$ is non-compact implies that $\rho(\Gamma_\mu)$ is not relatively compact, finishing the proof.
\end{proof}

\subsection{Measures on parabolic groups}\label{sec;parabolic}
Our next goal is to exhibit probability measures supported on proper parabolic subgroups of $H$ which are $H$-expanding. Combining general criteria with the notion of the expanding cone, which was introduced by the third-named author in~\cite{s15} (see also the slightly earlier work \cite{M-SG}) and which traces back to the works of Shah and Weiss~\cite{shah96,shah-weiss,weiss.epimorphic}, we will obtain another easy-to-verify sufficient condition for $H$-expansion.

We start by explaining our general setup. Let $H$ be a connected semisimple real Lie group without compact factors and with finite center and let $a$ be an $\Ad$-diagonalizable element of $H$. Then given a representation $(\rho, V)$ of $H$, we have a direct sum decomposition   
\begin{align*}
V=V^+_a\oplus V^0_a\oplus V^-_a,
\end{align*}
where $V^+_a, V^0_a, V^-_a$ are the sums of the eigenspaces of $\rho(a)$ with eigenvalues $>$, $=$ or $<1$, respectively. Let $U$ be a connected $\Ad$-unipotent subgroup of $H$ normalized by $a$. Following~\cite{s}, we say that $U$ is \emph{$a$-expanding} if for every non-trivial irreducible representation $(\rho, V)$ of $H$, the subspace $V^U$ of $U$-fixed vectors is contained in $V_a^+$. It is equivalent (\cite[Lemma~A.1]{s}) to require that in any irreducible representation of $(\rho,V)$ of $H$ and for any nonzero $v\in V$, the $\rho(U)$-orbit of $v$ is not contained in $V^0_a \oplus V_a^-$. For example, if $a$ has a non-trivial projection to every simple factor of $H$, then the unstable horospherical subgroup $H^+_a=\set{ h\in H\for\lim_{n\to\infty}a^{-n} h a^n=1_H }$ is $a$-expanding (\cite[Lemma~5.2]{shah96}). In fact, it can be shown that $U$ is $a$-expanding if and only if $U\cap H_a^+$ is (\cite[Lemma~A.2]{s}).

Now let $Q\leqs H$ be a parabolic subgroup with maximal connected $\R$-split torus $A$. Using the above, we will give two criteria for a measure on $Q$ to be $H$-expanding. To state the first, write $Q=MA_cN$ for the Langlands decomposition of $Q$. In particular, this means that $N$ is the unipotent radical of $Q$, $MA_c=C_H(A_c)$ is a (reductive) Levi subgroup of $Q$, and $A_c$ is a maximal central connected $\R$-split torus in $MA_c$ (see e.g.~\cite[\S VII.7]{knapp} for details on Langlands decomposition). We may assume that $A_c\leqs A$.
Given a probability measure $\mu$ on $Q$, by using the diffeomorphism $Q\cong M\times A_c\times N$ given by multiplication and projecting to some of the factors, we obtain associated probability measures $\mu_M$, $\mu_{A_c}$, $\mu_{MA_c}$ etc.
Finally, we denote by $\lambda_c\colon Q\to\mathfrak a$ the composition of the projection to $A_c$ with the logarithm map $\log\colon A\to\mathfrak a$, where $\mathfrak a$ is the Lie algebra of $A$.
\begin{prop}[$H$-expanding measures (1)]\label{prop.examples}
Let $\mu$ be a probability measure on $H$ with finite first moment such that $\mu(Q)=1$ for some parabolic subgroup $Q=MA_cN$ of $H$. Denote by $a_{c,\operatorname{avg}}(\mu)=\exp\bigl(\int\lambda_c(g)\dd\mu(g)\bigr)\in A_c$ the $A_c$-average of $\mu$. Let $U$ be a connected Lie subgroup of $N$ and suppose the following:

\begin{enumerate}[label=\textup{(\arabic*)}]
    \item $\supp(\mu) \subset  M A_c U\cap N_H(U)$ and the Zariski closure of $\Ad(\Gamma_\mu)$ contains $\Ad(U)$, 
    \item $U$ is $a_{c,\operatorname{avg}}(\mu)$-expanding, and 
    \item $\mu_M$ is non-contracting in every representation of $H$.
\end{enumerate}
Then $\mu$ is $H$-expanding.
\end{prop}

Before proceeding with the preparations for the proof of the above proposition, let us provide a few brief comments on its hypotheses.
\begin{rem}[On the hypotheses of Proposition~\ref{prop.examples}]\leavevmode
\begin{itemize}
\item In fact, there is no freedom in the choice of $U$: Condition~(1) implies that it needs to be the Zariski closure of the projection of $\Gamma_\mu$ to $N$. %i.e.\ the smallest closed connected subgroup of $N$ containing it.
\item When $U=N$ and the parabolic group $Q$ is absolutely proper, condition~(2) can conveniently be checked using the notion of expanding cone to be discussed in \S\ref{subsec;exp_cone}.
\item The non-contraction requirement on $\mu_M$ in condition~(3) is satisfied, for instance, when the identity component of the Zariski closure of $\Ad(\Gamma_{\mu_M})$ is reductive with compact center (for example, the identity component of $\Ad(M)$ itself). Indeed, in this case similar arguments as in the proof of Lemma~\ref{lem;useful2} can be used to show that $\Gamma_{\mu_M}$ acts completely reducibly and by transformations of determinant $\pm 1$ in every representation $(\rho,V)$ of $H$. Then the Lyapunov exponents of $\mu_M$ in any $\Gamma_{\mu_M}$-irreducible subspace of $V$ sum to $0$ and one concludes using Theorem~\ref{thm:furstenberg_kifer}.
\item Another useful fact for the verification of condition~(3) is that the connected component $M^\circ$ of $M$ is the almost direct product of its semisimple part $S=[M^\circ,M^\circ]$ and a compact center. Provided $\mu_M$ is supported on $M^\circ$, one can thus project to the non-compact part $S^{nc}$ and is only left checking non-contraction for $\mu_{S^{nc}}$. The latter could follow from Zariski density (Proposition~\ref{prop;Z-dense}), or by a recursive application of Proposition~\ref{prop.examples} above to $H=S^{nc}$. In the general case, one can obtain from $\mu_M$ a probability measure $\mu_M^\circ$ on $M^\circ$ defined as the law of the first return to $M^\circ$ of the random walk on $M$ induced by $\mu_M$; see~\cite[\S5.2]{bq}. Using \cite[Proposition~5.9]{bq} and Theorem~\ref{thm:furstenberg_kifer}, one sees that the non-contraction property of $\mu_M^\circ$ implies that of $\mu_M$. 
\end{itemize} 
\end{rem}

For the proof of Proposition~\ref{prop.examples} we require the following lemma, which reduces checking expansion to vectors fixed by some unipotent subgroup of the image of the algebraic group generated by $\supp(\mu)$.
\begin{lem}[A criterion for expansion]\label{lem;useful}
Let $V$ be a finite-dimensional real vector space and $\mu'$ a probability measure on $\GL(V)$ with finite first moment. Denote by $Q'$ the Zariski closure of $\Gamma_{\mu'}$ and let $U'$ be a unipotent subgroup of $Q'$. Suppose that every nonzero vector $v\in V^{U'}$ is $\mu'$-expanded, where $V^{U'}$ denotes the subspace of $U'$-fixed vectors. Then $\mu'$ is uniformly expanding.
\end{lem}
\begin{proof}
Let us suppose for a contradiction that $\mu'$ is not uniformly expanding. Then there exists a vector $v \in V\setminus\set{0}$ with $\liminf_{n \to \infty} \frac{1}{n}\log \norm{g_n \dotsm g_1 v} \le 0$ for a positive measure subset of $(g_i)_i\in (Q')^\N$ with respect to $(\mu')^\N$.  By Theorem~\ref{thm:furstenberg_kifer}, there exists a non-trivial $\Gamma_{\mu'}$-invariant subspace $W\leqs V$ such that for every $w \in W$, we have $\lim_{n \to \infty} \frac{1}{n}\log \norm{g_n \dotsm g_1 w} \le 0$ for $(\mu')^\N$-a.e.\ $(g_i)_i\in (Q')^\N$. Since $Q'$ is the Zariski closure of $\Gamma_{\mu'}$, the subspace $W$ is stabilized by $Q'$ and hence, by $U'$. By the Lie--Kolchin theorem, we have $W^{U'} \neq \set{0}$. This implies that for any nonzero $w \in W^{U'}\leqs V^{U'}$, we have $\lim_{n \to \infty} \frac{1}{n}\log \norm{g_n \dotsm g_1 w} \le 0$ for $(\mu')^\N$-a.e.\ $(g_i)_i\in (Q')^\N$, contradicting expansion on $V^{U'}$.
\end{proof}

\begin{proof}[Proof of Proposition~\textup{\ref{prop.examples}}]
Let $(\rho,V)$ be a non-trivial irreducible representation of $H$. By Lemma~\ref{lem;first_moment}, the measure $\rho_*\mu$ has a finite first moment, and Lemma~\ref{lem;useful2} implies that $\rho(U)$ is a unipotent subgroup of the Zariski closure of $\rho(\Gamma_\mu)$. In view of Lemma~\ref{lem;useful}, to prove uniform expansion of $\rho_*\mu$ it  suffices to show that for every nonzero $v \in V^U$, we have 
\begin{align*}
\liminf_{n \to \infty} \frac{1}{n}\log \norm{g_n\dotsm g_1\acts v}>0
\end{align*}
for $\mu^\N$-a.e.\ $(g_i)_i \in H^\N$. Since condition~(1) implies that $\Gamma_\mu\subset MA_cU$ and $v$ is $U$-fixed, it suffices to prove the above for $\mu_{MA_c}$-a.e.\ $(g_i)_i\in H^\N$, where $\mu_{MA_c}$ is the $MA_c$-projection of $\mu$. Writing $g_i=m_ia_i$ for the $MA_c$-factorization of $g_i\in MA_c$ and using that $M$ and $A_c$ commute, we see that
\begin{align}\label{MA_decomp}
\frac{1}{n}\log \norm{g_n\dotsm g_1\acts v}=\frac1n\log\frac{\norm{a_n\dotsm a_1m_n\dotsm m_1\acts v}}{\norm{m_n\dotsm m_1\acts v}}+\frac1n\log \norm{m_n\dotsm m_1\acts v}.
\end{align}

The second term above is almost surely non-negative in the limit, by the assumed non-contraction property of $\mu_M$.

To deal with the first term, let $\Phi(A_c,\rho)$ be the set of weights of $A_c$ for the representation $(\rho,V)$. Let $\set{\chi_1,\dots,\chi_t}$ be the subcollection of those $\chi\in\Phi(A_c,\rho)$ with $\chi(a_{c,\operatorname{avg}}(\mu))>1$ and denote the corresponding weight spaces by $V_1,\dots,V_t$. Then by the assumption on $U$, we have $V^U\subset \bigoplus_{j=1}^tV_j\eqqcolon W$. Since $A_c$ and $M$ commute, $W$ is $M$-invariant. Lemma~\ref{lem;expansion_only_with_A} below applied to the space $W$ and $\mu'=\mu_{A_c}$ with $v_n=m_n\dotsm m_1\acts v$ thus implies that the first term in \eqref{MA_decomp} has strictly positive limit inferior $\mu_{MA_c}^\N$-almost surely. This finishes the proof.
\end{proof}

\begin{lem}\label{lem;expansion_only_with_A}
Let $V$ be a finite-dimensional real vector space and $A'\leqs\GL(V)$ a closed connected diagonalizable subgroup with Lie algebra $\mathfrak a$. Write $V=\bigoplus_{\chi\in\Phi(A')} V^\chi$ for the weight space decomposition of $V$ with respect to $A'$, where $V^\chi=\set{v\in V\for av=\chi(a)v\text{ for all }a\in A'}$ and $\Phi(A')$ is the set of characters $\chi$ of $A'$ such that $V^\chi\neq\set{0}$. Let $\mu'$ be a probability measure on $A'$ with finite first moment and denote $a_{\operatorname{avg}}(\mu')=\exp(\int\log(a)\dd\mu'(a))$. Suppose that $\chi(a_{\operatorname{avg}}(\mu'))>1$ for every $\chi\in\Phi(A')$. Then for $(\mu')^\N$-a.e.\ $(a_i)_i\in (A')^\N$ we have
\begin{align*}
\liminf_{n\to\infty}\frac1n\log\frac{\norm{a_n\dotsm a_1 v_n}}{\norm{v_n}}>0
\end{align*}
for every choice of nonzero vectors $v_n\in V$.
\end{lem}
\begin{proof}
For convenience, we assume the norm $\norm{\cdot}$ on $V$ is Euclidean and that the distinct weight spaces are orthogonal. Given a nonzero $v \in V$, write $v=\sum_{\chi\in\Phi(A')} v^\chi(v)$ for the corresponding weight space decomposition, where $v^\chi(v) \in V^\chi$. Then for any $a_1,\dots,a_n\in A'$ and nonzero $v_n\in V$ we have
\begin{align*}
a_n\dotsm a_1v_n =\sum_{\chi\in\Phi(A')} \chi(a_n\dotsm a_1) v^\chi(v_n).
\end{align*}
Choosing for every $n\in\N$ a character $\chi_n$ such that $\norm{v^{\chi_n}(v_n)}\ge \frac{1}{\sqrt{\dim(V)}}\norm{v_n}$ and recalling that $\chi(a_{\operatorname{avg}}(\mu'))>1$ for all $\chi\in\Phi(A')$ by assumption, we conclude that
\begin{align*}
\frac1n\log\frac{\norm{a_n\dotsm a_1v_n}}{\norm{v_n}} &\ge o(1) + \frac{1}{n}\log \chi_n(a_n\dotsm a_1)  \ge o(1)+\min_{\chi\in\Phi(A')}\frac{1}{n} \sum_{i=1}^n \log \chi(a_i)\\
&\overset{n\to\infty}{\longrightarrow} \min_{\chi\in\Phi(A')}\log \chi(a_{\operatorname{avg}}(\mu'))>0,
\end{align*}
where the last convergence holds $(\mu')^\N$-almost surely by the classical law of large numbers.
\end{proof}

One drawback of Proposition~\ref{prop.examples} is that, in some sense, it requires the $M$- and $A_c$-parts of $\mu$ to both exhibit expansion (or at least non-contraction) individually. It would be natural to only ask the combination of both to be expanding, a behavior which should be reflected in the $A$-average of $\mu$. When $\mu$ does not charge $M$ in a too complicated way, we can also prove $H$-expansion in this case.

To state this second criterion, let $U\leqs H$ be any connected $\Ad$-unipotent subgroup. Then there exists a parabolic subgroup $Q$ of $H$ containing $U$ in its unipotent radical such that also $N_H(U)\leqs Q$~\cite{borel-tits}. As before, let $A\leqs Q$ be a maximal $\R$-split torus and denote by $K$ a maximal compact subgroup of $Q$. Given a non-trivial subtorus $A' \leqs A$ normalizing $U$, set $K'=C_K(A')\cap N_H(U)$ and let $P$ be the closed subgroup $K' A' U$ of $Q$. We write $\lambda\colon P \to \mathfrak{a}$ for the morphism given by $\lambda(kau)= \log a$.
\begin{prop}[$H$-expanding measures (2)]\label{prop.examples2}
Retain the notation from the paragraph above and let $\mu$ be a probability measure on $H$ with finite first moment such that $\mu(P)=1$. Denote by $a_{\operatorname{avg}}(\mu)=\exp\bigl(\int\lambda(g)\dd\mu(g)\bigr)\in A$ the $A$-average of $\mu$. Suppose that:
\begin{enumerate}[label=\textup{(\arabic*)}]
    \item The Zariski closure of $\Ad(\Gamma_\mu)$ contains $\Ad(U)$, and
    \item $U$ is $a_{\operatorname{avg}}(\mu)$-expanding.
\end{enumerate}
Then $\mu$ is $H$-expanding.
\end{prop}
We emphasize that, in contrast to Proposition~\ref{prop.examples}, here the $A$-average is considering also the part of the torus $A$ inside $M$, if $Q=MA_cN$ is the Langlands decomposition of $Q$.

%Remark: if $U=N$ is the unipotent radical of $Q$, then by considering the expanding cone, it seems that $N$ being $a_{\operatorname{avg}}(\mu)$-expanding implies $a_{c,\operatorname{avg}}(\mu)$-expanding, by considering the projection onto $\mathfrak{a}_c$. 
%Indeed, the center of the Levi factor is the orthogonal complement of the sum of spans of all coroots used to define the parabolic. All root reflections associated to the latter are realized by the action of the Weyl group of the Levi factor, hence stabilize the set $\Phi(\mathfrak u)$. Therefore the expanding cone is symmetric around the center and the projection to it preserves the cone.

\begin{proof}
Exactly as in the proof of Proposition~\ref{prop.examples}, given a non-trivial irreducible representation $(\rho,V)$ of $H$, it suffices to prove that
\begin{align*}
\liminf_{n\to\infty}\frac1n\log\norm{g_n\dotsm g_1\acts v}>0
\end{align*}
for $\mu_{K'A'}^\N$-a.e.\ $(g_i)_i\in H^\N$ and every $v\in V^U$, where $\mu_{K'A'}$ is the pushforward of $\mu$ by the map $K'A'U\to K'A',\, kau\mapsto ka$. As $K'$ is compact and commutes with $A'$, we can ignore the $K'$-component and consider only $\mu_{A'}$, defined in the analogous way. Now the statement follows from Lemma~\ref{lem;expansion_only_with_A}.
\end{proof}
\subsubsection{Expanding cone}\label{subsec;exp_cone}
Now we present a construction which can be used to ensure the expansion condition on $U$ with respect to the $A$- or $A_c$-average of $\mu$ in the criteria above (condition~(2) in Propositions~\ref{prop.examples} and \ref{prop.examples2}) in the case that $U$ is the unipotent radical of an absolutely proper parabolic subgroup $Q$ of $H$, where ``absolutely proper'' means that the projection of $Q$ to each simple factor of $H$ is non-surjective. As before, we let $A$ be a maximal connected $\R$-split torus of $Q$.

The \emph{expanding cone} of $U$ in $A$ is defined  to be 
\begin{align*}
A^+_U=\set{a\in A\for  U  \text{ is } a\text{-expanding}}. 
\end{align*}
It is proved in~\cite[Theorem~1.2]{s15} that 
$A_U^+$ only depends on the Lie algebras $\mathfrak h\df\Lie(H)$ and $\mathfrak u\df\Lie(U)$,  and that it can be described explicitly as follows. Let $\mathfrak a$ be the Lie algebra of $A$ and let $\Sigma(\mathfrak h, \mathfrak a)\subset \mathfrak a^*\df\mathrm{Hom}(\mathfrak a, \R)$ be the restricted root system of $(\mathfrak h,\mathfrak a)$. Denote by $\Sigma(\mathfrak u)\subset \Sigma(\mathfrak h, \mathfrak a)$ the subset of roots whose eigenvectors lie in $\mathfrak u$. 
Recall that by semisimplicity, the Killing form $\langle \cdot, \cdot\rangle $ of $\mathfrak h$ is positive definite on $\mathfrak a$. So for each $\alpha\in \mathfrak a^*$
we can associate $s_\alpha\in \mathfrak a$ by $\langle s_\alpha, v\rangle=\alpha( v)$ for every $ v\in \mathfrak a$.
Using this isomorphism, we associate to $\Sigma(\mathfrak u)$ the following convex cone in $\mathfrak a$:
\begin{align*}
\mathfrak a^+_{\mathfrak u}\df
\set[\Big]{
\sum_{\alpha\in \Sigma(\mathfrak u)} t_\alpha s_\alpha \for t_\alpha> 0
}.
\end{align*}
The expanding cone $A_U^+$ of $U$ is then given by $A_U^+= \exp \mathfrak a^+_{\mathfrak u}$; see \cite[Theorem~1.2]{s15}. By abuse of language, we shall sometimes also refer to $\mathfrak{a}_{\mathfrak u}^+$ as the expanding cone of $U$.

Using these notions, we get the following immediate corollary of Proposition~\ref{prop.examples2}.
\begin{cor}\label{cor;cone_expansion}
Let $U$ be the unipotent radical of an absolutely proper parabolic subgroup $Q$ of $H$, $A\leqs Q$ a maximal connected $\R$-split torus and $A'\leqs A$ a non-trivial subtorus. Moreover, let $K$ be a maximal compact subgroup of $H$, $K'=C_K(A')\cap Q$, set $P=K'A'U$ and let $\mu$ be a probability measure on $H$ with finite first moment such that $\mu(P)=1$. Suppose that the Zariski closure of $\Ad(\Gamma_\mu)$ contains $\Ad(U)$ and that $\int\lambda(g)\dd\mu(g)\in\mathfrak{a}_{\mathfrak u}^+$. Then $\mu$ is $H$-expanding.
\end{cor}

\subsubsection{Explicit examples}
We end this subsection by giving two explicit examples where the criteria developed so far are applicable.

The first of them is the prototypical example of an expanding cone. Although simple, it turns out to be of significant importance to Diophantine approximation problems on fractals. We will take up this point and elaborate on the connection in \S\ref{sec;diophantine}. 
\begin{ex}\label{ex;main} 
Let $H=\SL_{m+n}(\R)$, and
\begin{align*}
 Q &=\set*{
 \begin{pmatrix}
 p_{11} & p_{12}\\
 0 & p_{22} 
 \end{pmatrix}
 \in H\for p_{11}\in \GL_m(\R), p_{22}\in \GL_n(\R), p_{12}\in \Mat_{m\times n}(\R)
 } ,\\
 U &=\set*{
 \begin{pmatrix}
 I_m & p_{12}\\
 0 & I_n 
 \end{pmatrix}
 \in H\for p_{12}\in \Mat_{m\times n}(\R)
 },
\end{align*}
where we denote by $I_d$ the $d\times d$-identity matrix. The group $A$ consists of diagonal matrices in  $H$ with positive entries, and we have 
\begin{align*}
A^+_U=\set{\diag(e^{r_1}, \dots, e^{r_m}, e^{-s_1}, \dots, e^{-s_n})\in H\for r_i, s_j>0}
\end{align*}
(see~\cite[Example~1.1]{s15}).

For concreteness, we exemplify a class of $H$-expanding measures on $Q$: Fix a Borel subset $B_U$ of $U$ not contained in a proper vector subspace of $U\cong \R^{mn}$. For example, $B_U$ can be taken to be a non-degenerate curve in $U$ or a collection of $k \ge mn$ points in $U \cong \R^{mn}$ that linearly spans $U$. Let $\mu$ be a compactly supported  probability measure on $AU$ such that
\begin{itemize}
\item its support contains an element of $A_U^+$,
\item  the set of unipotent parts $u_g$ of elements $g=a_gu_g$ in $\supp(\mu)\subset AU$ contains $B_U$, and
\item its $A$-average lies in the expanding cone of $U$, i.e.\ $\int \lambda(g) \dd\mu(g) \in \mathfrak{a}_\mathfrak{u}^+ $.
\end{itemize}
Then $\mu$ can be seen to be $H$-expanding by Corollary~\ref{cor;cone_expansion}. Indeed, as we will see in \S\ref{sec;diophantine} on Diophantine approximation on fractals, the first two points above imply that the Zariski closure of $\Gamma_\mu$ contains $U$ (see the proof of Theorem~\ref{thm;dioph.insection}).
\end{ex} 
Note that the above example covers in particular~\cite[Example~2.8]{prohaska-sert}. We also point out that, in Example~\ref{ex;main}, the assumption that $\supp(\mu)$ contains an element of $A_U^+$ is not strictly necessary. The first two bullet points could be replaced by a certain ``irreducibility condition'' of an affine action of the group generated by the support of $\mu$ (which is what we will do in \S\ref{sec;diophantine}), or, alternatively, by the assumption that the commutator group $[\Gamma_\mu,\Gamma_\mu]$ is Zariski dense in $U$.

The second example is one where the reductive group $M$ in the Langlands decomposition of $Q$ (see the paragraph before Proposition~\ref{prop.examples}) contributes to expansion in a non-trivial way.
\begin{ex}\label{ex.explicit.expand}
Let $Q$ be the standard parabolic subgroup of $\SL_4(\R)$ given by
\begin{align*}
Q=\begin{pmatrix}
 \ast & \ast & \ast & \ast \\
 \ast & \ast & \ast & \ast \\
  & & \ast & \ast \\
  & & & \ast
\end{pmatrix}\leqs\SL_4(\R).
\end{align*}
The maximal connected $\R$-split torus $A$ consists of diagonal matrices with positive entries. In the Langlands decomposition $Q=MA_cN$ we have
\begin{gather*}
A_c=\set[\Big]{
d_{\alpha,\beta}\df\diag\bigl((\alpha \beta)^{-1/2}, (\alpha \beta)^{-1/2}, \alpha, \beta\bigr)
\for\alpha, \beta>0}
,\\
M=\begin{pmatrix}\SL_2(\R)&\\&I_2\end{pmatrix},\text{ and } N=\begin{pmatrix}
 1 &  & \ast & \ast \\
  & 1 & \ast & \ast \\
  & & 1 & \ast \\
  & & & 1
\end{pmatrix}.
\end{gather*}
Using the explicit description of the expanding cone in \S\ref{subsec;exp_cone}, one can calculate directly that the intersection of the expanding cone of $U=N$ in $A$ with $A_c$ is given by
\begin{align*}
A_c\cap A_U^+ = \set{d_{\alpha,\beta}\for \beta<1,\alpha\beta<1}.
\end{align*}
For $i,j \in \set{1,2,3,4}$ let $u_{i,j}$ be the unipotent element whose only nonzero off-diagonal term is $1$ at the $(i,j)$-entry. Let $g=(\begin{smallmatrix}
 1 & 1\\
 & 1
\end{smallmatrix})$ and consider the element $s$ of $Q$ given by the block diagonal matrix $s=(g,I_2)$. Now let $\mu$ be any compactly supported probability measure on $Q$ whose support is given by the union of $\set{s, s^\top, u_{2,3}, u_{3,4}}$ and some diagonal matrices $d_{\alpha, \beta}$ in $Q$. It is not difficult to see that $U\leqs \Zcl(\Gamma_\mu)$ and the $A_c$-part $\mu_{A_c}$ of $\mu$ consists of the latter diagonal matrices. Moreover, $M$ is semisimple and the $M$-part of $\mu$ is Zariski dense in $M$. So, in view of Propositions~\ref{prop.examples} and \ref{prop;Z-dense}, provided that the integral $\int (\log \alpha,\log \beta) \dd\mu_{A_c}(d_{\alpha,\beta})$ is in the cone in $\R^2$ defined by the inequalities $x+y<0$ and $y<0$, the measure $\mu$ is $\SL_4(\R)$-expanding. 
%if $a_1,\dots,a_4$ are the diagonal coordinates, the following inequalities describe the expanding cone $\mathfrak a_{\mathfrak u}^+$
%\begin{align*}
%a_1+a_2+a_3+a_4&=0\\
%a_1-3a_2+a_3+a_4&<0\\
%a_1+a_2+a_3-3a_4&>0\\
%-3a_1+a_2+a_3+a_4&<0
%\end{align*}
%
\end{ex}
\subsection{Split solvable epimorphic subgroups}\label{sec;solv_epi}
The goal of this part is to discuss a further class of $H$-expanding measures. They will be supported on solvable epimorphic subgroups $F=A'U$ of semisimple real algebraic groups $H$, where $A'$ is a one-dimensional algebraic $\R$-split torus and $U$ is unipotent. The arguments rely on Proposition~\ref{prop.examples2}, ideas going back to  Weiss~\cite{weiss.epimorphic} and Shah--Weiss~\cite{shah-weiss}, and the work of Bien--Borel~\cite{bien-borel1,bb3}.

We start with a brief discussion of epimorphic subgroups, which have close connections to the notion of $H$-expanding measures.

\subsubsection{Epimorphic subgroups}\label{subsec;epimorphic}

The concept of epimorphic subgroups of algebraic groups was introduced by Bien--Borel~\cite{bien-borel1,bien-borel2}. In~\cite{s15}, this notion was adapted to subgroups of connected semisimple Lie groups without compact factors.
\begin{de}
A subgroup $F$ of $H$ is said to be \emph{epimorphic} in $H$ if for every representation of $H$, the vectors fixed by $F$ are also fixed by $H$.
\end{de}
It can be shown that if $H$ is almost algebraic in the sense of Remark~\ref{rem;alg} and $F\leqs H$ is a connected Lie subgroup or a Zariski connected algebraic subgroup, it suffices to check the epimorphic property of $F$ in real algebraic representations of $H$ (see Proposition~\ref{prop;defs_coincide}). Consequently, in the algebraic category the above definition coincides with that of Bien--Borel. Moreover, it follows that a connected Lie subgroup $F$ is epimorphic in $H$ if and only if its Zariski closure $\Zcl(F)$ is.

Mozes~\cite{mozes95} proved that an $F$-invariant probability measure on $G/\Lambda$ is already invariant under $H$ (and thus homogeneous by Ratner's theorem) in the case where all of $F,H,G$ are real algebraic groups. This measure rigidity result was later generalized by Shah--Weiss~\cite[Theorem~1.8]{shah-weiss} to actions of connected epimorphic Lie subgroups which are not necessarily algebraic.
 
Examples of epimorphic subgroups include parabolic subgroups of $H$ and Zariski dense subgroups, in case $H$ is {almost algebraic}. One may notice that these classes of subgroups also prominently featured in the previous parts of this section, where we gave our first examples of $H$-expanding measures. That this is not a coincidence becomes clear with the following observation.
\begin{prop}\label{prop;epimorphic}
If $\mu$ is $H$-expanding, then the closed subgroup $\Gamma_\mu$ generated by the support of $\mu$ is epimorphic in $H$.
\end{prop}

\begin{proof}
In any given representation $(\rho,V)$ of $H$, a $\Gamma_\mu$-fixed vector $v\in V$ cannot be $\rho_*\mu$-expanded. In view of Proposition~\ref{prop;facts}(i), it follows that $v$ is $H$-fixed.
\end{proof}

On the other hand, there exist connected epimorphic subgroups of $H$ which do not support any $H$-expanding probability measure.
\begin{ex}\label{ex;counter}
We take $H=\SL_3(\R)$, $A'=\set{\diag(e^t, e^{-\sqrt 2 t},e^{(\sqrt 2-1)t} ) \for t\in\R}$ and $U$ to be as in Example~\ref{ex;main} for $m=2, n=1$. The Zariski closure of $A'U$ contains $AU$ where $A\leqs H$ is the diagonal subgroup with positive entries. 
It follows that $A'U$ is an epimorphic subgroup of $H$, since $AU$ is. On the other hand, $A'$ has empty intersection with the expanding cone $A_U^+$ which is 
described explicitly in Example~\ref{ex;main}. Therefore, for any probability measure $\mu$ on $A'U$ with finite first moment, we have 
\begin{align*}
a\df a_{\operatorname{avg}}(\mu)\not \in A_U^+,
\end{align*}
where $a_{\operatorname{avg}}$ is as in Proposition~\ref{prop.examples2}. It follows from the definition of the expanding cone that there is a non-trivial irreducible representation $V$ of $H$ such that $V^U\cap (V_a^-\oplus V_a^0) \neq \set{0}$.
Therefore, $\mu$ is not $H$-expanding. 
\end{ex}
We point out that the phenomenon in the above example crucially depends on the one-dimensional torus $A'$ not being algebraic, as the discussion in the upcoming part will show.

\subsubsection{Expanding rays in one-dimensional algebraic tori}
We now state an observation (Lemma \ref{lemma.1.torus}) ensuring the expansion of the unipotent part of a split solvable group with respect to its one dimensional torus. Based on this observation, in \S \ref{subsub.bb.examples} we will outline two constructions due to Bien--Borel--Koll\'{a}r~\cite{bb3}, which, thanks to Proposition~\ref{prop.examples2}, yield further classes of $H$-expanding measures with small support on a semisimple group $H$.

Let $H$ be a connected almost algebraic semisimple real Lie group without compact factors and $F$ a connected epimorphic subgroup of $H$ of the form $F=A'U$ where $A'$ is a connected algebraic $\R$-split torus and $U$ is a unipotent subgroup of $H$ normalized by $A'$. It is known that any connected algebraic epimorphic subgroup of $H$ contains an epimorphic subgroup of this form~\cite[\S10,~Theorem~2]{bien-borel1}.

The following lemma can be proved in a similar way as Lemma \ref{lem;expansion_only_with_A} using additionally \cite[Lemma~1]{weiss.epimorphic}. We omit the routine details of the proof for brevity.
\begin{lem}\label{lemma.1.torus}
Let $H$ and $F=A'U$ be as above and suppose that $A'$ is one-dimensional. Then there exists a parametrization $A'=(a(t))_{t\in\R}$ as one-parameter subgroup such that for every representation $(\rho,V)$ of $H$ and $U$-fixed vector $v\in V^U$, either $v$ is $H$-fixed or $\lim_{t\to\infty}\norm{\rho(a(t))v}=\infty$. For such a parametrization, $U$ is $a(t)$-expanding in the sense of \textup{\S\ref{sec;parabolic}} for every $t>0$. \qed
\end{lem}

\subsubsection{Examples}\label{subsub.bb.examples}
Let $H$ be a connected almost algebraic semisimple real Lie group and denote its Lie algebra by $\mathfrak h$. Let $Z$ be a one-parameter unipotent subgroup of $H$ and $z$ a generator of the Lie algebra of $Z$. By the Jacobson--Morozov theorem $z$ is part of an $\mathfrak{sl}_2$-triple $(a,z,z_-)$. Let $\mathfrak s$ be the Lie algebra spanned by this triple and $S$ the corresponding connected subgroup of $H$. Let $A'$ be the one-parameter diagonalizable subgroup with Lie algebra spanned by $a$. Via the adjoint representation, write $\mathfrak h$ as direct sum of the centralizer $\mathfrak z_o$ of $\mathfrak s$ and of non-trivial irreducible $\mathfrak s$-submodules $\mathfrak m_1=\mathfrak s,\mathfrak m_2,\dots,\mathfrak m_k$.
\begin{ex}[{\cite[Proposition~4.5]{bb3}}]
Retain the notation from the paragraph above and suppose that $z$ has non-trivial projections to each of the simple factors of $\mathfrak h$. Let $z_i$ be highest weight vectors of the irreducible $\mathfrak s$-modules $\mathfrak m_i$, with $z_1=z$. Write $\mathfrak u$ for the direct sum of their $\R$-spans. %By the representation theory of $\mathfrak{sl}_2$, $\mathfrak u$ is normalized by $a$.
Denoting by $U$ the corresponding unipotent subgroup of $H$, it follows that $F=A'U$ is a split solvable algebraic subgroup of $H$, which can be seen to be epimorphic in $H$ thanks to~\cite[Proposition~4.5]{bb3}. Therefore, by virtue of Proposition \ref{prop.examples2}, we see that any probability measure $\mu$ on $F$ whose $A'$-average lies in the expanding ray given by Lemma~\ref{lemma.1.torus} is $H$-expanding.
%Indeed, the idea is that the Lie algebra $\mathfrak h'$ generated by $\mathfrak s$ and $\mathfrak u$ is all of $\mathfrak h$. To see this, note that $\mathfrak h'$ contains the sum $\mathfrak m$ of all $\mathfrak m_i$ by the representation theory of $\mathfrak s$, which implies that $\mathfrak h=\mathfrak z_o+\mathfrak h'$. Moreover, $\mathfrak m$ is stabilized by $\mathfrak z_o$, since the former is the direct sum of isotypic components of non-trivial irreducible $\mathfrak s$-submodules of $\mathfrak h$. Thus, $\mathfrak z_o$ also stabilizes $\mathfrak u$, since this is the centralizer of $z$ in $\mathfrak m$. It follows that $\mathfrak z_o$ stabilizes $\mathfrak h'$, hence the latter is an ideal in $\mathfrak h$. Since by assumption, $z\in\mathfrak h'$ has non-trivial projection to each of the simple factors of $\mathfrak h$, we conclude that, indeed, $\mathfrak h'=\mathfrak h$.
\end{ex}

\begin{ex}[{\cite[\S4.6]{bb3}}]
Retain the notation from above. Suppose that $H$ is an $\R$-split simple real algebraic group and that the one-parameter unipotent subgroup $Z$ of $H$ contains ``regular'' unipotent elements. For example, the generator $z$ can be taken as sum of eigenvectors for all positive  simple roots of $\mathfrak h$. Then the subgroup $S$ whose Lie algebra is spanned by the $\mathfrak{sl}_2$-triple $(a,z,z_-)$ is a ``principal TDS'' (three-dimensional subgroup) in $H$. It is known that either $S$ is properly contained in exactly one proper connected subgroup $R$ of $H$, or $S$ is maximal among proper connected subgroups of $H$, in which case we set $R=S$. See Kostant~\cite{kostant} for a treatment of the notions used here. Choose $\mathfrak m_j$ so that it does not intersect the Lie algebra $\mathfrak r$ of $R$ and let $Z_j$ be the subgroup of $H$ whose Lie algebra is generated by a highest weight vector of $\mathfrak m_j$. Then, as discussed in \cite[\S4.6]{bb3}, $F=A'ZZ_j$ is a three-dimensional split solvable algebraic epimorphic subgroup of $H$. Therefore, as in the previous example, three-dimensional solvable subgroups obtained by this construction support many $H$-expanding measures thanks to Proposition~\ref{prop.examples2} and Lemma~\ref{lemma.1.torus}.
\end{ex}

We end this section by mentioning an ensuing question, which was also posed to us by Barak Weiss.

\begin{qu} Let $H$ be a semisimple real algebraic group without compact factors. Is it true that every algebraic epimorphic subgroup $F\leqs H$ supports an $H$-expanding probability measure? 
\end{qu}

The answer to the above question is negative if we do not require $F$ to be epimorphic (Proposition~\ref{prop;epimorphic}) or to be algebraic (Example~\ref{ex;counter}).

On the other hand, let $F=A'U$ be an $\R$-split solvable epimorphic subgroup of $F$, where $U$ is a unipotent group and $A'$ is an $\R$-split algebraic torus normalizing $U$. Then \cite[\S7,~Lemma~(iii)]{bien-borel1} provides a sufficient condition (in terms of finite-generation of a monoid generated by certain characters of $A'$) for $F$ to contain an $\R$-split solvable epimorphic subgroup $F_0=A'_0U$ with one-dimensional $\R$-split algebraic torus $A'_0<A'$. In view of Lemma~\ref{lemma.1.torus}, any such subgroup $F_0$ supports $H$-expanding probability measures. However, we do not know whether the hypothesis of the aforementioned lemma of Bien--Borel is always satisfied in the context of the question above, or whether a different construction can be used to obtain $H$-expanding probability measures on $F$ in case it is not.

\section{Measure rigidity}\label{sec;rigidity}

This section is dedicated to the statements outlined in \S\ref{sec;intro.rigidity}. In \S\ref{sec;rigidity_proof}, we first prove our general measure rigidity result (Theorem~\ref{thm;rigidity}), followed by a discussion of stationary measures charging an orbit of the centralizer in \S\ref{sec;centralizer}, which leads to the proof of Corollary~\ref{cor;rigidity}. Finally, we more closely analyze, in \S\ref{sec;exp_grass}, expansion in which representations is necessary to obtain the conclusion of Theorem~\ref{thm;rigidity}. This will yield a finite criterion weaker than $H$-expansion for measure rigidity to hold when the ambient Lie group $G$ is fixed.

\subsection{Rigidity for expanding measures}\label{sec;rigidity_proof}
Let $\Lambda$ be a discrete subgroup of a real Lie group $G$ and $X=G/\Lambda$. Moreover, we let $H\leqs G$ be a connected semisimple subgroup without compact factors and with finite center and $\mu$ a probability measure on $H$.
For the proof of Theorem~\ref{thm;rigidity}, we will follow the strategy in the proof of~\cite[Theorem~1.3]{el}. The argument is based on the following measure classification result of Eskin--Lindenstrauss.
\begin{de}[{\cite[Definition~1.6]{el}}]\label{de;modZ}
Let $Z$ be a connected Lie subgroup of $G$. A probability measure $\mu$ on $G$ is said to be \emph{uniformly expanding mod $Z$} if the following hold:
\begin{enumerate}[label=(\alph*)]
\item $Z$ is normalized by $\Gamma_\mu$,
\item the conjugation action of $\Gamma_\mu$ on $Z$ factors through the action of a compact subgroup of $\Aut(Z)$, and
\item there is a $\Gamma_\mu$-invariant direct sum decomposition $\mathfrak g=\Lie(Z)\oplus V$ such that $\mu$ is uniformly expanding on $V$.
\end{enumerate}
\end{de}
\begin{thm}[{Eskin--Lindenstrauss~\cite[Theorem~1.7]{el}}]\label{thm;el}
Let $G$ be a real Lie group and $\Lambda<G$ a discrete subgroup. Suppose that $\mu$ is a probability measure on $G$ with finite first moment for which there exists a connected Lie subgroup $Z$ of $G$ such that $\mu$ is uniformly expanding mod $Z$. Let $\nu$ be any ergodic $\mu$-stationary probability measure on $G/\Lambda$. Then one of the following holds:
\begin{enumerate}[label=\textup{(\alph*)}]
\item There exists a closed subgroup $N\leqs G$ with $\dim(N)>0$, an $N$-homogeneous probability measure $\nu_0$ on $G/\Lambda$, and a $\mu$-stationary probability measure $\eta$ on $G/N$ such that
\begin{align*}
    \nu=\int_{G/N}g_*\nu_0\dd\eta(g).
\end{align*}
\item The measure $\nu$ is $\Gamma_\mu$-invariant and supported on a finite union of compact subsets of $Z$-orbits.
\end{enumerate}
\end{thm}

The following two lemmas will go into the proof of Theorem~\ref{thm;rigidity}. 
\begin{lem}
	\label{lem;modZ}
	Suppose that $\mu$ is $H$-expanding. Then the Lie algebra $\mathfrak g$ of $G$ admits an $H$-invariant direct sum decomposition $\mathfrak g=\mathfrak l\oplus \mathfrak v$, where $\mathfrak l$ is the Lie algebra of the centralizer $L$ of $\Gamma_\mu$ in $G$ and $\mathfrak v\subset\mathfrak{g}$ is a subspace on which $\mu$ is uniformly expanding. In particular, $\mu$ is uniformly expanding mod $L^\circ$ in the sense of Definition~\textup{\ref{de;modZ}}.
\end{lem}
\begin{proof}
Since, by Proposition~\ref{prop;epimorphic}, $\Gamma_\mu$ is epimorphic in $H$, $\mathfrak l$ is the space of $H$-fixed vectors in the adjoint representation of $G$. Semisimplicity thus implies the existence of an $H$-invariant complementary subspace $\mathfrak v$. Now the claim follows directly from the definition of $H$-expansion. 
\end{proof}
The second lemma concerns $\mu$-stationary measures assigning positive mass to centralizer orbits.
\begin{lem}[{\cite[Lemma~7.6]{bq13}}]
\label{lem;key}
Suppose that $\nu$ is an ergodic $\mu$-stationary probability measure on $X$  such that $\nu$ assigns positive  mass to some $L$-orbit in $X$, where $L=C_G(\Gamma_\mu)$. Let $L_0$ be any open subgroup of  $L\cap  \Stab_G(\nu)$.
Then  $\nu$ is homogeneous under the closed subgroup  $\Gamma_\mu L_0$ and $L_0$ is open in $\Stab_G(\nu)$.
\end{lem}
We point out that the last claim in the statement above follows from the proof of~\cite[Lemma~7.6]{bq13}, where it is shown that the support of $\nu$ is a finite union of closed $L_0$-orbits which are transitively permuted by $\Gamma_\mu$. In fact, even more conclusions can be drawn in the context of this lemma; see Proposition~\ref{prop;old_key}.

\begin{proof}[Proof of Theorem~\textup{\ref{thm;rigidity}}]
Our main tool is Theorem~\ref{thm;el}. 
Its assumptions are satisfied, since by  Lemma~\ref{lem;modZ}, $\mu$ is uniformly expanding mod $L^\circ$, where $L$ denotes the centralizer of $\Gamma_\mu$ in $G$. 
If Theorem~\ref{thm;el}(b) holds, then by Lemma~\ref{lem;key}, $\nu$ is homogeneous and the connected component of $\Stab_G(\nu)$ is contained in $L$. By the epimorphic property of $\Gamma_\mu$ in $H$ from Proposition~\ref{prop;epimorphic} applied to the adjoint representation of $G$, the connected components of $C_G(\Gamma_\mu)$ and $C_G(H)$ coincide. Thus, it follows that the connected component of $\Stab_G(\nu)$ is centralized by $H$.
	
If Theorem~\ref{thm;el}(a) holds, then there exists a closed subgroup $N$ of $G$ with $\dim(N)>0$,  an $N$-homogeneous probability measure $\nu_0$ on $G/\Lambda$, 
	and a $\mu$-stationary probability measure $\eta $ on $G/N$ such that 
\begin{align}\label{nu_decomp}
	\nu=\int _{G/N} g_*\nu_0 \dd \eta (g). 
\end{align}
	We may assume that $\eta$ is $\mu$-ergodic. Indeed, if $\eta=\int_ Y  \eta_y \dd y$	is a $\mu$-ergodic decomposition of $\eta$, then 
	\begin{align*}
	\nu =\int_Y \biggl(\int_{G/N} g_* \nu_0 \dd \eta_y(g)\biggr)\dd y
	\end{align*}
	is a convex decomposition of
	$\nu$ into $\mu$-stationary
	 measures. 
	Since $\nu$ is $\mu$-ergodic, we must have $\nu=\int_{G/N} g_* \nu_0 \dd \eta_y(g) $ for 
	almost every $y$. Thus, we can replace $\eta$ by one of the $\eta_y$, if necessary. We consider $N$ such that $\dim(N)$ is maximal among possible representations of $\nu$ of the form \eqref{nu_decomp}.

Now consider the adjoint action of $G$ on $S^2(\mathfrak{g}^{\wedge\dim(N)})$, where $S^2$ denotes the symmetric square representation. Let $\omega=v\otimes v$, where $v\in\mathfrak{g}^{\wedge\dim(N)}$ corresponds to a basis of the Lie algebra of $N$. Let $P$ be the stabilizer of $\omega$ in $G$. Since $N$ admits a lattice, it is unimodular, so that $N$ acts on $v$ by $\pm 1$. Thus, $N$ fixes $\omega$, that is $N\leqs P$. 
Let $\eta'$ be the pushforward of $\eta$ via the natural projection map $G/N\to G/P$. The measure $\eta'$ can be thought of as an ergodic $\mu$-stationary measure on $S^2(\mathfrak{g}^{\wedge\dim(N)})$. By  Corollary~\ref{cor;stationary}, the measure $\eta'$ must concentrate on the subspace of $H$-fixed vectors. Then by ergodicity, $\eta'$ is a Dirac measure.  After replacing $N$ and $P$ by their conjugates, we may assume without loss of generality that $\eta'$ is the Dirac measure on the coset $P$. It follows that $\omega$ is $H$-fixed. Hence $H\leqs P$ and $H\cap N^\circ$ is a normal subgroup of $H$. If $H\leqs N^\circ$, then the action of $H$ on $P/N$ is trivial, so that by ergodicity of $\eta$ we have 
$\nu=g_*\nu_0$ for an element $g\in P$ with $\supp(\eta)=\set{gN}$ and we are done.

So let us now assume that $H$ is not contained in $N^\circ$. In this case, we consider the action of $(H/(H\cap N^\circ), \mu') $ on $P/N \cong (P/N^\circ)/(N/N^\circ)$ with the $\mu'$-stationary measure $\eta$, where $\mu'$ is the pushforward of $\mu$ under the natural projection map $H\to H/(H\cap N^\circ)$. Since $\mu$ is $H$-expanding and $H$ is not contained in $N^\circ$, $\mu'$ is $H/(H\cap N^\circ)$-expanding in view of Proposition~\ref{prop;facts}(iii). Now, in view of Lemma~\ref{lem;modZ}, we are in a position to apply Theorem~\ref{thm;el} again for $\mu'$. We claim that thanks to the choice of $N$ as having maximal dimension in \eqref{nu_decomp}, the case (a) in Theorem~\ref{thm;el} does not occur. Suppose it does. This means that there exist a closed subgroup $M<P/N^{\circ}$ of positive dimension, an $M$-homogeneous probability measure $\nu_0'$ on $P/N$ and a $\mu'$-stationary probability measure $\eta'$ on $(P/N^{\circ})/M$ such that we have
	\begin{align}\label{eta_decomp}
	    \eta=\int_{(P/N^{\circ})/M} g_* \nu_0' \dd\eta'(g).
	\end{align}
	Denote by $\hat{M}$ the preimage of $M$ under the projection $P \to P/N^{\circ}$ so that we can identify $(P/N^{\circ})/M$ with $P/\hat{M}$. By combining \eqref{nu_decomp} and \eqref{eta_decomp}, we deduce that
	\begin{equation*}
	    \nu=\int_{P/N} \int_{P/\hat{M}} (gh)_*\nu_0 \dd\eta'(g) d\nu_0'(h)=\int_{P/\hat{M}} g_* \biggl(\int_{P/N} h_* \nu_0 \dd\nu_0'(h)\biggr) \dd\eta'(g).
	\end{equation*}
Now it is easily observed that the probability measure $\Psi=\int_{P/N} h_* \nu_0 \dd\nu_0'(h)$ on $G/\Lambda$ is $\hat{M}$-invariant and supported on finitely many $\hat{M}$-orbits. 
By $\mu$-ergodicity of $\nu$, for every $\hat{M}$-ergodic component $\Psi_y$ of $\Psi$, we have $$\nu=\int_{P/\hat{M}}g_*\Psi_y\dd\eta'(g).$$ Take such a component $\Psi_y$ which assigns positive mass to an $\hat{M}$-orbit. Then $\Psi_y$ is $\hat{M}$-homogeneous and the fact that $\dim(\hat{M})>\dim(N)$ yields a contradiction to the maximality of $\dim(N)$ in \eqref{nu_decomp}.

Therefore we can conclude by case (b) of Theorem~\ref{thm;el} that $\eta$ is $\Gamma_{\mu'}$-invariant and supported on finitely many compact subsets of $C_{P/N^\circ}(\Gamma_{\mu'})$-orbits. 
    By Lemma~\ref{lem;key}, $\eta$ is $M$-homogeneous for a closed subgroup $M<P/N^\circ$. In particular, $\eta$ can be written in the form \eqref{eta_decomp} with $\nu_0'=\eta$ and $\eta'$ the Dirac mass at the identity coset, the latter being $\mu'$-stationary since $\eta$ is $\Gamma_{\mu'}$-invariant. As we have argued above, this cannot happen if the support of $\eta$ has positive dimension. Thus, $\eta$ is a finite periodic orbit measure, and using \eqref{nu_decomp} it directly follows that $\nu$ is homogeneous. The connected component of $\Stab_G(\nu)$ is $N^\circ$, which is normalized by $H$, as we already established above. Hence, the proof is complete.
\end{proof}

\subsection{Stationary measures charging an orbit of the centralizer}\label{sec;centralizer}

The following proposition gives additional information about the measure $\nu$ in the setting of Lemma~\ref{lem;key}, or more generally, in the setting of~\cite[\S7.3]{bq13}. It will be used below to deduce Corollary~\ref{cor;rigidity}(i) from Theorem~\ref{thm;rigidity}. 

The general setting is as follows: $G$ is a locally compact second countable group, $\Lambda$ a discrete subgroup of $G$, $\mu$ is a probability measure on $G$, $L$ denotes the centralizer of $\Gamma_\mu$ in $G$, and $\nu$ is a $\mu$-ergodic $\mu$-stationary probability measure on $X=G/\Lambda$ assigning positive mass to some $L$-orbit. Finally, $L_0$ is any open subgroup of $L\cap\Stab_G(\nu)$.

\begin{prop}\label{prop;old_key}
Retain the notation and assumptions above and fix $x=g\Lambda\in \supp(\nu)$.
 Let $\nu_0$ be the restriction of $\nu$ to $L_0x$, $\Gamma_0$ the stabilizer of $\nu_0$ in $\Gamma_\mu$ and 
  \begin{align*}
 	\Gamma_0^L=\set{l \in L_0 \for \text{there exists } h \in \Gamma_0 \text{ such that } hl \in g \Lambda g^{-1}}.
 \end{align*}
Then in addition to the conclusion of Lemma~\textup{\ref{lem;key}}, the following holds:
\begin{enumerate}[label=\textup{(\roman*)}]
	\item  $\Gamma_0$ has finite index in $\Gamma_\mu$,
	\item $\Gamma_0^L$ is a dense subgroup of $L_0$ with $\Gamma_0x=\Gamma_0^Lx$, and 
	\item  $L_0\cap g\Lambda g^{-1}$ is a cocompact normal subgroup of $L_0$.
\end{enumerate}
 In particular, $\nu$ is compactly supported and is the unique ergodic $\mu$-stationary probability measure on $X$ assigning positive measure to $\supp(\nu)$.
\end{prop}

\begin{proof}
By~\cite[Lemma~7.6]{bq13} and its proof, we know that $\nu$ is the homogeneous measure 
on $\Gamma_\mu L_0 x$ and that $\supp(\nu)$ consists of finitely many closed $L_0$-orbits which are transitively permuted by $\Gamma_\mu$. In particular, we have $\nu(L_0x)>0$. 
It follows  that $\Gamma_0$ has finite index in $\Gamma_\mu$. Moreover, since $\Gamma_\mu$ preserves $\nu$ and acts ergodically, the group $\Gamma_0$ acts ergodically with respect to $\nu_0$. This implies that we can find $l_0\in L_0$ such that $\Gamma_0 l_0 x$ is dense in $L_0 x$. As $l_0$ commutes with $\Gamma_0$, it immediately follows that $\Gamma_0x$ is dense in $L_0x$. Since $\Gamma_0^L$ is precisely defined for $\Gamma_0x=\Gamma_0^Lx$ to hold, we conclude that $\Gamma_0^L=\Gamma_0^L(L_0\cap g\Lambda g^{-1})$ is dense in $L_0$.

We next prove that $L_0\cap g\Lambda g^{-1}$ is a cocompact normal subgroup of $L_0$. Since we have already shown that $\Gamma_0^L$ is dense in $L_0$, it suffices to show that $L_0\cap g\Lambda g^{-1}$ is normal in $\Gamma_0^L$. To see this, taking an arbitrary $l\in \Gamma_0^L$ and choosing $h\in \Gamma_0$ with $hl\in g\Lambda g^{-1}$, we calculate
\begin{align*}
l (L_0 \cap g\Lambda g^{-1}) l^{-1}= hl (L_0 \cap g\Lambda g^{-1}) (hl)^{-1}= L_0\cap g\Lambda g^{-1},
\end{align*}
where we used again that $\Gamma_\mu$ and $L_0$ commute. Since there is a finite $L_0$-invariant measure on the locally compact group $L_0/(L_0\cap g\Lambda g^{-1})$, the latter  must be compact.
	 
It remains to prove the uniqueness of $\nu$. Let $\nu'$ be an arbitrary ergodic $\mu$-stationary probability measure on $X$ with $\nu'(\supp(\nu))>0$. Take $x\in\supp(\nu)\cap\supp(\nu')$. Then by what we have shown above, $\nu'$ is homogeneous and $\supp(\nu)=\overline{\Gamma_\mu x}=\supp(\nu')$. Hence, $\nu=\nu'$ by homogeneity.
\end{proof}

Loosely speaking, the group $\Gamma_0^L$ in Proposition~\ref{prop;old_key} consists of translations in the centralizer direction arising from the action of $\Gamma_0\leqs \Gamma_\mu$ on the centralizer orbit under consideration. This is illustrated by the following simple example.
\begin{ex}
Let $G=\SL_2(\R)\times K$, where $K$ is a connected compact Lie group,
and let $\psi\colon\SL_2(\Z)\to K $ be a fixed group homomorphism with dense image. Let $X=G/\Lambda$ for the lattice
\begin{align*}
\Lambda=\set{ (\gamma, \psi(\gamma))\for \gamma\in \SL_2(\Z)}<G.
\end{align*}
Moreover, let
$H=\SL_2(\R)$, identified with the first factor of $G$, and choose a probability measure $\mu$ on $H$ with $\Gamma_\mu=\SL_2(\Z)<H$. Let $\nu$ be the homogeneous measure on the $K$-orbit of the identity coset $x=\Lambda$ in $G/\Lambda$, induced by the Haar probability measure on~$K$, where we identify $K$ with the second factor of $G$. Then  the action of $\gamma\in\Gamma_\mu$ on $x$ is given by 
\begin{align*}
(\gamma,1)x= (\gamma,1) (\gamma ^{-1}, \psi (\gamma^{-1}))x= (1,\psi (\gamma^{-1})) x.
\end{align*}
Thus, the $K$-orbit of $x$ is given by $(\Gamma_\mu \times K)/\Lambda$, $\nu$ is $\Gamma_\mu$-invariant, and also ergodic for the $\Gamma_\mu$-action since $\psi$ has dense image. If we set $L_0=K$, then in the notation of Proposition~\ref{prop;old_key} we have 
$\Gamma_0=\Gamma_\mu$
and $\Gamma_0^L=\psi(\Gamma_\mu)$, which is a dense subgroup of $L_0$.
\end{ex}
The key point of Corollary~\textup{\ref{cor;rigidity}}(i)
is that we cannot have examples of the type above when $X$ is the quotient of a semisimple group $G$ by an irreducible lattice $\Lambda$, such as $G/\Lambda=(\SL_2(\R)\times \SL_2(\R))/\SL_2(\Z[\sqrt{2}])$.

To keep the continuity, we now proceed to the proof of Corollary~\textup{\ref{cor;rigidity}}, even though one part of the statement relies on the countability result for homogeneous subspaces to be established in \S\ref{subsec;countability}. 
The central part of the proof makes heavy use of concepts from the theory of algebraic and arithmetic groups; in particular Margulis' arithmeticity theorem~\cite{margulis}. See the book by Witte Morris~\cite{wittemorris.arithmetic} for a gentle introduction to this topic.
\begin{proof}[Proof of Corollary~\textup{\ref{cor;rigidity}}]
Let $\nu$ be a $\mu$-ergodic $\mu$-stationary probability measure on $X=G/\Lambda$. By Theorem~\ref{thm;rigidity} we know that $\nu$ is homogeneous and $\Stab_G(\nu)^\circ$ is normalized by~$H$. By conjugating if necessary, we may assume the identity coset $\Lambda$ is in the support of~$\nu$.

If $\Stab_G(\nu)\cap H$ is non-discrete, then $\Stab_G(\nu)$ must contain a normal subgroup of $H$ of positive dimension. Since $\Lambda$ is irreducible, this implies that $\nu$ is $G$-invariant. Indeed, $\Stab_G(\nu)\Lambda$ is closed since the stabilizer intersects $\Lambda$ in a lattice (\cite[Theorem~1.13]{raghunathan}), and also dense by irreducibility of $\Lambda$ if $\Stab_G(\nu)$ contains a simple factor of $G$.

Let us now assume that $\Stab_G(\nu)\cap H$ is discrete and $H\neq G$ and use this to derive a contradiction. Since $\Stab_G(\nu)^\circ$ is normalized by $H$, we may view its Lie algebra as $H$-submodule of $\mathfrak g=\Lie(G)$. As every non-trivial $H$-isotypic component of $\mathfrak g$ is contained in $\Lie(H)$, it follows from the discreteness assumption that we must have $\Stab_G(\nu)^\circ\leqs C_G(H)\leqs C_G(\Gamma_\mu)$. 
This puts us in the setting of Proposition~\ref{prop;old_key}, namely, the homogeneous measure $\nu$ gives positive mass to an orbit of the centralizer $L$ of $\Gamma_\mu$ in $G$. We apply this proposition with $x=\Lambda$ and $L_0$ the connected component of $\Stab_G(\nu)\cap L$ and let $\Gamma_0$ and $\Gamma_0^L$ be as defined there. Then $L_0\cap\Lambda$ is central by irreducibility of $\Lambda$ (\cite[Corollary~5.21]{raghunathan}), hence finite, which by part~(iii) of the proposition implies that $L_0$ is compact.

We now invoke Margulis' arithmeticity theorem~\cite{margulis}. The conclusion is that we may assume that
\begin{align}\label{G-prod-rep}
    G=\prod_{\sigma\in S}\mathbf G^\sigma (k_\sigma),
\end{align}
where $\mathbf G$ is a Zariski connected absolutely simple linear algebraic group defined over a number field $k$, $k_\sigma\in\set{\R,\C}$ is the completion of $\sigma(k)$ for a field embedding $\sigma\colon k\to\C$, and $S$ is a finite set of inequivalent such embeddings with the property that $\mathbf G^\sigma(k_\sigma)$ is non-compact if and only if $\sigma$ or $\overline{\sigma}$ is in $S$. The lattice $\Lambda$ is given as the diagonal embedding of $\mathbf G(\mathcal{O}_k)$ in $G$ via $k\ni z\mapsto(\sigma(z))_{\sigma\in S}$, where $\mathcal{O}_k$  is the ring of integers of $k$. 
As $H\neq G$ is a connected normal subgroup of $G$ of positive dimension, there is a non-empty proper subset $S_1\subset S$ such that 
$H =\prod_{\sigma \in S_1} \mathbf G^\sigma (k_\sigma)$. Without loss of generality we assume that the identity embedding $\iota$ is contained in $S_1$. We also write $S_2=S\setminus S_1$, which is non-empty by construction. In this setup, $\Gamma_0^L$ is a dense subgroup of  $L_0$, which is a connected and compact subgroup 
contained in  $\prod_{\sigma\in S_2}\mathbf G^\sigma (k_\sigma)$. 

The following subtlety should be noted regarding Zariski topologies: The linear algebraic group $\mathbf{G}$  naturally carries the complex Zariski topology, defined by complex polynomials in the entries of the complex matrices in $\mathbf{G}$ (similarly for the Galois conjugates $\mathbf{G}^\sigma$). However, in the product representation \eqref{G-prod-rep} of $G$, the point of view is that of real algebraic groups. This means that when $k_\sigma=\C$, the group $\mathbf{G}^\sigma(\C)$ has to be seen as the group of real points of the restriction of scalars $\Res_{\C/\R} \mathbf G^\sigma$ with the real Zariski topology, defined by real polynomials in the real and imaginary parts of the entries of matrices in $\mathbf{G}^\sigma$. This gives rise to the real Zariski topology on $G$.

We also remark that in \eqref{G-prod-rep} and the associated product representation of $H$, strictly speaking, we should take the analytic identity components of the groups appearing as factors on the right-hand side. But we ignore this point for ease of notation and without loss of generality. 

\textsc{Relative compactness in Galois conjugates.} Recall that each $\gamma_0\in \Gamma_0$ preserves the homogeneous measure on $L_0x$, so there exists $l_0\in L_0$ such that $\gamma_0 x=l_0^{-1}x$, which implies 
that $\gamma_0 l_0 \in \Lambda$. 
Let $\Gamma_1$ be the projection of $\Gamma_0$ to $\mathbf G(k_\iota)$ (the factor corresponding to the identity embedding).
Then $\Gamma_0$ consists of $\prod_{\sigma\in S_1}\sigma(\gamma_1)$ and 
$\Gamma_0^L$ consists of $\prod_{\sigma\in S_2}\sigma(\gamma_1)$ for $\gamma_1\in \Gamma_1$.
So we have 
$\Gamma_1\leqs \mathbf{G}(\mathcal{O}_k)$ and for every $\sigma \in S_2$ the group $\sigma(\Gamma_1)$ (obtained by component-wise application of $\sigma$) has compact closure. The latter conclusion holds also for $\sigma\notin S$, since $\mathbf{G}^\sigma(k_\sigma)$ is compact in this case. So we conclude that $\sigma(\Gamma_1)$ is relatively compact for all embeddings $\sigma\notin S_1$.

\textsc{Zariski density properties of $\Gamma_*$.} From Proposition~\ref{prop;epimorphic} it follows that the Zariski closure of $\Gamma_\mu$ is an epimorphic subgroup of $H$ in the category of real algebraic groups (see Appendix~\ref{app;epi} for a discussion of the epimorphic property in different categories). As $\Gamma_0$ has finite index in $\Gamma_\mu$, also $\Zcl(\Gamma_0)$ is epimorphic in $H$. 
We claim that $\Zcl(\Gamma_0)$ is also reductive. Otherwise, its projection to one of the simple factors of $H$ is not reductive. Without loss of generality assume that this holds for the projection to $\mathbf{G}(k_\iota)$. This means that the Zariski closure of $\Gamma_1$ (in the real Zariski topology of $\mathbf{G}(k_\iota)$) has a non-trivial unipotent radical. Now consider the Zariski closure $\mathbf{F}$ of $\Gamma_1$ in the complex Zariski topology of $\mathbf{G}$. Since the real Zariski topology of $\mathbf{G}(\C)$ is finer than the complex Zariski topology of $\mathbf{G}$, also $\mathbf{F}$ has a non-trivial unipotent radical, i.e.\ is not reductive. Moreover, $\mathbf{F}$ is defined over $k$, since $\Gamma_1\leqs \mathbf{G}(\mathcal{O}_k)$. So the failure to be reductive carries over to the Galois conjugates of $\mathbf{F}$. Then we get a contradiction since for each $\sigma\in S_2$ the algebraic group $\mathbf{F}^\sigma$ is reductive, because it is the Zariski closure of the relatively compact group
$\sigma(\Gamma_1)$ (in the complex Zariski topology).
So we obtain that $\Zcl(\Gamma_0)$ is a reductive epimorphic subgroup of $H$, which can only happen if $\Gamma_0$ is Zariski dense in $H$. %Indeed, reductive subgroups are observable by~\cite[Corollary~2.4]{grosshans}
By projecting to the simple factors, we find that $\sigma(\Gamma_1)$ is Zariski dense in the real Zariski topology of $\mathbf{G}^\sigma(k_\sigma)$ for all $\sigma\in S_1$. In particular, this implies that $\Gamma_1$ is Zariski dense in $\mathbf{G}$ in the complex Zariski topology. This latter property can now be carried over to all Galois conjugates, showing that for every embedding $\sigma$, $\sigma(\Gamma_1)$ is Zariski dense in $\mathbf{G}^\sigma$ (in the complex Zariski topology). For every embedding $\sigma$, the Zariski closure of $\sigma(\Gamma_1)$ in the real Zariski topology must therefore be $\mathbf{G}^\sigma(\C)$ or a real form of it. In particular, for every embedding with $k_\sigma=\R$, $\sigma(\Gamma_1)$ is Zariski dense in $\mathbf{G}^\sigma(\R)$ in the real Zariski topology.

\textsc{A more natural field of definition.} Let $k'$ be the subfield of $k$ generated by the set $\operatorname{Tr}(\Ad(\Gamma_1))$. Then $\Ad(\Gamma_1)$ is definable over $k'$ (see \cite[IX.1.8]{margulis}). 
So we may and will assume that $\mathbf G$ is defined over $k'$ and $\Gamma_1\leqs \mathbf G(k')$. 
The group $\Res_{k'/\Q}\mathbf{G}(\R)=\prod_{\tau\colon k'\to\C}\mathbf G^\tau (k'_\tau)$ is naturally embedded in $\Res_{k/\Q}\mathbf{G}(\R)=\prod_\sigma\mathbf G^\sigma (k_\sigma)$ as a real algebraic subgroup, by identifying $\mathbf{G}^\tau (k'_\tau)$ with its diagonal embedding in $\prod_{\sigma\colon \sigma|_{k'}=\tau}\mathbf{G}^\sigma (k_\sigma)$. 
We deduce the following facts:     
\begin{enumerate}[label=(\alph*)]
   
    \item 
      We have $k\neq k'$. 
    Indeed, for $\sigma \in S_2$ the group $\mathbf{G}^\sigma(k_\sigma)$ is non-compact and  $\sigma(\Gamma_1)$ has compact closure (in the Lie group topology). In view of the Zariski density properties established in the paragraph above, and using that compact groups are closed in the real Zariski topology, it follows that we must have 
    $k_\sigma=\C$, and $\sigma(\Gamma_1)$ is contained in 
   a real form of $\mathbf G^\sigma(\C)$. %If g' is the Lie alg of the Z-closure in the real Z-topology, then g'+ig' is a complex lie subalgebra of g and invariant under Ad(something Z-dense), hence a Lie ideal, so it is all of g. now look at g' cap ig'. this is a Lie ideal by the same argument, so it is either 0 (=> g' is real form) or all of g (=> g' is already everything)
   The latter and the definition of $k'$ (via traces in the adjoint representation) imply $\sigma(k')\subset\R$, which would contradict $k_\sigma=\C$ if we had $k=k'$. 
 \item
    The embeddings $\sigma|_{k'}$ of $k'$ for $\sigma\in S_1$ are pairwise distinct,
    since $\Gamma_0$ is Zariski dense in $H$.
\end{enumerate}

\textsc{Combining everything to a contradiction.} 
In view of (a) above, the identity embedding $k'\to\C$ must admit a non-identity extension $\sigma\colon k\to \C$. This embedding $\sigma$ cannot be contained in $S_1$, since by (b) above, the elements of $S_1$ have pairwise distinct restrictions to $k'$. But  $\sigma\not\in S_1$ would imply $\sigma (\Gamma_1 )=\Gamma_1$ is relatively compact, which is impossible since $\Gamma_1$ is Zariski dense in the non-compact group $\mathbf{G}(k_\iota)$ in the real Zariski topology.
This contradiction finishes the proof of~(i).

In the case $H=G$ of part (ii), the arguments at the beginning of the proof show that either $\nu=m_X$ or $\Stab_G(\nu)$ is discrete. In the latter case, $\nu$ must be the uniform probability measure on a finite $\Gamma_\mu$-orbit (see~\cite[Lemma~8.3]{bq11}). Moreover, in this case we have that $C_G(\Gamma_\mu)$ is discrete by the epimorphic property of $\Gamma_\mu$ in $G$ from Proposition~\ref{prop;epimorphic}. Proposition~\ref{prop;countability} thus implies that there are only countably many distinct finite $\Gamma_\mu$-orbits in $X$. Hence, if $\nu$ is any non-atomic $\mu$-stationary probability measure on $X$, $\nu=m_X$ follows by considering an ergodic decomposition of $\nu$. This completes the proof.
\end{proof}

\subsection{Expansion on Grassmannians}\label{sec;exp_grass}
The $H$-expansion condition on $\mu$ is a universal requirement in the sense that all our results (including the measure classification theorem) hold for any embedding $H \hookrightarrow G$ and any discrete subgroup $\Lambda $ in $G$. Having fixed $H\leqs G$, however, close inspection of the proof of Theorem~\ref{thm;rigidity} reveals that it is sufficient to have uniform expansion on the quotient of each exterior power of $\mathfrak{g}$ by the corresponding $H$-fixed subspace.

\begin{de}\label{de;exp_grass}
Let $G$ be a real Lie group and $H\leqs G$ a connected semisimple subgroup with finite center. A probability measure $\mu$ on $H$ is said to be \emph{$H$-expanding relative to $G$} if $\mu$ is uniformly expanding on the quotient of $(\Ad^{\wedge k},\mathfrak{g}^{\wedge k})$ by the corresponding $H$-fixed subspace for every $1\le k\le \dim(G)-1$. 
\end{de}
We remark that a related notion was previously studied by the first two authors~\cite{prohaska-sert}.
\begin{thm}\label{thm.exp.grass}
Let $G$ be a real Lie group, $\Lambda\leqs G$ a discrete subgroup, and $H$ a connected semisimple subgroup of $G$ with finite center. Let $\mu$ be an $H$-expanding probability measure relative to $G$ with finite first moment.
Then the conclusions of Theorem~\textup{\ref{thm;rigidity}} hold for every ergodic $\mu$-stationary probability measure $\nu$ on $G/\Lambda$.
\end{thm}
\begin{proof}
We analyze the applications of the $H$-expansion property in the proof of Theorem~\ref{thm;rigidity}, so we retain the notation used there.
\begin{itemize}
\item The first application of Lemma~\ref{lem;modZ} is possible without problems.
\item Next, expansion is used for the representation $S^2(\mathfrak{g}^{\wedge \dim(N)})$. If $\dim(N)=\dim(G)$, then the probability measure $\eta$ in \eqref{nu_decomp} is finitely supported and $\Gamma_\mu$-invariant by \cite[Lemma~8.3]{bq11}, so all claims follow. Otherwise, the measure $\eta'$ on $S^2(\mathfrak{g}^{\wedge \dim(N)})$ is supported on $\set{v\otimes v\for v\in\mathfrak{g}^{\wedge\dim(N)}}$ by construction. Using that $\norm{v\otimes v}=\norm{v}^2$ and the assumed expansion in $\mathfrak{g}^{\wedge\dim(N)}$, we can again draw the desired conclusion that $\eta'$ is supported on the set of $H$-fixed vectors.
\item Finally, expansion is needed to reapply Theorem~\ref{thm;el} in the quotient by $N^\circ$. The assumption there implies that $H/(H\cap N^\circ)$ is still a semisimple group, so that $\dim(N)\le \dim(G)-3$. Let $v\in\mathfrak{g}^{\wedge\dim(N)}$ correspond to a basis of the Lie algebra $\mathfrak{n}$ of $N$. Then a norm on $\mathfrak{g}/\mathfrak{n}$ is given by $\norm{w+\mathfrak{n}}=\norm{w\wedge v}$ for $w\in\mathfrak{g}$. 
Since $H$ fixes the vector $\omega=v\otimes v$ in $S^2(\mathfrak{g}^{\wedge \dim(N)})$, $H$ acts on $v$ by $\pm 1$. As $H$ is connected, $v$ is fixed by $H$. Thus, for every $h\in H$ and $w\in\mathfrak{g}$ we have
\begin{align}\label{quotient_norm}
\norm{h\acts(w+\mathfrak{n})}=\norm{h\acts w\wedge v}=\norm{h\acts(w\wedge v)}.
\end{align}
Hence, we again obtain expansion for every vector in $\mathfrak{g}/\mathfrak{n}$ that is not $H$-fixed. This justifies the application of Lemma~\ref{lem;modZ} in the quotient.\qedhere
\end{itemize}
\end{proof}

Combining the above with some properties of epimorphic subgroups, we obtain the following.
\begin{cor}\label{cor;epi_class}
Let $G$ be a real algebraic group, $\Lambda<G$ a lattice, and $H\leqs G$ a Zariski connected semisimple algebraic subgroup without compact factors. Then any Zariski connected real algebraic epimorphic subgroup $F\leqs H$ supports probability measures $\mu$ for which the conclusions of Theorem~\textup{\ref{thm;rigidity}} hold.
\end{cor}

\begin{proof}
It is known that $F$ contains a split solvable algebraic subgroup $A'U$, where $A'$ is an algebraic $\R$-split torus and $U$ is unipotent and normalized by $A'$, that is still epimorphic in $H$ (see~\cite[\S10,~Theorem~2]{bien-borel1}). Thus we may assume $F=A'U$ is of this form to begin with. By \cite[Lemma~1]{weiss.epimorphic} there is a non-empty open cone $A'_+$ in $A'$ such that $\chi(a)>1$ for all $a\in A'_+$ and all characters of $A'$ having an eigenvector in one of the $U$-fixed subspaces $V_k^U$ of the finitely many representations $V_1,\dots,V_r$ appearing in the statement of Theorem~\ref{thm.exp.grass}. Then any probability measure $\mu$ on $F$ with finite first moment whose $A'$-average $a_{\operatorname{avg}}(\mu)$ lies in $A'_+$ and for which the Zariski closure of $\Gamma_\mu$ contains $U$ is uniformly expanding in all of the representations $V_k$. Indeed, this follows directly by combining Lemma~\ref{lem;useful} and \ref{lem;expansion_only_with_A}. Theorem~\ref{thm.exp.grass} thus applies to all measures $\mu$ satisfying these conditions.
\end{proof}

\section{Countability of homogeneous subspaces}\label{subsec;countability}
Let $\Gamma$ be a closed subsemigroup of $G$ and $\Lambda<G$ a lattice. A homogeneous subspace $Y\subset X=G/\Lambda$ is said to be $\Gamma$-invariant if $\Gamma$ preserves the associated homogeneous probability measure $\eta$ on $Y$. It is called $\Gamma$-ergodic if $\Gamma$ acts ergodically on $(Y,\eta)$. Define
\begin{align*}
\calS(\Gamma)=\set{\Gamma\text{-invariant }\Gamma\text{-ergodic homogeneous subspaces }Y\subset X}.
\end{align*}
A key input to the proof of Theorem~\ref{thm;orbit} is countability of $\calS(\Gamma_\mu)$ modulo the centralizer of $H$.
Our strategy to prove this result closely follows the approach in~\cite{bq132}, where this result is proved under the assumption that the Zariski closure of $\Ad(\Gamma_\mu)$ is semisimple and has no compact factors. The goal of this subsection is therefore to prove the following analogue of~\cite[Proposition~2.1]{bq132}. 

\begin{prop}\label{prop;countability}
Let $G$ be a real Lie group, $H\leqs G$ a connected semisimple subgroup with finite center, and $\Gamma<H$ a subsemigroup that supports a probability measure with finite first moment that is $H$-expanding relative to $G$. Denote by $L$ the centralizer of $\Gamma$ in $G$. Then there exists a countable subset $\mathcal{Y}$ of $\calS(\Gamma)$ such that
\begin{align}\label{eq.representatives}
\calS(\Gamma)=\set{lY\for l\in L,Y\in\mathcal{Y}}.
\end{align}
\end{prop}
Note that the set $\calS(\Gamma)$ remains the same if we replace the semigroup $\Gamma$ by the closed group that it generates. Therefore, in the proof of the previous result, we can suppose that $\Gamma$ is a closed subgroup of $H$.

The key ingredient of the proof of this proposition is Lemma~\ref{lem;T_countability} below, which will imply countability of the closed subgroups of $G$ that arise as the stabilizer of homogeneous subspaces in $\calS(\Gamma)$. To this end, we introduce the following definition, which, in view of Theorem~\ref{thm;rigidity}, is the appropriate replacement of~\cite[Definition~2.4]{bq132}.

\begin{de}\label{def;T}
Let $\Delta\subset\Sigma$ be discrete subgroups of a real Lie group $G$. The set $\calT(G,\Delta,\Sigma)$ is defined to be the set of closed subgroups $N$ of $G$ such that
\begin{enumerate}[label=(\roman*)]
\item $\Sigma$ is contained in $N$ and is a lattice in $N$,
\item $\Delta=\Sigma\cap N^\circ$, where $N^\circ$ is the connected component of $N$,
\item there exist a connected semisimple Lie group  $H_N\leqs G$ and a subgroup $\Gamma\leqs H_N \cap N$ which acts ergodically on $N/\Sigma$ and which supports an $H_N$-expanding probability measure relative to $G$. \end{enumerate}
\end{de}

\begin{lem}\label{lem;T_countability}
Let $G$ be a real Lie group and $\Delta\subset \Sigma$ finitely generated discrete subgroups of $G$. Then the set $\calT(G,\Delta,\Sigma)$ is countable.
\end{lem}

The proof of this lemma requires the following strengthening of~\cite[Lemma~2.6]{bq132}. 
\begin{lem}\label{lem;normalization}
Let $G$ be a real Lie group, $\mathfrak{g}$ its Lie algebra, and $\Delta\subset \Sigma$ discrete subgroups of $G$. Let $N$ belong to $\calT(G,\Delta,\Sigma)$, $H_N$ be any connected semisimple subgroup of $G$ as in \textup{(iii)} of Definition~\textup{\ref{def;T}}, and let $M$ be a unimodular Lie subgroup of $G$ containing $\Sigma$. Let $\omega\in S^2(\mathfrak{g}^{\wedge \dim(M)})$ correspond to a basis of the Lie algebra of $M$. Then $\omega$ is fixed by $N$ and $H_N$, and hence $M^\circ$ is normalized by $N$ and $H_N$. In particular, this holds whenever $M\in\calT(G,\Delta,\Sigma)$.
\end{lem}

In the statement above, $S^2(\mathfrak{g}^{\wedge \dim(M)})$ denotes the symmetric square of $\mathfrak{g}^{\wedge \dim(M)}$. If $v\in\mathfrak{g}^{\wedge \dim(M)}$ corresponds to a basis of the Lie algebra of $N$, the appearing vector $\omega$ is given by $\omega=v\otimes v$.
\begin{proof}
If $\dim(M)=\dim(G)$, then $M^\circ=G^\circ$ and the statement is clear. So we assume that $\dim(M)<\dim(G)$.
Since $M$ is unimodular and contains $\Sigma$, the action of $\Sigma$ fixes $\omega$. Therefore, the map
\begin{align*}
N\to S^2(\mathfrak{g}^{\wedge \dim(M)}),\,h\mapsto h\acts\omega
\end{align*}
descends to a map $N/\Sigma\to S^2(\mathfrak{g}^{\wedge \dim(M)})$.
Denote by $\eta$ the pushforward of the Haar probability measure on $N/\Sigma$ to $S^2(\mathfrak{g}^{\wedge \dim(M)})$ by this map and let  $\Gamma\leqs N \cap H_N$ be as in (iii) of the definition of $\calT(G,\Delta,\Sigma)$. Then $\eta$ is an ergodic $\Gamma$-invariant probability measure supported on the set $\set{v\otimes v\for v\in\mathfrak{g}^{\wedge\dim(M)}}$. Since $\Gamma$ supports an $H_N$-expanding probability measure relative to $G$ and $\norm{v\otimes v}=\norm{v}^2$, Lemma~\ref{lem;stationary} implies that 
$\eta$ is concentrated on the subspace of $H_N$-fixed vectors.
The ergodicity forces $\eta$ to be the Dirac mass at $\omega$. Hence, $\omega$ is $N$- and $H_N$-fixed, as required. 
\end{proof}

We can now prove Lemma~\ref{lem;T_countability}. The argument is basically the same as in the proof of~\cite[Lemma~2.5]{bq132}, but we need to handle an additional difficulty coming from the fact that $\Gamma$ is not necessarily Zariski dense in $H_N$, but only carries a probability measure that is $H_N$-expanding relative to $G$.

\begin{proof}[Proof of Lemma~\textup{\ref{lem;T_countability}}]
For every $N \in \calT(G,\Delta,\Sigma)$, we fix a connected semisimple group $H_N$ as in (iii) of Definition~\ref{def;T}. Considering the closure of the group generated by the set $\bigcup_{N\in\calT(G,\Delta,\Sigma)}H_N N,$ we can assume that this set generates a dense subgroup of $G$. By Lemma~\ref{lem;normalization}, 
\begin{align*}
 M\df\bigcap_{N\in\calT(G,\Delta,\Sigma)}N^\circ%=\Bigl(\bigcap_{N\in\calT(G,\Delta,\Sigma)}N\Bigr)^\circ
\end{align*}is a normal subgroup of $G$.
Let $\pi\colon G\to G/M$ be the natural projection map.

We will argue next that $\iota\colon N\mapsto \pi(N)$ gives an injection of $\calT(G,\Delta,\Sigma)\setminus \set{\Sigma M}$ into $\calT(G/M,\set{e},\pi(\Sigma))$. First, note that $N\mapsto\pi(N)$ is an injective map from $\calT(G,\Delta,\Sigma)$ into the set of closed subgroups of $G/M$. Since $\Sigma \cap M = \Delta$ is a lattice in $M$, $\Sigma M$ is closed in $G$ by~\cite[Theorem~1.13]{raghunathan}, which implies that $\pi(\Sigma)$ is discrete. As there is an equivariant projection $N/\Sigma\to \pi(N)/\pi(\Sigma)$, $\pi(\Sigma) $ is a lattice in $\pi(N)$. If $\pi(n)\in\pi(\Sigma)$ for some $n\in N^\circ$, then $n=\sigma m$ for some $m\in M$ and $\sigma\in\Sigma$. Since $M\subset N^\circ$, it follows that $\sigma\in\Sigma\cap N^\circ=\Delta\subset M$, which proves that $\pi(N)^\circ\cap\pi(\Sigma)=\set{e}$ is the trivial group. So we have verified conditions (i) and (ii) of Definition~\ref{def;T} for any element $\pi(N)$ in the image of $\iota$.
To also verify condition~(iii), let $H_N\leqs G$ be the connected semisimple subgroup from condition~(iii) for $N$ and $\Gamma$ a subgroup of $H_N\cap N$ that acts ergodically on $N/\Sigma$ and carries an $H_N$-expanding probability measure $\mu$ relative to $G$.
Then it is clear that $\pi(\Gamma)$ acts ergodically on $\pi(N)/\pi(\Sigma)$.
Now, if $H_N\leqs M$, then ergodicity of this action forces $N=\Sigma M$. 
Otherwise, $\pi(H_N) $ is a connected semisimple Lie group. By Lemma~\ref{lem;normalization} and connectedness, $H_N$ fixes a vector $v\in\mathfrak{g}^{\wedge\dim(M)}$ corresponding to a basis of the Lie algebra $\mathfrak{m}$ of $M$. For $1\le k\le \dim(G/M)-1$, on $(\mathfrak{g}/\mathfrak{m})^{\wedge k}$ we may use a norm with the property that $\norm{[w]}=\norm{w\wedge v}$ for every $w\in\mathfrak{g}^{\wedge k}$, where $[w]$ denotes the projection of $w$ to $(\mathfrak{g}/\mathfrak{m})^{\wedge k}$. Then the same calculation as in \eqref{quotient_norm} shows that $\pi_*\mu$ is $\pi(H_N)$-expanding relative to $G/M$. So also condition~(iii) of Definition~\ref{def;T} holds for $\pi(N)$.
 
Therefore, it suffices to prove the lemma under the assumption that $\Delta=\set{e}$ is the trivial group and that for every $N \in {\calT(G,\set{e},\Sigma)}$, the connected component $N^\circ$ is normal in $G$. % for the normality, we can repeat the reduction step at the very beginning of the proof
In view of condition~(ii), this implies that $N^\circ$ is a compact normal subgroup of $G$. By~\cite[Lemma~2.7]{bq132}, there are only countably many such $N^\circ$. 
Similar to the first reduction step above, after fixing $N^\circ$ and replacing $G$ by $G/N^\circ$
and $\Sigma$ by $\Sigma N^\circ /N^\circ$,
we are left to show that the set 
$\mathcal V(G, \Sigma)$ of discrete subgroups $N$ containing $\Sigma$ as a finite index subgroup such that (iii) of Definition~\ref{def;T} holds is countable. 
For each $N \in \mathcal V(G, \Sigma)$, there is a finite index subgroup $\Sigma'\leqs \Sigma$ such that $\Sigma'$ is normal in $N$. %by considering the kernel of the permutation action of $N$ on $N/\Sigma$
Recall that by assumption $\Sigma $ is finitely generated, so that it admits only finitely many homomorphisms to any fixed finite group.  It follows that there are countably many such $\Sigma'$.
Therefore, it suffices to show that, {for $\Sigma'$ fixed,} the set $\mathcal V(G, \Sigma',\Sigma)$ of 
$N\in \mathcal V(G, \Sigma)$ with $\Sigma'$ normal in $N$ is countable. Let $S$ be the closed subgroup generated by $\bigcup_{N \in \mathcal  V(G, \Sigma',\Sigma)} N$. Then $\Sigma'$ is a discrete  normal subgroup of $S$. 
For any $g\in \Sigma'$, the set $\set{s g s^{-1}\for s\in S^\circ }$ is a connected subset of $\Sigma'$, so it has to be $\set{ g}$. It follows that  $\Sigma'$ centralizes $S^\circ$.   
Given $N\in  \mathcal V(G, \Sigma',\Sigma)$, let $\Gamma$ be a subgroup of $H_N \cap N$ acting ergodically on $N/\Sigma$ as in (iii) of Definition~\ref{def;T}. By ergodicity, we have $N=\Gamma \Sigma$ and since $\Gamma \Sigma=\Gamma (\Sigma' \Sigma)=(\Gamma \Sigma')\Sigma $, $N$ is uniquely determined by the discrete group $\Gamma\Sigma'$.
So it suffices to show that the set of subgroups 
$\Gamma\Sigma'$ appearing in this way is countable. 
The finite index subgroup  $\Gamma \cap \Sigma'$ of $\Gamma$ centralizes $S^\circ$ and $\Gamma$ normalizes $S^\circ$. %because $\Gamma\leqs S$ and $S^\circ$ is normal in $S$
It follows that the conjugation action of $\Gamma$ on $S^\circ$ factors through a finite group. Now, according to (iii) of Definition~\ref{def;T}, there exists a probability measure on $\Gamma$ that is $H_N$-expanding relative to $G$. By (i) of Proposition~\ref{prop;facts} applied to the adjoint representation of $H_N$ on $\mathfrak{g}$, we conclude that every element of the Lie algebra of $S$ is fixed by $H_N$. %Notice that here we are using a consequence of supporting an $H$-expanding measure, that is stronger than mere epimorphicity. This step replaces the part in BQ's proof where they say that ``$\Ad \Gamma$ is finite and hence trivial''.
This implies that $\Gamma<H_N$ centralizes $S^\circ$. Therefore $\Gamma\Sigma'/\Sigma'$
is a finite subgroup of  $S/\Sigma'$ centralizing  $S^\circ \Sigma'/\Sigma'$. 
By~\cite[Lemma~2.8]{bq132}, the set of compact subgroups of $S/\Sigma'$ centralizing $S^\circ \Sigma'/\Sigma'$ is countable. This gives the required countability and 
hence completes the proof. 
\end{proof}

We also need the following version of~\cite[Lemma~2.2]{bq132}.
\begin{lem}\label{lem;countable_fixed_point_orbits}
Let $G$ be a real Lie group, $H$ a connected semisimple subgroup of $G$, and $\Gamma$ a subgroup of $H$ that supports an $H$-expanding probability measure relative to $G$. Moreover, let $L$ be the centralizer of $\Gamma$ in $G$ and $N$ a closed unimodular subgroup of $G$. Then the set of $\Gamma$-fixed points in $Y=G/N$ is a countable union of $L$-orbits.
\end{lem}
\begin{proof}
It is enough to consider the case $\dim(N)<\dim(G)$. Denote by $Y^\Gamma$ the set of $\Gamma$-fixed points in $Y$. Then it suffices to show that every $L$-orbit $Ly$ in $Y^\Gamma$ is open in $Y^\Gamma$. After a conjugation we may assume $y=eN$ is the identity coset. In particular, we then have $\Gamma\leqs N$. Let $\mathfrak l$ denote the Lie algebra of $L$. By finite-dimensionality, we can find $\gamma_1,\dots,\gamma_r\in\Gamma$ such that 
\begin{align*}
\mathfrak l=\set{v\in\mathfrak g\for \Ad(\gamma_i)v=v\text{ for }1\le i\le r}.
\end{align*}
In view of unimodularity of $N$, considering a vector in $S^2(\mathfrak{g}^{\wedge\dim(N)})$ corresponding to a basis of the Lie algebra $\mathfrak n$ of $N$ and arguing as in Lemma~\ref{lem;normalization} yields that $\mathfrak n$ is $H$-invariant. Thanks to the expansion in the adjoint representation, it moreover follows that $\mathfrak{l}$ coincides with the space of $H$-fixed vectors in $\mathfrak g$. We choose an $H$-invariant complement $\mathfrak{v}$ of $\mathfrak n+\mathfrak l$ in $\mathfrak{g}$. Then for any $v\in \mathfrak{v}$ sufficiently small, if  $\exp (v) y$ is $\Gamma$-fixed,  then for all $1\le i\le r$ we have
\begin{align*}
\exp (\Ad (\gamma_i) v)y=\gamma_i \exp (v) y= \exp (v)y,
\end{align*}
which implies $\Ad(\gamma_i)v=v$ and thus $v\in\mathfrak l\cap\mathfrak v=\set{0}$.  This shows that $Ly$ is open in $Y^{\Gamma}$ and hence finishes the proof that $Y^\Gamma$ is a countable union of $L$-orbits.
\end{proof}

Finally, we can prove the main result of this subsection. We adapt the proof of~\cite[Proposition~2.1]{bq132} by
substituting Lemmas~\ref{lem;T_countability} and~\ref{lem;countable_fixed_point_orbits} for the corresponding results, and extend it to cover semigroups that are not compactly generated.
\begin{proof}[Proof of Proposition~\textup{\ref{prop;countability}}] 
We first establish the statement assuming additionally that $\Gamma$ is compactly generated. Let $Y\in\calS(\Gamma)$ and denote by $G_Y$ the stabilizer of the homogeneous probability measure $\nu$ corresponding to $Y$. Let $\mu$ be a probability measure on $\Gamma$ that is $H$-expanding relative to $G$. Choose $g\in G$ such that $g\Lambda\in Y$ and consider $N=g^{-1}\Gamma G_Y^\circ g$, which is a closed subgroup of $G$ because $\Gamma$ is contained in $G_Y$ and thus normalizes $G_Y^\circ$. Now, the discrete groups $\Delta=N^\circ\cap \Lambda$ and $\Sigma=N\cap\Lambda$ are lattices in $N^\circ$ and $N$, respectively. Being a lattice in a connected Lie group, $\Delta$ is finitely generated (see~\cite[6.18]{raghunathan}). As $N=g^{-1}\Gamma G_Y^\circ g$ and $\Gamma$ is compactly generated, $N/N^\circ$ is finitely generated. Since $\Sigma/\Delta$ has finite index in $N/N^\circ$, also $\Sigma$ is finitely generated.
As $\Lambda$ admits only countably many finitely generated subgroups, one may assume that $\Delta$ and $\Sigma$ are fixed. We claim that $N$ belongs to $\calT(G,\Delta,\Sigma)$. To see this, we first record that (i) and (ii) in Definition~\ref{def;T} are immediate.  Considering $H_N=g^{-1}Hg$, its subgroup $g^{-1}\Gamma g$ and the image of $\mu$ by conjugation by $g^{-1}$, also (iii) is seen to hold. 
Consequently, we can also assume $N$ to be fixed by virtue of Lemma~\ref{lem;T_countability}. As the point $gN\in G/N$ is $\Gamma$-invariant, by Lemma~\ref{lem;countable_fixed_point_orbits} one may further assume the $L$-orbit $LgN\subset G/N$ is fixed. It only remains to note that for $l\in L$, the orbit $lgN\Lambda\subset X=G/\Lambda$ is precisely the translate $lY$ of $Y$.

To treat the general case without the compact generation assumption, given an arbitrary probability measure $\mu'$ on $\Gamma$ with finite first moment that is $H$-expanding relative to $G$, we consider the probability measure $\mu$ given as the normalized restriction of $\mu'$ to a sufficiently large compact ball $B$ around the identity. By choosing $B$ large enough, we can guarantee that the integral characterization of uniform expansion from Proposition~\ref{prop;expansion_char} still holds for the finite collection of representations in Definition~\ref{de;exp_grass}. In view of expansion in the adjoint representation, the connected components of $L=C_G(\Gamma)$ and $L_\mu=C_G(\Gamma_\mu)$ coincide. Therefore, applying the above to the compactly generated subgroup $\Gamma_\mu$, we can find a countable collection $\mathcal{Y}_\mu\subset\calS(\Gamma_\mu)$ such that $\calS(\Gamma_\mu)=\set{lY_\mu\for l\in L,Y_\mu\in\mathcal{Y}_\mu}$. We claim that $\mathcal{Y}=\set{\overline{\Gamma Y_\mu}\for Y_\mu\in\mathcal{Y}_\mu}\cap\calS(\Gamma)$ satisfies the conclusion of the proposition.
To see this, let $Y\in\calS(\Gamma)$ be arbitrary and $\nu_Y$ be the associated $\Gamma$-invariant $\Gamma$-ergodic homogeneous measure. By Theorem~\ref{thm.exp.grass} we know that every $\Gamma_\mu$-ergodic component of $\nu_Y$ is an element of $\calS(\Gamma_\mu)$. By Fubini's theorem and $\Gamma$-ergodicity of $\nu_Y$, we can thus find $Y_\mu'\in\calS(\Gamma_\mu)$ such that almost every  point $x\in Y_\mu'$ with respect to the homogeneous measure on $Y_\mu'$ satisfies $Y=\overline{\Gamma x}$. We also know that $Y_\mu'=lY_\mu$ for some $Y_\mu\in\mathcal{Y}_\mu$ and $l\in L=C_G(\Gamma)$. We conclude that $Y=\overline{\Gamma Y_\mu'}=l\overline{\Gamma Y_\mu}$, which shows that $\overline{\Gamma Y_\mu}\in\mathcal{Y}$ and completes the proof.
\end{proof}

\section{Height functions with contraction properties}\label{sec;height_functions}
A Markov chain on a standard Borel space $X$ is a measurable map $X\ni x\mapsto P_x$ from $X$ to the space of Borel probability measures on $X$, specifying the transition probabilities at each point of $X$. The associated Markov operator $P$ is defined by 
\begin{align*}
P(f)(x)=\int_Xf\dd P_x
\end{align*}
for a non-negative Borel function $f$ on $X$ and $x\in X$. If $G$ is a locally compact second countable group with a Borel $G$-action on $X$, then a choice of a probability measure $\mu$ on $G$ induces a Markov chain on $X$ with transition probabilities $P_x=\mu*\delta_x$, which can be thought of as the formalization of the concept of the random walk on $X$ given by $\mu$. We denote the associated Markov operator by $A_\mu$, which is given in this context by the explicit formula
\begin{align*}
A_\mu(f)(x)=\int_Gf(gx)\dd\mu(g).
\end{align*}
We also refer to $A_\mu$ as the \emph{averaging operator} associated to $\mu$. See \cite[\S3]{bq132} and \cite[\S2]{bq} for more background on Markov operators in the context of the study of random walks.

Coming back to our setting, recall that $\Lambda$ denotes a lattice in a Lie group $G$ and $H$ a connected semisimple subgroup of $G$ without compact factors and with finite center, and $\mu$ is an $H$-expanding probability measure on $H$. 

The goal of this section is to construct height functions on $X=G/\Lambda$ that are contracted by the averaging operator $A_\mu$ (also known as \emph{Lyapunov functions} or sometimes \emph{Margulis functions}), which will yield the recurrence properties of the random walk on $X$ necessary for the proof of our main theorems. As already explained in \S\ref{sec;rec_intro}, two types of height functions are required. First, one needs a height function that is proper but stays bounded on prescribed compact subsets of the space $X$, which prevents the random walk from escaping to infinity. Secondly, in order to ensure equidistribution towards a homogeneous measure sitting on the orbit closure, we will need to construct height functions which are unbounded near lower dimensional homogeneous subspaces. These ensure that the random walk does not accumulate near such ``singular subspaces'', i.e.\ does not spend too much time in their vicinity.

\subsection{Height function with respect to the cusps}\label{subsec;height}
We first present the construction of the height functions responsible for ruling out escape of mass.
\begin{thm}[Exponential $\mu$-unstability of the cusps,~\cite{bq12}]\label{thm;hgeneral}
Let $\mu$ be an $H$-expanding probability measure with finite exponential moments. 
For any compact subset $Z$ of $X=G/\Lambda$, there exist constants $m\in \N$, $a \in (0,1)$, $b>0$, and a lower semicontinuous function $\beta_\infty\colon X\to [1, \infty]$ uniformly bounded on $Z$ such that for every $x \in X$ we have
\begin{align}\label{contr}
A_\mu^m(\beta_\infty) (x)\le a \beta_\infty(x)+b.
\end{align}
Moreover, 
\begin{enumerate}[label=\textup{(\roman*)}]
\item for every $\ell>1$, the set $\beta_{\infty}^{-1}([1,\ell])$ is compact,
\item the set $\beta_\infty^{-1}(\set{\infty})$ is $H$-invariant, and
\item there exists a constant $\kappa>0$ such that for every $h \in H$ and $x \in G/\Lambda$ we have $\beta_\infty(hx)\le \operatorname{N}(\Ad h)^\kappa \beta_\infty (x)$.
\end{enumerate}
\end{thm}

By slight abuse of terminology, we sometimes just say that a height function is ``proper'' when referring to property~(i) above. %although with the usual topology on the compact $[0,\infty]$ the function is not proper.

Let $\mathfrak{g}$ be the Lie algebra of $G$, $\mathfrak{r}$ the largest amenable ideal of $\mathfrak{g}$ and $\mathfrak{s}=\mathfrak{g}/\mathfrak{r}$. A Lyapunov function as in the above theorem is constructed in~\cite{bq12} in the case the non-compact part of the Zariski closure of the group generated by the support of the probability measure $(\Ad_\mathfrak{s})_*\mu$ is semisimple. However, as it turns out, 
this Zariski density assumption in a semisimple group without compact factors is only critically used, via Furstenberg's result of positivity of the top Lyapunov exponent, to ensure \eqref{eq;azero} below, which
is also guaranteed by our dynamical $H$-expansion assumption. Therefore, Benoist--Quint's proof goes through in our setting with minor adaptations. We now explain this in more detail.

A version of the following elementary but key lemma was already used in~\cite{em} (see also~\cite[Lemma~6.12]{bq13}). In our case, it holds true thanks to the characterization of uniform expansion expressed in Proposition~\ref{prop;expansion_char}. 

\begin{lem}
\label{lem;contract}
Let $\mu$ be an $H$-expanding probability measure on $H$ with finite exponential moments and $(\rho,V)$ be a representation of $H$ without nonzero $H$-fixed vectors. Then there exists $\delta_0>0$ such that for every $\delta \in (0,\delta_0)$ and $c \in (0,1)$, for every $n \in \N$ large enough, we have
\begin{align}\label{eq;azero}
     \int_ H \frac{1}{\norm{h\acts v}^\delta} \dd \mu ^{*n}(h)\le  
     \frac{c}{\norm{v}^\delta}
\end{align}
for every $v\in V\setminus \set{0}$.
\end{lem}
\begin{proof}
Using the elementary fact that for every $\varepsilon \in (0,1)$, $x \in (0,\varepsilon)$ and $a>0$, we have $a^x=1 +x \log a + (\frac{x}{\varepsilon})^2 R_a(x)$ with $\abs{R_a(x)} \le e^{\varepsilon\abs{\log a}}$  together with $\abs{\log \frac{\norm{v}}{\norm{gv}}} \le \log \operatorname{N}(g)$ for every $g \in \GL(V)$, we see that for every $n \in \N$, $\varepsilon \in (0,1)$ and $\delta \in (0,\varepsilon)$ 
\begin{align}\label{eq.eskin-margulis}
\int_H \frac{\norm{v}^\delta}{\norm{h\acts v}^\delta} \dd\mu^{*n}(h) \le 1 + \delta \int_H \log\frac{\norm{v}}{\norm{h\acts v}} \dd\mu^{*n}(h)+ \biggl(\frac{\delta}{\varepsilon}\biggr)^2 \int_H  \operatorname{N}(\rho(h))^\varepsilon \dd\mu^{*n}(h).
\end{align}
By Proposition~\ref{prop;expansion_char}, there exists $N\in\N $ and $C>0$ such that for all $v\in V\setminus \set{0}$, we have 
\begin{align}\label{eq.vector.expand}
\int_H \log \frac{\norm{v}}{\norm{h\acts v}} \dd \mu^{*N} (h)\le -C.
\end{align}
Since $\rho_*\mu$ has finite exponential moments by Lemma~\ref{lem;first_moment}, we can choose $\varepsilon_0>0$ such that $\int_H \operatorname{N}(\rho(h))^{\varepsilon_0} \dd\mu^{*n}(h)<\infty$ for every $n\in\N$. Now applying \eqref{eq.eskin-margulis}
with $n=N$, $\varepsilon=\varepsilon_0>0$ and using \eqref{eq.vector.expand}, we get that for every $\delta>0$ smaller than some $\delta_0>0$, there exists $c' \in(0,1)$ such that we have 
\begin{align}\label{eq.iterate}
\int_H \frac{1}{\norm{h\acts v}^\delta} \dd \mu ^{*N}(h)\le  
     \frac{c'}{\norm{v}^\delta}
\end{align}
for every $v\in V \setminus \set{0}$. Writing an arbitrary $n \in \N$ as $n=m N + k$ with $m,k \in \N$ and $k<N$, using the facts that $\mu^{*n}=\mu^{* m N} * \mu^{* k}$,  $\frac{1}{\norm{h\acts v}}\le \operatorname{N}(\rho(h)) \frac{1}{\norm{v}}$ and the existence of finite exponential moments, iterating \eqref{eq.iterate} now yields
\begin{align*}
\int_H \frac{1}{\norm{h\acts v}^\delta} \dd \mu ^{*n}(h)\le  \frac{(c')^m}{\norm{v}^\delta}\biggl(\int_H\operatorname{N}(\rho(h))^\delta\dd\mu(h)\biggr)^k,
\end{align*}
the right-hand side of which can be made to be smaller than $c/\norm{v}^\delta$ by requiring $m$ to be large enough.
\end{proof}

\begin{proof}[Proof of Theorem~\textup{\ref{thm;hgeneral}}]
We start the proof with a few general remarks on Lyapunov functions and their construction.
\begin{enumerate}
    \item It suffices to construct the function $\beta_\infty$ with values in $[0,\infty]$. Indeed, in the end one can simply add $1$, if necessary, to ensure values in $[1,\infty]$. 
    \item The conclusion of the theorem is unaffected when replacing $\Lambda$ by a commensurable lattice $\Lambda'$, that is, a lattice such that the intersection $\Lambda_0=\Lambda\cap\Lambda'$ has finite index in both $\Lambda$ and $\Lambda'$. Indeed, given a Lyapunov function $G/\Lambda\to[0,\infty]$, one can just precompose it with the projection $G/\Lambda_0\to G/\Lambda$, and, conversely, starting with a function $\beta\colon G/\Lambda_0\to[0,\infty]$, one can define the function $\beta_\infty$ on $G/\Lambda$ by setting
\begin{align*}
\beta_\infty(g\Lambda)=\sum_{\lambda\in\Lambda/\Lambda_0}\beta(g\lambda\Lambda_0)
\end{align*}
for $g\in G$, which is easily seen to have the correct properties.
\item We may always assume that the lattice $\Lambda$ is non-uniform, i.e.\ that $X=G/\Gamma$ is non-compact. For on a compact quotient, the constant function $1$ already has all required properties.
\item In the construction, we may without loss of generality replace $G$ by any open subgroup $G_0$. Indeed, $X$ is the disjoint union of $G_0$-orbits, and these are $\Gamma_\mu$-invariant since $H$ is connected. Thus, one can translate a function $\beta_\infty$ on $G_0/(G_0\cap \Lambda)$ to other 
$G_0$-orbits.
\end{enumerate}

From now on, we assume $G$ is connected and prove the existence of the height function $\beta_\infty$ with the required properties. 
The proof  proceeds in several steps.

\textsc{Case 1: $G=\SL_d(\R)$ and $X=\SL_d(\R)/\SL_d(\Z)$.} 
We show that the Benoist--Quint height function in~\cite{bq12} has the required properties. We endow $E=\R^d $ with a Euclidean structure invariant by some maximal compact subgroup of $H$. We endow the vector space $\bigwedge^* E=\bigoplus_{i=0}^d  \bigwedge^i E$ with the induced Euclidean structure and use $\norm{\cdot}$ to denote the corresponding norm on $E$ and on $\bigwedge^* E$.
For $0\le i\le d$, we fix constants $\delta_i=(d-i)i$; they satisfy 
\begin{align}\label{eq.unused1}
\delta_{r+s}+\delta_{r+t}\ge \delta _r +\delta_{r+s+t}+1    
\end{align}
for every $r,s,t \in \N$ with $s>0$ and $t>0$.

We fix a maximal split torus $A$ of $H$. Let $\mathfrak a$ and $\mathfrak h$ be the Lie algebras of $A$
and $H$, respectively. 
 Let  $\Sigma(\mathfrak h, \mathfrak a)$ be the set of restricted roots. We fix a positive system in $\Sigma(\mathfrak h, \mathfrak a)$. Let 
 $\mathcal W \subset \mathfrak a^*$ be the set of restricted weights appearing in finite-dimensional representations of $H$. 
We define a partial order on $\mathcal W$  by 
\begin{align}\label{eq.unused2}
\lambda\le \eta  \iff \eta -\lambda \text{ is a sum of positive roots}. 
\end{align}

Recall that any representation of a connected semisimple real Lie group is completely reducible and each irreducible representation has a unique highest weight. We denote by $\mathcal{W}^+\subset \mathcal{W}$ the set of highest weights and let $\mathcal S\subset \mathcal W^+$ be the set of nonzero highest weights corresponding to the non-trivial irreducible representations of $H$ appearing as direct summands in $\bigwedge ^* E$, where the representation of $H$ on $E$ is just the restriction of the standard representation of $G$.
So the action of $H$ on 
$\bigwedge ^*E$ decomposes into
a direct sum 
\begin{align*}
\bigwedge\nolimits^*  E= E_*^H\oplus\bigoplus_{\lambda\in \mathcal S} E_*^\lambda,
\end{align*}
where $E_*^H$ is the space of $H$-fixed vectors in $\bigwedge^* E$  and $E_*^\lambda$ is the sum of all the irreducible subspaces of $\bigwedge^* E$ with highest weight $\lambda$ (i.e.\ the isotypic component of $\lambda$). We fix $s_0\in \mathfrak a$ in the interior of the positive Weyl chamber and define $\delta_\lambda= \lambda(s_0)$ for $\lambda\in\mathcal{W}^+$, so that the $\delta_\lambda$ satisfy $\lambda \le \mu$ if and only if $\delta_\lambda \le \delta_\mu$ and $\delta_\lambda=0$ if and only if $\lambda=0$ for all $\lambda,\mu \in \mathcal{W}^+$. For $\lambda \in \mathcal{S}$, we use $q_\lambda$ (resp.~$q_0$) to denote the $H$-equivariant projection from $\bigwedge^* E$ to $E_*^\lambda$ (resp.~$E_*^H$). 
For any $\varepsilon>0$ and $v\in \bigwedge^i E$ with $0<i<d $, define 
\begin{align*}
\varphi_\varepsilon(v)=\begin{cases}
\min_{\lambda\in \mathcal  S} \varepsilon ^{\frac{\delta_i}{\delta_\lambda}} \norm{ q_\lambda(v)}^{-\frac{1}{\delta_\lambda}},  & \text{if } \norm{q_0(v)}< \varepsilon^{\delta_i}, \\
\hfill 0, & \text{otherwise}, 
\end{cases}
\end{align*}
with the convention $\min \emptyset =\infty$.
Using Lemma~\ref{lem;contract} and $H$-equivariance of the projections $q_\lambda$, one readily observes (cf.~\cite[Lemma~4.3]{bq12}) that for every $\delta>0$ small enough, for every $c\in (0,1)$, there exists $n \in \N$ such that for every $i=1,\dots,d$ and $v \in \bigwedge^i E$ it holds that
\begin{align}\label{eq.bqlemma4.3}
A_\mu^n\varphi_\varepsilon^\delta(v) \le a \varphi_\varepsilon^\delta(v)
\end{align}
for every $\varepsilon>0$. For every $\varepsilon>0$, let the function $\beta_{\varepsilon,\infty}$ on $G/\Lambda$ be defined by 
\begin{align*}
\beta_{\varepsilon,\infty}(x)= \max \varphi_\varepsilon(v),
\end{align*}
where, writing $x=g\Lambda$, the maximum is taken over all $0<i<d$ and nonzero $v\in \bigwedge^i E$ such that $v=v_1 \wedge \dots \wedge v_i$ with $v_j \in \Lambda_x\df g\Z^d$ for $j=1,\dots,i$  (following~\cite{bq12}, such pure wedge products $v$ will be called ``$x$-integral monomials''). % in fact it suffices to consider primitive ones, and also the sup is a max

Note that by construction we have $\beta_{\varepsilon,\infty}(x)=\infty$ if and only if there exists a nonzero $H$-fixed $x$-integral monomial $v\in\bigwedge^i E$ whose norm 
is less than $\varepsilon^{\delta _i}$. 
Therefore, the set $\beta_{\varepsilon,\infty}^{-1}(\set{\infty})$ is  $H$-invariant. Moreover, for every $\varepsilon>0$, the function $\beta_{\varepsilon,\infty}$ is proper and lower semicontinuous (see~\cite[Remark~5.2]{bq12}). Setting $\kappa'=\max_{\lambda \in \mathcal{S}}\delta_\lambda^{-1}$, it is also readily verified  that for every $h \in H$ we have $\beta_{\varepsilon,\infty}(hx) \le \operatorname{N}(h)^{d\kappa'}\beta_{\varepsilon,\infty}(x)$.

Now it follows precisely in the same way as in~\cite[Proposition~5.3]{bq12}, by simply replacing ~\cite[Lemma~4.3]{bq12} by \eqref{eq.bqlemma4.3}, that for every $\delta>0$ and $\varepsilon>0$ small enough, there exist $n \in \N$, $a \in (0,1)$ and $b>0$ such that
\begin{align*}
A_\mu^n \beta_{\varepsilon,\infty}^\delta \le a \beta_{\varepsilon,\infty}^\delta + b. 
\end{align*}
For brevity and to avoid mere repetition, we will not reproduce this part of the proof here. We note however that this passage is the part where the crucial ``Mother inequality''~\cite[\S 3]{bq12} and the convexity assumptions \eqref{eq.unused1} and \eqref{eq.unused2} are used.

Finally, given a compact set $Z$ as in the statement, by Mahler's compactness criterion, we can choose $\varepsilon>0$ and $\delta>0$ small enough so that the function $\beta_\infty\df\beta_{\varepsilon,\infty}^\delta$ is uniformly bounded on $Z$. %If all components in $v$ are large and we force $q_0(v)$ to be sufficiently small, one of the $q_\lambda(v)$ needs to be large
By the discussion above, this function has all desired properties.

\textsc{Case 2: $G$ is closed subgroup of $\SL_d(\R)$ and $\Lambda=G\cap\SL_d(\Z)$.} Then $X=G/\Lambda$ is a closed subset of $X_0=\SL_d(\R)/\SL_d(\Z)$ by~\cite[Theorem~1.13]{raghunathan}. Thus, we can use the height  function from Case~1 above. 

\textsc{Case 3: $G=H$ is a connected real rank one simple
Lie group.} 
We assume  $X=G/\Lambda$ is noncompact.  Let
$V=\bigwedge^r \mathfrak g$ endowed with a norm
$\norm{\cdot}$, where
$r$ is the dimension of the unipotent radical of a minimal parabolic subgroup of  $G$. 
Let $v_0\in V$ be a nonzero vector which corresponds to the Lie algebra of such a unipotent radical. 
	It follows from~\cite{gr70} (cf.~\cite[Proposition~3.1]{kw} 
	and~\cite[p.~54]{bq12})
	that there 
	exist $g_1, \dots, g_r\in G$ such that 
	for $i=1,\dots,r$ the vectors 
	$v_i=g_i\acts v_0$ in $V$
	have the following properties: 
	
	\begin{enumerate}[label=(\alph*)]
		\item $\Lambda v_i$ is closed and hence discrete in $V$ for $1\le i\le r$.
		\item  For any subset $F\subset G$, the set $F\Lambda\subset G/\Lambda$ is relatively compact if and only if 
		there exists $a>0$ such that $\norm{g\lambda\acts v_i}>a$ for any $\lambda\in \Lambda$, $g\in F$
		and $1\le i\le r$.
		\item There exists $a_0>0$ such that for 
		any $g\in G$ there exists at most one $v\in \bigcup_{1\le i \le r} \Lambda\acts v_i$  such that $\norm{g\acts v}<a_0$.
	\end{enumerate}
Let $V' $ be the $H$-invariant subspace complementary to $V^H$. In view of property~(b), we know that $v_0\in V'$. By Lemma~\ref{lem;contract}, for every $\delta>0$ small enough, for every $c>0$, we have that for every $n \in \N$ large enough  
	\begin{align}\label{eq;wu}
	\int_H \norm{h\acts v}^{-\delta}\dd \mu^{*n}(h)< c\norm{v}^{-\delta}
	\end{align}
holds for all nonzero $v\in V'$.
Using properties (a)--(c) and \eqref{eq;wu}	
it is straightforward to check that 
\begin{align*}
	 \beta_\infty(g\Lambda)=\max_{1\le i\le r}\max_{\lambda\in \Lambda}\norm{g\lambda\acts v_i}^{-\delta}
\end{align*}
is continuous, proper and satisfies \eqref{contr} when $\delta>0$ is small enough. It is also readily checked that $\beta_\infty(hx)\le \operatorname{N}(\Ad h)^{\kappa'\delta}\beta_\infty(x)$ for some $\kappa'$ depending only on $G$.

\textsc{Case 4: $G=\Aut(\mathfrak g)$ for $\mathfrak g$ semisimple without compact ideals.}
In view of  (4) at the begining of the proof,  we may assume that $G$ is connected. As $G$ is of adjoint type, it is center-free.  By~\cite[Theorem~5.22]{raghunathan}, after replacing $\Lambda$ by a finite index subgroup, there is a collection of semisimple factors $G_i$ of $G$ such that $G=\prod_i G_i$ and $\Lambda_i=G_i\cap\Lambda$ is an irreducible lattice in $G_i$. Then we have $G/\Lambda=\prod_i G_i/\Lambda_i$. Thus, if we can construct functions with the desired properties on all spaces $G_i/\Lambda_i$, then their sum is a Lyapunov function on $X=G/\Lambda$ with the same properties (possibly with different constants). In other words, we have further reduced to the case where the lattice $\Lambda$ in $G$ is irreducible. We can also assume that $\Lambda$ is non-uniform in view of (3) at the beginning of the proof.

Case~3 handles the case of $G$ with real rank one. Thus, we  may additionally assume that the rank is at least two. Then Margulis' arithmeticity theorem says that $\Lambda$ is arithmetic. In our setting, this implies that there is an isomorphism $\sigma\colon G \to G'$ where $G'$ is the connected component of a semisimple real algebraic subgroup of $\SL_{d'}(\R)$ defined over $\Q$  such that $\sigma(\Lambda)$ and $\Lambda'=G'\cap\SL_{d'}(\Z)$ are commensurable (see~\cite[Corollary~6.1.10]{zimmer}). %$\Lambda'$ is a lattice by Borel--Harish--Chandra
Then by Proposition~\ref{prop;facts}(iii), $\sigma_*\mu$ is $\sigma(H)$-expanding, and we conclude using Case~2 and the comment
(2) on commensurability at the start of the proof.
% the bound $\beta_\infty(hx)\le \operatorname{N}(\Ad(h))^\kappa\beta(x)$ holds since the norms in the two faithful representations $\sigma$ and $\Ad$ are comparable

\textsc{Case 5: General case.}
Let $\mathfrak r$ be the maximal amenable ideal of $\mathfrak g$, set $\mathfrak s=\mathfrak g/\mathfrak r$ and $R=\ker(\Ad_{\mathfrak s})$. Then $\mathfrak s$ is the largest semisimple quotient of $\mathfrak g$ without compact ideals and, by semisimplicity, $G/R$ identifies with a finite index subgroup $S$ of $\Aut(\mathfrak s)$. From~\cite[Lemma~6.1]{bq12} we know that $\Lambda\cap R$ is a cocompact lattice in $R$ and the image $\Lambda_S=\Ad_{\mathfrak s}(\Lambda)$ is a lattice in $S$. In particular, the projection $G/\Lambda\to S/\Lambda_S$ is proper. Setting $H_S=\Ad_{\mathfrak s}(H)$, we moreover have that $(\Ad_{\mathfrak s})_*\mu$ is $H_S$-expanding by Proposition~\ref{prop;facts}(iii). By Case~4 above, the theorem holds for $S/\Lambda_S$. Precomposing the obtained Lyapunov function with the projection $G/\Lambda\to S/\Lambda_S$ produces the  desired function $\beta_\infty$ on $X$. Properties (i)--(iii) carry over from the subcases, using for the latter property that the norm in the adjoint representation controls the norms in any other representation after taking a suitable power. 
\end{proof}

Before moving on, we make a simple remark that will be of use in the next part.
\begin{rem}\label{rk.modify.kappa}
Notice that by considering a small power of $\beta_\infty$, at the cost of increasing the constants $a \in (0,1)$ and $b$, one can modify $\kappa>0$ that satisfies property~(iii) in Theorem~\ref{thm;hgeneral}. Indeed, given $\delta \in (0,\kappa)$, using Jensen's inequality, the function $\beta_\infty^{\delta/\kappa}$ is seen to also satisfy the contraction condition \eqref{contr} with the same $m \in \N$ and possibly different constants $a \in (0,1)$ and $b>0$. Moreover, $\beta_\infty^{\delta/\kappa}(hx) \le \operatorname{N}(\Ad h)^\delta\beta_\infty^{\delta/\kappa}(x)$.
\end{rem}

\subsection{Height function with respect to singular subspaces}
\label{subsec;singular}

In this section we construct a height function with respect to a relatively compact subset of a lower-dimensional homogeneous subspace of $X=G/\Lambda$. In contrast to the height function used in~\cite{bq13}, which satisfies a contraction property with respect to a first return Markov operator, our height function will satisfy a contraction property with respect to $A_\mu$ itself. Our construction is inspired by the work of Eskin--Mirzakhani--Mohammadi~\cite{emm} on random walks on moduli space.

To state the main result of this subsection, we start by recalling some notation and fixing some data. Let $G$ be a Lie group and $\Lambda< G$ a lattice. Let $H\leqs G$ be a connected semisimple Lie subgroup with finite center and no compact factors. Let $\mu$ be an $H$-expanding probability measure on $H$ with finite exponential moments. Since $\mu$ has finite exponential moments, we can fix $\delta_0\in(0,1)$ such that $\int_H \operatorname{N}(\Ad(h))^{\delta_0}\dd\mu(h)<\infty$. Fix an arbitrary compact subset $Z$ of $G/\Lambda$ and let $\beta_\infty\colon G/\Lambda \to [1,\infty]$ be the proper lower semicontinuous function given by Theorem~\ref{thm;hgeneral}. By passing to a small enough power, we will suppose that $\beta_\infty$ satisfies $\beta_\infty(hx) \le \operatorname{N}(\Ad(h))^{\delta_0}\beta_\infty(x)$ for every $h \in H$ and $x \in G/\Lambda$ (see Remark~\ref{rk.modify.kappa}). Moreover, given $\varepsilon>0$, we define 
\begin{align*}
X_\varepsilon=\set{x \in G/\Lambda \for \beta_\infty(x)\le\varepsilon^{-1}}.
\end{align*}
Since $\beta_\infty$ is lower semicontinuous and proper, $X_\varepsilon$ is a compact subset of $X$. Here is the result we aim to prove.
\begin{thm}\label{thm;hN}
Given $\varepsilon>0$ sufficiently small, {for any sufficiently small open neighborhood $O$ of the identity in $C_G(\Gamma_\mu)$ and} for any $Y \in \calS(\Gamma_\mu)$ there exists a height function $\beta_{\mathcal N }\colon HX_\varepsilon\to [1, \infty]$ together with constants $n\in \N$, $a_0 \in (0,1)$ and $b_0>0$ such that for any $x\in HX_\varepsilon$ we have
	\begin{align*}
	A_\mu^{n}(\beta_{\mathcal N})(x)\le a_0 \beta_{\mathcal N }(x)+b_0, 
	\end{align*}
	and such that, denoting $\mathcal N=OY$,
	\begin{enumerate}[label=\textup{(\roman*)}]
	\item 	 	 $\beta_{\mathcal N }(x)=\infty$ if and only if $x\in \mathcal N\cap HX_\varepsilon $,
	\item        {$\beta_{\mathcal{N}}$ is bounded on compact subsets of $X_\varepsilon\setminus\overline{O}Y$},
	\item        for any $\ell\ge 1$, the set $\beta^{-1}_{\mathcal N}([1, \ell])$ is a compact subset of $X$. 
	\end{enumerate}
\end{thm}

The rest of this subsection is devoted to the proof of this result, which will require two preliminary lemmas. We fix an inner product on $\mathfrak g$, denote by $\norm{\cdot}$ the associated operator norm on $\End(\mathfrak{g})$, and to ease the notation, we set
\begin{align*}
\operatorname{N}_a(h)\df \operatorname{N}(\Ad h)=\max\set{\norm{\Ad(h)}, \norm{\Ad(h^{-1})}}, 
\end{align*}
where $\Ad$ denotes the  adjoint action of $H$ on the Lie algebra $\mathfrak g$ of $G$. 

\begin{lem}\label{lem;return}
There exist constants $C\ge 1$, $k\in\N$ and $\varepsilon_0>0$ such that for any $\varepsilon \in (0,\varepsilon_0)$ and any $x\in HX_\varepsilon$ there exists $h \in \Gamma_\mu$ with $\operatorname{N}_a(h) \le C  \beta_\infty(x)^k$ such that $hx \in X_\varepsilon$.
\end{lem}

\begin{proof}
Set $M\df\int \operatorname{N}_a(h)^{\delta_0} \dd\mu(h) <\infty$ and let a positive
\begin{align*}
\varepsilon< \min\set[\bigg]{\frac{1}{4},\biggl(\frac{1-a}{1-a+b}\biggr)^2}\eqqcolon\varepsilon_0
\end{align*}
be given, where $a \in (0,1)$ and $b>0$ are the constants given by Theorem~\ref{thm;hgeneral}. Let $x\in HX_\varepsilon$. Since $\beta_\infty^{-1}(\set{\infty})$ is $H$-invariant, we have $\beta_\infty(x)<\infty$, so that we may define $n_x \ge 1$ to be the smallest integer such that $a^{n_x} \beta_\infty(x) \le 1$. It follows that
\begin{align*}
A_\mu^{m n_x}(\beta_\infty)(x) \le a^{n_x} \beta_\infty(x) + \frac{b}{1-a} \le\frac{1}{\sqrt{\varepsilon}},
\end{align*}
where $m \in \N$ is as in Theorem~\ref{thm;hgeneral}.

Now decompose $\mu^{*m n_x}$ as a sum of two non-negative measures $\mu_1+\mu_2$ where $\mu_2$ is the restriction of $\mu^{*mn_x}$ to the set $\set{\operatorname{N}_a(\cdot)\ge R_x}$ for $R_x= 2^{1/\delta_0}M^{m n_x/\delta_0}$. By submultiplicativity of $N_a$ we have $\int \operatorname{N}_a(h)^{\delta_0} \dd\mu^{*m n_x}(h) \le M^{m n_x}$. Using this bound together with the Markov inequality, we deduce that $\mu_2(H) \le \frac{1}{2}$ and hence $\mu_1(H) \ge \frac{1}{2} \ge \sqrt{\varepsilon}$. On the other hand, we know
\begin{align*}
\int_H \beta_\infty(hx) \dd\mu_1(h) \le A_\mu^{mn_x}(\beta_\infty)(x) \le \frac{1}{\sqrt{\varepsilon}}.
\end{align*}
Now, considering the probability measure $\hat{\mu}_1= \frac{1}{\mu_1(H)}\mu_1$ , we deduce $A_{\hat{\mu}_1}\beta_\infty(x) \le \frac{1}{\varepsilon}$. This means that there exists $h\in\supp(\hat{\mu}_1) \subset \Gamma_\mu$ such that $\beta_\infty(hx) \le \frac{1}{\varepsilon}$.
Finally, since by construction $n_x \le 1+ \frac{\log \beta_\infty(x)}{-\log a}$, we also obtain
\begin{align*}
\operatorname{N}_a(h) \le R_x= 2^{1/\delta_0} M^{m n_x/\delta_0} \le 2^{1/\delta_0}M^{m/\delta_0} \beta_\infty(x)^{m\frac{\log M}{-\delta_0 \log a}}.
\end{align*}
This shows that the statement holds by setting $C=2^{1/\delta_0}M^{m/\delta_0}$ and $k=m \lceil \frac{\log M}{-\delta_0 \log a} \rceil$.
\end{proof}

Let $Y$ be a homogeneous space in $\calS(\Gamma_\mu)$ and denote by $N$ its stabilizer group. Recall that this means that $N\geqs\Gamma_\mu$ is a closed subgroup of $G$, $Y$ is given by $Nx$ for some $x \in G/\Lambda$, and there is an $N$-invariant probability measure on $Nx$ which is invariant and ergodic with respect to $\Gamma_\mu$. By Theorem~\ref{thm;rigidity}, the Lie algebra
$\mathfrak n$ of $N$ is $H$-invariant with respect to the adjoint action. 
We write $\mathfrak g$ as a direct sum  of $\Ad
(H)$-invariant subspaces
\begin{align}\label{eq;add3}
\mathfrak g=(\mathfrak n+\mathfrak l)\oplus \mathfrak v,
\end{align}
where $\mathfrak l$ is the centralizer of $\mathfrak h$ and $\mathfrak v$ is a complementary $H$-invariant 
subspace of $\mathfrak n+\mathfrak l$. Recall that by the epimorphic property of $\Gamma_\mu$ in $H$, $\mathfrak l$ is also the Lie algebra of $C_G(\Gamma_\mu)$.

\begin{lem}\label{lem;tedious}
With the notation of the previous paragraph, for every $Y \in \calS(\Gamma_\mu)$ and compact set $K\subset X=G/\Lambda$, there exist an open neighborhood $O$ of the identity in $C_G(\Gamma_\mu)$ and $r \in (0,1)$ with the property that for any $x\in K$, there is at most one $v\in\mathfrak v$ such that 
\begin{align}\label{eq;sr}
\exp(v)x\in  OY \quad \text{and} \quad  \norm{v}< r.
\end{align}
Moreover, the set $E$ of $x\in X$ for which $v\in\mathfrak v$ with \eqref{eq;sr} exists is open in $X$ and the map $E\cap K\to\mathfrak v,\,x\mapsto v$ is continuous.
\end{lem}
\begin{proof}
Let $K'$ be a compact neighborhood of $K$. In view of \eqref{eq;add3}, we can choose $O$, $r$ and a neighborhood $U$ of the identity in $G$ so that all of the following hold:
\begin{enumerate}[label=(\alph*)]
    \item we have $UK\subset K'$,
    \item the natural map  $U \to Uy $ is injective for all $y\in K'$,
    \item for every $y\in Y\cap K'$ we have $U y\cap Y= (U\cap N)y$,
    \item the map $B_r(\mathfrak v)\times (U\cap ON)\to G,\,(v,g)\mapsto\exp(v)g$ is a diffeomorphism onto an open neighborhood of the identity in $G$, where $B_r(\mathfrak v)$ denotes the open $r$-ball in $\mathfrak v$, and
    \item we have $o_2^{-1}\exp(v_2) \exp(-v_1)o_1\in U$ for every $v_1,v_2 \in \mathfrak{g}$ with $\norm{v_i}<r$, $i=1,2$, and $o_1,o_2\in O$.
\end{enumerate}
Now let $x\in K$ and $v_1, v_2\in \mathfrak v$ satisfy \eqref{eq;sr}, say $\exp(v_i)x= o_iy_i$  with $o_i\in O$ and $y_i \in Y$ for $i=1,2$. Using properties (a) and (e) we know $y_1\in K'$. Moreover, $y_2= o_2^{-1}\exp(v_2) \exp(-v_1) o_1y_1$. Applying properties (b), (c) and (e), we deduce that
\begin{align*}
o_2^{-1}\exp(v_2) \exp(-v_1) o_1=n\in U\cap N,
\end{align*}
which means that
\begin{align*}
\exp(-v_1)o_1=\exp(-v_2)o_2n.
\end{align*}
Using (e) once more, we see that $o_1,o_2n\in U\cap ON$. Hence, property~(d) implies that $v_1=v_2$, giving uniqueness. Since $O\subset U$, the final claims of the lemma also follows from (d).
\end{proof}

\begin{proof}[Proof of Theorem~\textup{\ref{thm;hN}}]
Since there is a substantial amount of relevant notation and auxiliary objects, let us start the proof by recalling the initial data. The probability measure $\mu$ on $H$ is $H$-expanding with finite exponential moments, $Z$ is a compact subset of $X=G/\Lambda$ and $\beta_\infty\colon G/\Lambda \to [1, \infty]$ is as given by Theorem~\ref{thm;hgeneral}. By the latter (and Remark~\ref{rk.modify.kappa}), the function $\beta_\infty$ satisfies \eqref{contr} with some $m \in \N$, $a \in (0,1)$ and $b>0$ and $\beta_\infty(hx) \le \operatorname{N}_a(h)^{\delta_0}\beta_\infty(x)$ for every $x \in G/\Lambda$ and $h \in H$, where $\delta_0\in(0,1)$ is chosen so that $\int_H \operatorname{N}_a(h)^{\delta_0} \dd\mu(h)<\infty$. Let $\varepsilon_0>0$, $k \in \N$ and $C\ge 1$ be given by Lemma~\ref{lem;return} and fix $\varepsilon \in (0,\varepsilon_0)$.  Let $O$ be a {relatively compact open} neighborhood of the identity in $C_G(\Gamma_\mu)$ and $r \in (0,1)$ 
{such that the conclusion of Lemma~\ref{lem;tedious} holds} with a compact neighborhood $K$ of $X_\varepsilon=\set{x\in X\for \beta_\infty(x)\le \varepsilon^{-1}}$. Let $Y \in \calS(\Gamma_\mu)$, denote by $N$ its stabilizer group, by $\mathfrak n$ its Lie algebra, and set $\mathcal N=OY$. Finally, let $\mathfrak l$ be the Lie algebra of $C_G(\Gamma_\mu)$ and choose an $\Ad(H)$-invariant complementary space $\mathfrak v$ so that \eqref{eq;add3} holds.

Since $\mu$ is $H$-expanding with finite exponential moments and $\mathfrak v$ has no nonzero $H$-fixed vectors, by Lemma~\ref{lem;contract} there exists
\begin{align}\label{eq;list}
 0<\theta <\min\set{\delta_0,1/k}  
\end{align}
such that for every $a'\in(0,1)$ we have, for all $n \in \N$ large enough, 
\begin{align}\label{eq;qing}
\int_H \norm{\Ad(h)v}^{-\theta}\dd \mu^{*n}(h)\le  a'\norm{v}^{-\theta}
\end{align}
for any nonzero $v\in \mathfrak v$. We fix such $n \in \N$ that is a positive multiple of $m \in \N$. Without loss of generality, we assume $a'>a$ and let $\varepsilon'>0$ be such that $a'=(1+\varepsilon')a$. Since $m | n$, \eqref{contr} implies that
\begin{align}\label{eq;lan}
\int_H \beta_\infty(hx)\dd\mu^{*n}(h)\le a \beta_\infty (x)+ \frac{b}{1-a}.   
\end{align}

For $x\in  HX_\varepsilon$, we define
\begin{align*}
r_x = rC^{-1}\beta_\infty(x )^{-k}.
\end{align*}
Next, we claim that for every $x \in HX_\varepsilon$, there exists at most one $v\in \mathfrak v$ such that 
\begin{align}\label{eq;small}
\exp(v) x\in \mathcal N\quad \text{and} \quad \norm{v}< r_x. 
\end{align}
Indeed, by Lemma~\ref{lem;return}, there exists $h\in \Gamma_\mu$ with $\operatorname{N}_a(h)\le  C\beta _{\infty }(x)^k$ such that $h x\in X_\varepsilon$. 
Since $\mathcal N$ is $\Gamma_\mu$-invariant, we have 
\begin{align*}
\exp(v) x\in \mathcal N \quad \text{if and only if}\quad h\exp(v) x= \exp(\Ad(h)v) hx \in \mathcal N.
\end{align*}
Since $\norm{\Ad(h)v} \le \operatorname{N}_a(h)\norm{v} \le r$, if such an $v \in\mathfrak v$ exists, it is unique thanks to Lemma~\ref{lem;tedious} (applied to $hx \in X_\varepsilon$) and the choice of $r>0$, where we are using that $\mathfrak v$ is $H$-invariant. 
	
Using the claim above, we may define $\alpha\colon  HX_\varepsilon\to [1, \infty]$ by
\begin{align*}
    \alpha(x)=\begin{cases}
      \norm{v}^{-\theta},   & \text{if there exists } v\in \mathfrak v \text{ satisfying \eqref{eq;small},}  \\
     \hfill r_x^{-\theta},    & \text{otherwise.}
    \end{cases}
\end{align*}
Using the corresponding property for $\beta_\infty$ and the choice of $\theta$ in \eqref{eq;list}, it is readily checked that for every $x \in HX_\varepsilon$ and $h \in \Gamma_\mu$, we have $\alpha(hx) \le \operatorname{N}_a(h)^{\delta_0} \alpha(x)$. We shall show that
	\begin{align*}
	\beta_{\mathcal N}=\beta_\infty (x) + \alpha(x).
	\end{align*}
satisfies all requirements of the theorem.

To proceed, we start by decomposing $\mu^{*n}$ as a sum $\mu_1 + \mu_2$ of two non-negative measures with $\mu_1$ of compact support and $\mu_2$ satisfying
\begin{align*}
    \int_H \operatorname{N}_a(h)^{\delta_0} \dd\mu_2(h)< \frac{1-a'}{2}.
\end{align*}
It follows that 
\begin{align}\label{eq.deal.with.mu2}
\int_H \alpha(hx) \dd\mu_2(h) \le \alpha(x) \int_H \operatorname{N}_a(h)^{\delta_0}\dd\mu_2(h) \le \alpha(x)\frac{1-a'}{2}.
\end{align}
Denote by $D$ the constant $r^{-1}C M^k$, where $M=\sup \set{\operatorname{N}_a(h) \for h \in \supp(\mu_1) }$. Then $D>M^k\ge 1$ by choice of $r$, and for any element $h \in S_{\pm}\df\supp(\mu_1)\cup \supp(\mu_1)^{-1}$  we have  
	\begin{align}
	    \label{eq;lip}
	    \beta_\infty(hx)\le M \beta_\infty(x) \text{ and hence } r_x\le D r_{hx}.
	\end{align}

We are now going to establish the contraction property for $\beta_{\mathcal N}$ by distinguishing several cases based upon the size of $\alpha(x)$.

If $\alpha(x)>D^2 r_x ^{-\theta}$, then there exists a uniquely determined  $v\in \mathfrak v$  so that \eqref{eq;small} holds and  $\alpha(x)= \norm{v}^{-\theta}$.
In particular, 
\begin{align*}
\norm{v}<  D^{-2/\theta }r_x< D^{-2}r_x.
\end{align*}
Together with \eqref{eq;lip}, the previous inequality  implies that for $h \in S_{\pm}$, we have 
\begin{align}\label{eq;estimate}
\norm{\Ad (h) v}\le   \operatorname{N}_a(h)\cdot \norm{v}<D \cdot D^{-2}r_x=D^{-1} r_x\le  r_{hx}.
\end{align}
Since $\exp(v) x$ belongs to the $\Gamma_\mu$-invariant set $\mathcal N$, we have
$\exp(\Ad (h)v)hx \in \mathcal N $. In view of \eqref{eq;estimate} and the definition of $\alpha$ it follows that $\alpha(hx)=\norm{\Ad(h)v}^{-\theta}$. By \eqref{eq;qing},
\begin{align*}
\int_ H \alpha(hx)\dd\mu_1(h)= \int_H \norm{\Ad (h)v}^{-\theta}\dd\mu_1(h) \le \int_H \norm{\Ad (h)v}^{-\theta}\dd\mu^{*n}(h) \le a' \alpha(x).
\end{align*}
Combining with \eqref{eq.deal.with.mu2}, we get
\begin{align*}
\int_H \alpha(hx) \dd\mu^{*n}(h)=\int_H \alpha(hx) \dd(\mu_1+\mu_2)(h) \le \frac{1+a'}{2}\alpha(x).
\end{align*}
Together with \eqref{eq;lan}, the previous inequality yields
\begin{align*}
\int_H \beta_{\mathcal{N}}(hx) \dd\mu^{*n}(h) \le \frac{1+a'}{2}\beta_{\mathcal{N}}(x)+\frac{b}{1-a}.
\end{align*}
Therefore, we proved the contraction property of $\beta_{\mathcal{N}}$ for $x \in HX_\varepsilon$ satisfying $\alpha(x)> D^2r_x^{-\theta}$.
	
Now let $x \in HX_\varepsilon$ be such that $\alpha(x) \le D^2r_x^{-\theta}$. In this case, we have
\begin{align}\label{eq;either}
\alpha(x)\le D^2 r_x^{-\theta}= D^2 r^{-\theta} C^\theta \beta_\infty^{k\theta}(x)\le D^3  \beta_\infty(x).
\end{align}
We claim that for any $h\in S_{\pm}$, we have
\begin{align}\label{addtwo}
\alpha(hx)\le D^4 r_{hx}^{-\theta}.
\end{align}
If not, then using \eqref{eq;lip} and the fact that $\alpha(hx) \le M\alpha(x)\le D \alpha(x)$, we find
\begin{align*}
\alpha(x)\ge D^{-1}\alpha(hx)> D^{-1} \cdot D^4 r_{hx}^{-\theta}=D^3r_{hx}^{-\theta} \ge D^{3-\theta} r_{x}^{-\theta},
\end{align*}
which contradicts the first inequality in \eqref{eq;either} since $\theta \in (0,1)$ and $D>1$.
By \eqref{addtwo} and \eqref{eq;lip}
\begin{align*}
\alpha(hx)&\le D^4r_{hx}^{-\theta}= D^4 r^{-\theta}C^\theta\cdot \beta_{\infty}^{k\theta }(hx)
\le D^5 \beta_{\infty}^{k\theta}(x) = D^5 \beta_{\infty}^{k\theta-1}(x)\cdot \beta_{\infty}(x).
\end{align*}
Since $k\theta<1$, if $\beta_\infty(x)$ is larger than some constant depending only on $\varepsilon'a,k\theta$ and $D$, we will have
\begin{align*}
D^5\beta_{\infty }^{k \theta-1 }(x)<\varepsilon'a. 
\end{align*}
In view of \eqref{eq;either}, we know that $\beta_\infty(x)$ is sufficiently large provided that
$\alpha(x)$ is (depending on $D$). 
Therefore, there exists $b'>0$ (depending on $\varepsilon'a, k \theta,  D$) so that if 
\begin{align}\label{eq;either_one}
b'\le \alpha(x)\le D^2 r_x^{-\theta},
\end{align}
then for any $h\in S_{\pm}$
\begin{align}\label{eq;or_this}
\alpha(hx)\le \varepsilon'a \beta_\infty(x). 
\end{align}
So in the case where \eqref{eq;either_one} holds,
combining \eqref{eq;lan}, \eqref{eq.deal.with.mu2} and \eqref{eq;or_this}, we deduce
\begin{align*}
\int_H \beta_{\mathcal N}(h x)\dd \mu^{*n}(h) \le \frac{1+a'}{2} \beta_{\mathcal{N}}(x)+\frac{b}{1-a},
\end{align*} 
proving the required contraction property.

To treat the remaining case, suppose now that  $x\in  H X_\varepsilon $
and $\alpha(x)\le\min\set{b', D^2 r_x^{-\theta}}$. We claim that $\alpha(hx)\le D^3 b'$ for all $h\in S_{\pm}$. Supposing the contrary, we would have
\begin{align*}
\alpha(hx)>D^3b' \ge D^3 \alpha(x)\ge D^3 r_x^{-\theta}.
\end{align*}
From this, using the inequality $\alpha(hx)\le D \alpha(x)$, it follows that
\begin{align*}
\alpha(x)\ge D^{-1} \alpha(hx)> 
D^2  r_x^{-\theta},
\end{align*}
a contradiction. 
Therefore, recalling \eqref{eq;lan} and \eqref{eq.deal.with.mu2}, we obtain
\begin{align*}
\int_H \beta_{\mathcal N}(h x)\dd \mu^{*n}(h)&=
\int_H \alpha(h x)\dd \mu^{*n}(h)+
\int_H \beta_{\infty}(h x)\dd \mu^{*n}(h)  \\
&\le D^3b'+\frac{1-a'}{2}\alpha(x) +a\beta_\infty (x)+\frac{b}{1-a}\\
	& \le \frac{1+a'}{2} \beta_{\mathcal N} (x)+D^3b'+\frac{b}{1-a}. 
\end{align*}
We have thus concluded the proof of the contraction property with $a_0=(1+a')/2$ and the additive constant $b_0=D^3b'+b/(1-a)$.

It remains to prove the claims (i)--(iii). Since $\beta_\infty$ is finite on $HX_\varepsilon$, (i) is directly seen to hold by  definition of $\beta_{\mathcal N}$. {Property (ii) is also immediate from the definition of $\beta_{\mathcal{N}}$, since $\beta_\infty$ is bounded on $X_\varepsilon$ and any compact subset not intersecting $\overline{O}Y$ has positive distance to $\mathcal{N}$.} To prove (iii), let $(x_j)_j$ be a sequence in $HX_\varepsilon$ with $\beta_{\mathcal N }(x_j)\le \ell$ for all $j\in\N$ for some $\ell \in \R$. Since $\beta_{\mathcal{N}}=\beta_\infty + \alpha$ with $\alpha\ge 0$, we also have $\beta_\infty(x_j)\le \ell$ for all $j$. Since $\beta_\infty$ is proper, we may suppose that $\lim_{j \to \infty} x_j=x$ for some point $x\in X$. We need to prove that $x\in HX_\varepsilon$ and $\beta_{\mathcal N}(x)\le \ell$.

We first show that $x\in HX_\varepsilon$.
It follows from Lemma~\ref{lem;return} that there is a compact subset $K_\ell$ of $\Gamma_\mu$ such that for any $j\in \N$, there exists $h_j\in K_\ell$ so that $h_j x_j\in X_\varepsilon$. Since $ X_\varepsilon$ is compact, by possibly passing to a subsequence, we may assume that $h_j x_j$ converges to some $y\in X_\varepsilon$ and $h_j$ converges to some $h\in \Gamma_\mu$. So we have 
\begin{align*}
\lim_{j \to \infty} h_j x_j= hx= y,
\end{align*}
which implies $x=h^{-1}y\in HX_\varepsilon$.
	
	Finally, we show that 
\begin{align}\label{eq;betan}
\alpha(x)\le \liminf_{j \to \infty}\alpha(x_j),    
\end{align}
which will complete the proof in view of the lower semicontinuity of $\beta_\infty$ and the definition of $\beta_{\mathcal N}$. First, let us pass to a subsequence so that the liminf in \eqref{eq;betan} is a limit, say $\liminf_{j\to\infty}\alpha(x_j)=\lim_{j\to\infty}\alpha(x_j)\eqqcolon\alpha_1$. If $\alpha(x)=r_x^{-\theta}$, then \eqref{eq;betan} follows from the definition of $r_x$ and lower semicontinuity of $\beta_\infty$. Suppose therefore that $\alpha(x)>r_x^{-\theta}$. This implies that there exists a unique $v \in \mathfrak v$ such that $\exp(v)x \in \mathcal{N}$ and $\norm{v} <r_x$. Using Lemma~\ref{lem;return}, choose $h\in\Gamma_\mu$ with $\operatorname{N}_a(h)\le C\beta_\infty(x)^k$ such that $hx\in X_\varepsilon$. Then $\norm{\Ad(h)v}<r$ and $\exp(\Ad(h)v)hx\in\mathcal N$. Now, since the points $hx_j$ converge to $hx$, for large $j$ they lie in the neighborhood $K$ of $X_\varepsilon$ to which we applied Lemma~\ref{lem;tedious}. Thus, the last claim in this lemma imply that there exist $v_j\in\mathfrak v$ with $v_j\to v$ such that $\exp(v_j)x_j\in\mathcal N$. Note that since the values $r_{x_j}^{-\theta}$ are contained in $[0,\ell]$, up to passing to a further subsequence, we may suppose that they converge to $\alpha_2$. Clearly, $\alpha_1 \ge \alpha_2$. 
If $\alpha_1>\alpha_2$, then for large $j$ we have $\alpha(x_j)\ge \norm{v_j}^{-\theta}$ %the $v_j$ might not be the vectors appearing in the first case of the def of $\alpha$, but the only thing that can happen is that they are bigger, so that the inequality claimed here holds
and it follows that \eqref{eq;betan} holds since $\norm{v_j}^{-\theta}\to\norm{v}^{-\theta}=\alpha(x)$. On the other hand, in case $\alpha_1=\alpha_2$ we know that for every $\epsilon>0$, for $j \in \N$ large enough, we have $\norm{v_j}+\epsilon>r_{x_j}$. But since $v_j \to v$ and $\epsilon>0$ is arbitrary, this implies that  $\alpha(x)=\norm{v}^{-\theta}\le \lim_{j \to \infty} r_{x_j}^{-\theta}=\alpha_2=\alpha_1$, as desired.
\end{proof}

\section{Recurrence, equidistribution, topology of homogeneous measures}\label{sec;rec-equi-top}
Using the ingredients from \S\S\ref{sec;rigidity}--\ref{sec;height_functions}, we can now give the proofs of our results on recurrence, orbit closures, equidistribution, and topology of $\calS(\Gamma_\mu)$. The following lemma is used to extract the necessary information from the height functions constructed in the previous section.
\begin{lem}\label{lem;no_mass}
Let $H$ be a locally compact $\sigma$-compact metrizable group and $X$ a locally compact $\sigma$-compact metrizable space endowed with a continuous $H$-action. Let $\mu$ be a Borel probability measure on $H$ and $\beta\colon X\to [1,\infty]$ be a lower semicontinuous function such that there exist $m\in\N$, $a \in (0,1)$ and $b>0$ such that
\begin{align}\label{contr_ineq}
A_\mu^m(\beta)(x)\le a \beta(x)+b 
\end{align}
for all $x\in X$. 
Suppose that for every $\varepsilon>0$ the set $X_\varepsilon=\beta^{-1}([0,\varepsilon^{-1}])$ is compact and that the set $X_\infty=\beta^{-1}(\set{\infty})$ is $\Gamma_\mu$-invariant. 
Then the following holds:
\begin{enumerate}[label=\textup{(\roman*)}]
\item For any $\delta >0$ there exists a compact subset $R_\delta\subset X\setminus X_\infty$ such that for any $x\in X$ with $\beta(x)<\infty$ {there exists $n_x\in\N$ with $n_x=O(\log\beta(x))$ such that 
\begin{align*}
\mu^{*n}* \delta _x (R_\delta)\ge  1-\delta
\end{align*}
for every $n\ge n_x$.}
\item For every $x\in X$ with $\beta(x)<\infty$,
 for $\mu^\N$-a.e.\ $(g_i)_i\in \Gamma_\mu^\N$, every weak* limit $\nu$ of the sequence $(\frac{1}{n}\sum_{k=0}^{n-1}\delta_{g_k\dotsm g_1x})_n$ of empirical measures satisfies $\nu(X\setminus X_\infty)=1$.
\end{enumerate}

\end{lem}
The techniques going into the first part of the lemma are by now standard. The second part is basically~\cite[Proposition~3.9]{bq132}. Related ideas also appear in Markov chain theory (see e.g.~\cite[Theorem~18.5.2]{meyn-tweedie} and the references given there). We include a brief proof for convenience.

\begin{proof}
Let $x\in X$ be such that $\beta(x)<\infty$. Iterating \eqref{contr_ineq}, we find for every $\varepsilon>0$ and $n\in\N$
\begin{align*}
\mu^{*mn}*\delta_x(X_\varepsilon^c)&\le \varepsilon\int_H\beta(hx)\dd\mu^{*mn}(h)\le \varepsilon\Bigl(a^n\beta(x)+\frac{b}{1-a}\Bigr).
\end{align*}
{For the proof of (i), given $\delta>0$, we set $\varepsilon=\frac{\delta(1-a)}{2b+2}$. Then the above estimate implies that for every $n \geq n_{0,x}\df\lceil\frac{\log\beta(x)}{-\log a}\rceil$, we have $\mu^{*mn}*\delta_x(X_\varepsilon)\ge 1-\delta/2$. Moreover, we may choose a compact subset $F \subset \Gamma_\mu$ such that $\mu^{*l} (F)\ge 1-\delta/2$ for all $0\le l<m$. Now setting $R_\delta$ to be the compact set $FX_\varepsilon$ which, since $X \setminus X_\infty$ is $\Gamma_\mu$-invariant, is contained in $X\setminus X_\infty$, we find
\begin{align*}
\mu^{*n}*\delta_x(R_\delta) \ge 1-\delta %here the correct lower bound is (1-\delta/2)^2
\end{align*}
for all $n\ge n_x\df mn_{0,x}$.}

For (ii), we appeal to~\cite[Proposition~3.9]{bq132}, which implies that for $\mu^\N$-a.e.\ $(g_i)_i\in \Gamma_\mu^\N$, for every $\delta>0$ there exists a compact subset $K\subset X\setminus X_\infty$ such that
\begin{align*}
\liminf_{n\to\infty}\frac1n\abs{\set{0\le k<n\for g_{km}\dotsm g_1x\in K}}\ge 1-\delta/2.
\end{align*}
Moreover, by the law of large numbers, by choosing a large enough compact set $F\subset\Gamma_\mu$ we can ensure that for $\mu^\N$-a.e.\ $(g_i)_i\in\Gamma_\mu^\N$
\begin{align*}
\liminf_{n\to\infty}\frac1n\abs{\set{0\le k<n\for g_{km+l}\dotsm g_{km+1}\in F\text{ for }0\le l<m}}\ge 1-\delta/2.
\end{align*}

Combining the above, it follows that for the compact subset $R=FK\subset X\setminus X_\infty$ we have
\begin{align*}
\liminf_{n\to\infty}\frac1n\abs{\set{0\le k<n\for g_k\dotsm g_1x\in R}}\ge 1-\delta
\end{align*}
for $\mu^\N$-a.e.\ $(g_i)_i\in \Gamma_\mu^\N$, and we conclude using a version of the Portmanteau lemma.
\end{proof}
\subsection{Recurrence}\label{subsec;recurrence}
We first prove our results about recurrence properties of $H$-expanding random walks.
\begin{proof}[Proof of Theorem~\textup{\ref{thm;recurrence}}]
Let $Z$ be a compact subset of $X\setminus \mathcal{N}$, where we recall that $\mathcal{N}=K_LY$ for a compact subset $K_L$ of $L=C_G(\Gamma_\mu)$, and let $\beta_\infty$ be a height function coming from Theorem~\ref{thm;hgeneral} such that $\beta_\infty$ is bounded on $Z$, say $Z\subset X_\varepsilon=\set{x\in X\for \beta_\infty(x)\le \varepsilon^{-1}}$ for some $\varepsilon>0$. If $Y=\emptyset$, we set $\beta=\beta_\infty$. Otherwise, we apply Theorem~\ref{thm;hN} to $Y_l=lY$ for finitely many points $l\in L$ such that the associated neighborhoods $O_l$ of the identity in $L$ coming out of the theorem satisfy $\overline{O_l}lY\cap Z=\emptyset$ and $K_L\subset \bigcup_lO_ll$. The associated height functions $\beta_l$ (extended to all of $X$ by the value $\infty$ on the complement of $HX_\varepsilon$) take the value $\infty$ on $O_llY$ and are bounded on $Z$. We set $\beta=\sum_l\beta_l$, which is a lower semicontinuous function on $X$ with compact sublevel sets by virtue of Theorem~\ref{thm;hN}(iii). 

In both cases, we now apply Lemma~\ref{lem;no_mass}(i) to the height function $\beta$. The set $R_\delta$ coming out of the lemma is a compact subset of $X\setminus\mathcal{N}$ such that for every $x\in X$ with $\beta(x)<\infty$, for $n\ge n_x$ with $n_x=O(\log\beta(x))$, we have $\mu^{*n}*\delta_x(R_\delta)\ge 1-\delta$. Since $\beta$ is bounded on $Z$ by construction, this estimate holds for all $n\ge n_0$ for all $x\in Z$. If $F$ is a compact subset of $\Gamma_\mu$ such that $\mu^{*n}(F)\ge 1-\delta$ for all $0\le n<n_0$, it follows that $\mu^{*n}*\delta_x(M_{Z,\delta})\ge 1-\delta$ for all $n\ge 0$ and all $x\in Z$ for the compact subset $M_{Z,\delta}\df R_\delta\cup FZ$ of $X\setminus\mathcal{N}$, where we used for the last containment that $\beta^{-1}(\set{\infty})$ is $\Gamma_\mu$-invariant.
\end{proof}

\begin{rem}
For $Y=\emptyset$, the recurrence property in Theorem~\ref{thm;recurrence} is referred to as (R1) in \cite{bq12,em}. In the case of a random walk given by a $G$-expanding probability measure on the quotient of $G$ by an irreducible lattice, a slightly stronger, ``uniform'' recurrence property (referred to as (R2)) can be established by using some results of \cite{em}.
\end{rem}

\subsection{Orbit closures and equidistribution}\label{subsec;equidistribution}
The proof of Theorem~\ref{thm;orbit} is similar to the proofs of the main results in~\cite{bq132}. 
\begin{proof}[Proof of Theorem~\textup{\ref{thm;orbit}}]
Provided $Y_x$ contains $x$, part (i) is an immediate consequence of~(ii). Moreover, taking a compactly supported and continuous test function, it is not hard to see that (ii) follows from (iii) by dominated convergence.

Let us thus prove (iii) with the additional property that $x\in Y_x$. 
For $\mu^{\N}$-a.e.~$(g_i)_i\in H^{\N}$, every weak* limit $\nu$ of the sequence $(\frac{1}{n}\sum_{k=0}^{n-1}\delta_{g_k\dotsm g_1x})_n$ of empirical measures is $\mu$-stationary by the Breiman law of large numbers (see~\cite[Corollary~3.3]{bq132}). By Theorem~\ref{thm;hgeneral} and Lemma~\ref{lem;no_mass}(ii),  
for $\mu^\N$-a.e.\ $(g_i)_i\in H^\N$ every such weak* limit is a probability measure on $X$. We restrict to a full measure set of $(g_i)_i$ where both these conclusions hold and let $\nu$ be a weak* limit of the sequence of empirical measures. 

Let $Y_0$ be a $\Gamma_\mu$-invariant homogeneous subspace of $X$ containing $x$ of minimal dimension. 
By Theorem~\ref{thm;rigidity} every ergodic component of $\nu$ is the homogeneous probability measure associated to an element of
\begin{equation*}
\calS(\Gamma_\mu,Y_0)\df\set{Y\in\calS(\Gamma_\mu)\for Y\subset Y_0}.
\end{equation*}
Let $Y\in\calS(\Gamma_\mu,Y_0)$ be such that $Y$ is not open in $Y_0$. Then by minimality of $\dim(Y_0)$ we know that $x\notin lY$ for any $l\in L\df C_G(\Gamma_\mu)$. 

Let $Z$ be an arbitrary compact subset of $X$, take a height function $\beta_\infty$ as in Theorem~\ref{thm;hgeneral}, and recall that $X_\varepsilon=\beta_\infty^{-1}([1,\varepsilon^{-1}])$. By Theorem~\ref{thm;hN}, for sufficiently small $\varepsilon>0$, there is an open neighborhood $O$ of the identity in $L$ and a height function $\beta_{\mathcal{N}}\colon HX_\varepsilon\to[1,\infty]$ satisfying the contraction property \eqref{contr_ineq} and such that
\begin{itemize}
\item for $x\in HX_\varepsilon$, $\beta_{\mathcal N}(x)=\infty$ if and only if $x\in OY$,
\item for every $\ell\ge 1$, $\beta_{\mathcal N}^{-1}([1,\ell])$ is a compact subset of $X$.
\end{itemize}
We extend $\beta_{\mathcal N}$ to all of $X$ with the value $\infty$ outside of $HX_\varepsilon$. Then the extension satisfies the assumptions of Lemma~\ref{lem;no_mass}. Write $X_{\infty,\mathcal N}$ for the set $\beta_{\mathcal N}^{-1}(\set{\infty})$, so that $HX_\varepsilon\cap OY\subset X_{\infty,\mathcal N}$. After further restricting to a full measure set of $(g_i)_i$ so that Lemma~\ref{lem;no_mass}(ii) holds, we thus find $\nu(HX_\varepsilon\cap OY)=0$. When $\varepsilon$ is small enough, this implies $\nu(Z\cap OY)=0$. %possible since $\beta_\infty$ is uniformly bounded on $Z$
We repeat this process for the homogeneous subspaces $lY$ for countably many $l\in L$ such that the translations $Ol$ of the associated neighborhoods $O$ cover $L$. This gives $\nu(Z\cap LY)=0$. Repeating again for countably many compact subsets $Z$ covering $X$, it follows that $\nu(LY)=0$.

Hence, in view of the countability statement in Proposition~\ref{prop;countability}, we deduce that $\nu(LY)=0$ holds for every $Y\in\calS(\Gamma_\mu,Y_0)$ that is not open in $Y_0$ (to be precise, after once more restricting to a countable intersection of full measure sets of $(g_i)_i\in H^\N$, once for each $Y$ in a countable set of representatives in \eqref{eq.representatives}). It follows that each ergodic component of $\nu$ must be a homogeneous measure of some $Y\in\calS(\Gamma_\mu,Y_0)$ that is open in $Y_0$. By~\cite[Lemma~2.9]{bq132}, these $Y$ are pairwise disjoint, so that there are only countably many of them. This means that for some $Y\in\calS(\Gamma_\mu,Y_0)$ open in $Y_0$ we must have $\nu(Y)>0$. Then necessarily $x\in Y$. By construction of $\nu$ and $\Gamma_\mu$-invariance of $Y$ it follows that $\nu(Y')=0$ for any $Y'\in\calS(\Gamma_\mu,Y_0)$ distinct from $Y$. Hence, all ergodic components of $\nu$ are in fact equal to the homogeneous probability measure on $Y$, which finishes the proof of  (iii).
\end{proof}

\begin{rem}[Non-averaged convergence in law]
It is a natural question, already posed by Benoist--Quint at the end of their survey~\cite{takagi_lectures}, whether, or under what conditions, the Ces\`aro average in Theorem~\ref{thm;orbit}(ii) can be removed. Unfortunately, in the generality of our results, this question of convergence of $\mu^{*n}*\delta_x$ towards $\nu_x$ seems to be out of reach with current methods. Answers are available only in certain special cases where additional structure can be exploited. For example, in the setting of toral automorphisms, the harmonic analytic approach used by Bourgain--Furman--Lindenstrauss--Mozes~\cite{bflm} allows them to obtain the convergence of $\mu^{*n}*\delta_x$ together with a speed depending on Diophantine properties of the starting point $x$. Their approach was recently refined and generalized to some nilmanifolds in the works \cite{he-desaxce,he-lakrec-lindenstrauss1, he-lakrec-lindenstrauss2} of He--de Saxc\'{e} and He--Lakrec--Lindenstrauss. Outside the realm of nilmanifolds, quantitative results on the convergence of $\mu^{*n}*\delta_x$ include the work of Buenger~\cite[\S3]{buenger} and Khalil--Luethi~\cite{khalil-luethi}, who consider some classes of measures supported on compact-by-solvable groups, and work of the first-named author~\cite{spread-out} on spread-out measures. 

Very recently, it was observed by B\'{e}nard~\cite{benard.foguel} that the non-averaged convergence can be ensured with some additional hypotheses using an old result of Foguel.
\end{rem}

\subsection{Topology of homogeneous measures}\label{subsec;topology}
Here we prove the Mozes--Shah type results regarding the weak* topology on the set of ergodic homogeneous subspaces of $X$. 

Let $G,H,\Lambda,X,\mu,\Gamma_\mu$ be as in Theorem~\ref{thm;orbit} and recall that $\calS(\Gamma_\mu)$ denotes the set of all $\Gamma_\mu$-invariant $\Gamma_\mu$-ergodic homogeneous subspaces $Y$ of $X$. Each element $Y$ of $\calS(\Gamma_\mu)$ carries an associated $\Gamma_\mu$-invariant and ergodic homogeneous probability measure $\nu_Y$. Using this, we embed $\calS(\Gamma_\mu)$ into the space $\mathcal{P}(X)$ of Borel probability measures on $X$ and endow $\calS(\Gamma_\mu)$ with the weak* topology induced from $\mathcal{P}(X)$. Also recall that for a subset $Z \subset X$, we let $\calS_Z(\Gamma)=\set{Y \in \calS(\Gamma) \for Y \cap Z \neq \emptyset}$. 

The following lemma will be useful for the proof of Proposition~\ref{prop.mozes-shah}. In the statement, given $Y\in\calS(\Gamma_\mu)$, we shall say that a point $y\in Y$ is \emph{$Y$-generic} if the conclusion of Theorem~\ref{thm;orbit}(ii) holds, i.e.\ if $\lim_{n\to\infty}\frac1n\sum_{k=0}^{n-1}\mu^{*k}*\delta_y=\nu_Y$ in the weak* topology. Note that $\nu_Y$-a.e.\ point is $Y$-generic by the Chacon--Ornstein ergodic theorem.
\begin{lem}\label{lem;no_mass2}
Let $(\nu_j)_j$ be a sequence of ergodic homogeneous measures associated to subspaces $Y_j\in\calS(\Gamma_\mu)$ converging to a finite measure $\nu$ on $X$ in the weak* topology. Let $\beta$ be a height function on $X$ satisfying the assumptions of Lemma~\textup{\ref{lem;no_mass}} and denote $X_\infty=\beta^{-1}(\set{\infty})$. Suppose that there is a sequence of $Y_j$-generic points $y_j\in Y_j$ such that $y_j\notin X_\infty$ for infinitely many $j$. Then $\nu(X\setminus X_\infty)=1$.
\end{lem}
\begin{proof}
We may assume $y_j\notin X_\infty$ for all  $j$. Let $\delta>0$. By Lemma~\ref{lem;no_mass}(i) there exists a compact subset $R_\delta\subset X\setminus X_\infty$ such that $\mu^{*n}*\delta_{y_j}(R_\delta)\ge 1-\delta$ for all $n\ge n_{y_j}$. Passing to the limit in the $Y_j$-genericity, this implies $\nu_j(R_\delta)\ge 1-\delta$. Letting $j\to\infty$, it follows that also $\nu(R_\delta)\ge 1-\delta$. The conclusion follows, since $R_\delta\subset X\setminus X_\infty$ and $\delta>0$ was arbitrary. 
\end{proof}

\begin{proof}[Proof of Proposition~\textup{\ref{prop.mozes-shah}}]
Let us first prove (ii). Let $(\nu_j)_j$ be a sequence of ergodic homogeneous probability measures associated to subspaces $Y_j$ in $\calS(\Gamma_\mu)$ converging to the homogeneous measure $\nu_\infty$ associated to $Y_\infty\in\calS(\Gamma_\mu)$. Take a sequence of $Y_j$-generic points $y_j\in Y_j$ such that $Z=\overline{\set{y_1,y_2,\dots}}$ is compact. Let $\beta_\infty$ be a height function from Theorem~\ref{thm;hgeneral} that is finite on $Z$, say with $Z\subset X_\varepsilon$ for some $\varepsilon>0$ sufficiently small. Let $O$ be a small neighborhood of the identity in $L=C_G(\Gamma_\mu)$ and $\beta_{\mathcal{N}}$ a height function from Theorem~\ref{thm;hN} taking the value $\infty$ on $HX_\varepsilon\cap OY_\infty$. Extending $\beta_{\mathcal{N}}$ from $HX_\varepsilon$ to $X$ using the value $\infty$, we are in the setting of Lemma~\ref{lem;no_mass2} and know $\nu_\infty(X_{\infty,\mathcal{N}})=1$, where $X_{\infty,\mathcal{N}}=\beta_{\mathcal{N}}^{-1}(\set{\infty})$. Thus, the lemma implies $\beta_{\mathcal{N}}(y_j)=\infty$ for all large $j$, which means that $y_j\in OY_\infty$ since $y_j\in Z\subset X_\varepsilon$. Since $O$ can be chosen arbitrarily small, (ii) is proved.

Now let us establish (i). Note that (ii) implies that for $Z\subset X$ compact, $\calS_Z(\Gamma_\mu)$ is closed in $\calS(\Gamma_\mu)$. So we only have to exhibit a limit point in $\calS(\Gamma_\mu)$ of a given sequence $(Y_j)_j$ in $\calS_Z(\Gamma_\mu)$. Thus, we may replace $Z$ by a compact neighborhood and assume that the homogeneous measures $\nu_j$ associated to the $Y_j$ all satisfy $\nu_j(Z)>0$. Then we can find $Y_j$-generic points $y_j\in Z$. Letting $\beta_\infty$ be a height function from Theorem~\ref{thm;hgeneral} that is finite on $Z$, say again with $Z\subset X_\varepsilon$, Lemma~\ref{lem;no_mass2} thus implies that any limit point $\nu$ of $(\nu_j)_j$ is a probability measure on $X$. Let us pass to a subsequence and assume that $\nu_j\to \nu$. Then $\nu$ is a $\Gamma_\mu$-invariant probability measure on $X$. By Proposition~\ref{prop;countability}, there exists $Y\in\calS(\Gamma_\mu)$ and a relatively compact neighborhood $O$ of the identity in $L$ such that $\nu(OY)>0$. We suppose that the dimension of $Y$ is minimal so that the latter holds. As in the first part of the proof, using a height function $\beta_{\mathcal{N}}$ and Lemma~\ref{lem;no_mass2}, this implies that $y_j\in OY$ for all large $j$. After passing to a subsequence, we have that $Y_j\subset l_jY_\infty$ for some $l_j\in C_G(\Gamma_\mu)$ converging to the identity and $Y_\infty=lY$ for some $l\in C_G(\Gamma_\mu)$. Then all ergodic components of the limit measure $\nu$ are homogeneous probability measures associated to some ergodic homogeneous subspace $Y'\subset Y_\infty$. If subspaces $Y'\subsetneq Y_\infty$ were to feature in the ergodic decomposition with positive weight, then another application of Proposition~\ref{prop;countability} would imply that $\nu(LY')>0$ for some $Y'\in\calS(\Gamma_\mu)$ of lower dimension, contradicting the choice of $Y$. Hence, we have established convergence of $\nu_j$ to the homogeneous probability measure associated to $Y_\infty$, proving compactness of $\calS_Z(\Gamma_\mu)$.

To obtain relative compactness of $\calS_{HZ}(\Gamma_\mu)$, note that by $H$-invariance of $\beta_\infty^{-1}(\set{\infty})$ for the height functions $\beta_\infty$ coming out of Theorem~\ref{thm;hgeneral}, we know that $\beta_\infty(x)<\infty$ for every $x\in HZ$ if $\beta_\infty$ is chosen to be finite on $Z$. Thus, Lemma~\ref{lem;no_mass}(i) implies that there exists a compact subset $R_{1/2}$ of $X$ such that $\calS_{HZ}(\Gamma_\mu)\subset\calS_{R_{1/2}}(\Gamma_\mu)$, and the latter set is compact, as shown above.

Finally, if a limit point of a sequence of probability measures in $\calS(\Gamma_\mu)\cup\set{\delta_\infty}$ has a point $x\in X$ in its support, then a subsequence is contained in $\calS_Z(\Gamma_\mu)$ for some compact neighborhood $Z$ of $x$, proving compactness of $\calS(\Gamma_\mu)\cup\set{\delta_\infty}$.
\end{proof}

\begin{proof}[Proof of Corollary~\textup{\ref{cor;hom_subspace_equidistribution}}]
Clearly, $\calS(\Gamma_\mu,Y_\infty)$ is closed in $\calS(\Gamma_\mu)$. In view of the last statement in Proposition~\ref{prop.mozes-shah}(i), we only have to show that the only possible limit point of $(Y_n)_n$ inside $\calS(\Gamma_\mu,Y_\infty)$ is $Y_\infty$. Let $Y$ be such a limit point.  By Proposition~\ref{prop.mozes-shah}, since $C_G(\Gamma_\mu)$ is assumed discrete, it follows that $Y_n\subset Y$ for infinitely many $n$. By assumption, this forces $Y=Y_\infty$, and we are done. 
\end{proof}

\subsection{Application to nilmanifolds}\label{sec;nilmanifolds}
Let $\Lambda'$ be a lattice in a connected simply connected nilpotent Lie group $N$ and let $X$ be the compact nilmanifold $N/\Lambda'$. The automorphism group $\Aut(\Lambda')$ of $\Lambda'$ is defined to be the subset of automorphisms of $N$ preserving $\Lambda'$. It is well known that any abstract automorphism of $\Lambda'$ extends to an automorphism of $N$, therefore defines an element of $\Aut(\Lambda')$ (see e.g.~\cite[\S II]{raghunathan}).

A probability measure $\mu$ on $\Aut(\Lambda')$ defines a random walk on $X=N/\Lambda'$ by nilmanifold automorphisms. Our results have the following immediate corollaries for such random walks. Under an \emph{affine submanifold} of $X$ we understand a closed subset of $X$ that is the translate of the image in $X$ of a closed subgroup of $N$.
\begin{cor}
Let $X=N/\Lambda'$ be a compact nilmanifold and $\mu$ a probability measure on $\Aut(\Lambda')$ with finite first moment such that the Zariski closure $H$ of $\Gamma_\mu$ in $\Aut(N)$ is a  connected semisimple group without compact factors.
 Then every $\mu$-ergodic $\mu$-stationary probability measure on $X$ is $\Gamma_\mu$-invariant, homogeneous, and supported on a finite union of affine submanifolds. \qed
\end{cor}
\begin{cor}\label{cor;nilm_equi}
Let $X=N/\Lambda'$ be a compact nilmanifold and $\mu$ a probability measure on $\Aut(\Lambda')$ with finite exponential moments such that the Zariski closure $H$ of $\Gamma_\mu$ in $\Aut(N)$ is a  connected semisimple group without compact factors.
Then:
	\begin{enumerate}[label=\textup{(\roman*)}]
		
		\item Every $\Gamma_\mu$-orbit closure in $X$ is a finite union of affine submanifolds.
		\item For every $x\in X$, for $\mu^{\N}$-a.e.~$(g_1, g_2 , \dots)$ one has 
		\begin{align*}
		\lim_{n\to \infty}\frac{1}{n}\sum_{k=0}^{n-1}\delta_{g_k\dotsm g_1x}=\nu_x,
		\end{align*}
		where $\nu_x$ is the homogeneous probability measure on $\overline{\Gamma_\mu x}$.
		\item The set $\calS(\Gamma_\mu)$ is compact. If $Y_n\to Y_\infty$ in $\calS(\Gamma_\mu)$, then there exists a sequence $(l_n)_n$ of $\Gamma_\mu$-invariant elements in $N$ converging to the identity such that $Y_n\subset l_nY_\infty$ for all large $n$.\qed
	\end{enumerate}
\end{cor}
The above corollaries are slight extensions of \cite[Corollary~1.3]{bq13} and \cite[Corollary~1.10]{bq132}, respectively, removing the assumption that the probability measure $\mu$ is finitely supported.
 
To deduce these corollaries from our general theorems, one needs to exhibit an embedding $X\hookrightarrow G/\Lambda$ into the quotient of a real Lie group $G$ containing $\Aut(\Lambda')$ by a lattice $\Lambda<G$. In the classical case of toral automorphisms, one has $\Aut(\Lambda')=\GL_d(\Z)$, and we may simply choose $G=\SL_{d+1}(\R)$ with its lattice $\Lambda=\SL_{d+1}(\Z)$ admitting the embedding $X=(\GL_d(\Z)\ltimes \R^d)/(\GL_d(\Z)\ltimes \Z^d)\hookrightarrow G/\Lambda$. More generally, we can define $G=\Zcl(\Aut(\Lambda'))\ltimes N$ and $\Lambda=\Aut(\Lambda')\ltimes \Lambda'$, where  $\Zcl(\Aut(\Lambda'))$ denotes the Zariski closure of $\Aut(\Lambda')$ inside $\Aut(N)$. Then $\Lambda$ is a lattice in $G$ by Borel--Harish-Chandra~\cite{bhs}, since $\Aut(\Lambda')$ is commensurable to the subgroup of integer points of $\Zcl(\Aut(\Lambda'))$ for a suitable $\Q$-structure on $\Aut(N)$ (see \cite[Theorem~2.12]{raghunathan} and its discussion). Hence, our results apply with $H=\Zcl(\Gamma_\mu)$ in view of Proposition~\ref{prop;Z-dense}. 
%The issue with log-lattices: The discussion gives commensurability with Z-points for log-lattices, i.e. ones such that log(Lambda') is a lattice in the Lie algebra. If Lambda' is arbitrary, let Lambda'' be the smallest log-lattice containing it (Z-span on the Lie algebra level). Then any automorphism preserving Lambda' also preserves Lambda'' => Aut(Lambda')<Aut(Lambda''). Since Lambda' is fin.gen. and has finite index in Lambda'', there are only finitely many possibilities for phi(Lambda') for any phi in Aut(Lambda''). We get a homomorphism from Aut(Lambda'') into a finite permutation group, whose kernel is contained in Aut(Lambda'). Hence, Aut(Lambda') has finite index in Aut(Lambda'') => commensurability with Z-points holds for general lattices

\section{Birkhoff genericity}\label{sec;birkhoff}
The aim of this section is to prove Theorem~\ref{thm;birkhoff}.
Recall that $H$ is a connected semisimple Lie group without compact factors and with finite center, $A'=\set{a(t)\for t\in \R }$ is a one-parameter $\Ad$-diagonalizable subgroup of $H$, and $U$ an $a(1)$-expanding 
subgroup of $H$ contained in $H^+_{a(1)}$. 
In particular, $U$ is connected, $\Ad$-unipotent, and normalized by $A'$. 
Moreover, having fixed a maximal compact subgroup $K$ of $H$, $K'$ is defined to be the compact group $C_K(A')\cap N_H(U)$, and $\mu$ is a probability measure on
$K'A'U\eqqcolon P\leqs H$ with finite exponential moments satisfying $\int_P\lambda(g)\dd\mu(g)>0$, where $\lambda$ is defined by the $K'A'U$-factorization $g=ka(\lambda(g))u$ for $g\in P$. Recall also that for $\omega=(g_i)_i\in P^\N$ and $n\in \N$, we write
\begin{align*}
g_{\omega, n}\df g_n \dotsm g_1 = k_{\omega, n} a_{\omega, n} u_{\omega, n}
\end{align*}
for the $K'A'U$-factorization of $g_n\dotsm g_1$. All these notations and assumptions will be understood to be in place until the end of this section.

The first lemma we prove ensures that the limit in condition~(3) of Definition~\ref{def.generated.by.expanding} exists almost surely.
\begin{lem}
	\label{lem;unipotent}
For $\mu^\N$-almost every $\omega\in P^\N$
	the sequence $(u_{\omega, n})_n$ converges to some $u_\omega\in U$. 
\end{lem}
% 1) U cannot intersect the center of H non-trivially, since any u in U lies in the unstable horospherical of some Ad-diagonalizable element a
% 2) U Ad-unipotent => nilpotent => simple connectedness can only fail if it contains a torus. this torus would map to upper triangular matrices under Ad => it is in the center of H, contradiction => U simply connected
\begin{proof}
Since $U$ does not intersect the (finite) center of $H$, the restriction $\Ad_H\colon U\to\Ad(U)$ is a Lie group isomorphism. To prove the claimed convergence, we may thus assume that $H$ is a linear group.
Let $\omega=(g_i)_i \in P^\N$. For $n \in \N$, write $g_n=k_n a_n u_n$ its (unique) factorization into $K'$, $A'$ and $U$ components. We also set $p_n=k_n a_n$. One readily observes that the term $u_{\omega,n}$ is equal to the product
\begin{align}\label{eq.prod}
u_n^{p_{n-1}\dotsm p_1} \dotsm u_3^{p_2 p_1} u_2^{p_1}u_1,
\end{align}
where we use the shorthand $g^h=h^{-1}gh$. In the product \eqref{eq.prod}, a term $u_k^{p_{k-1}\dotsm p_1}$ is equivalently expressed as $\exp\bigl(\Ad((p_{k-1}\dotsm p_1)^{-1})(\log u_k)\bigr)$.
Here, the $\log$ map is well-defined since $U$ being a unipotent linear group implies that the exponential map is a diffeomorphism from $\mathfrak{u}=\Lie(U)$ onto $U$. 
Moreover, since the Lie algebra $\mathfrak{u}$ is nilpotent, we know that $\exp\colon \mathfrak{u} \to U$ is given by $v\mapsto I + vq(v)$, where $q$ is a polynomial map. Therefore, to show that the product \eqref{eq.prod} converges for $\mu^\N$-almost every $\omega$, by a general  
convergence criterion for infinite matrix products (see e.g.~\cite[\S 8.10]{wedderburn.lectures}), it suffices to show that for $\mu^\N$-a.e. $\omega$,
\begin{align*}
\sum_{k \ge 1} \norm{\Ad((a_{k-1}\dotsm a_1)^{-1}) (\log u_k)}
\end{align*}
converges, where $\norm{\cdot}$ is an arbitrary matrix norm on $\mathfrak{u}$. We now prove this convergence. We start by observing that for $u\in U$ the logarithm $\log u$ is a polynomial in $u$. Hence, the random nilpotent elements $(\log u_k)_{k \ge 1}$ are i.i.d.\ and their distribution has a finite first moment.
By the law of large numbers, it follows that almost surely $\norm{\log u_k}= o(k)$. Almost surely, we thus obtain the bound  
\begin{equation}\label{decaying_bound}
\begin{aligned}
\norm{\Ad((a_{k-1}\dotsm a_1)^{-1}) (\log u_k)} &\le o(k) \max_{\alpha 
\in \Pi} \prod_{i=1}^{k-1} \exp(-\alpha\lambda(a_i))\\
&=o(k) \max_{\alpha 
\in \Pi} \exp\biggl(-\alpha\sum_{i=1}^{k-1}\lambda(a_i)\biggr),
\end{aligned}
\end{equation}
where 
\begin{align*}
\Pi=\set{ \alpha\in\R \for\Ad(a(t))v= e ^{\alpha t}v \text{ for all }t\in\R \text{ for some nonzero }v\in \mathfrak u }
\end{align*}
is the finite set of real numbers corresponding to the weights of $A'$ on $\mathfrak u$. Since $U$ is contained in $H_{a(1)}^+$, we have $\Pi\subset(0,\infty)$. Together with $\int_P \lambda(g) \dd\mu(g)>0$, it thus follows from the Birkhoff ergodic theorem that, $\mu^\N$-almost surely, the last term in \eqref{decaying_bound} decays exponentially. This gives the summability claimed above and hence the lemma.
\end{proof}

\begin{prop}
\label{prop;more}
Suppose that the Zariski closure of $\Ad(\Gamma_\mu)$ contains $\Ad(U)$. Then the probability measure $\mu$ is $H$-expanding.
For a discrete subgroup $\Lambda$ of a real Lie group $G$ containing $H$, any ergodic $\mu$-stationary probability measure on $G/\Lambda$ is $H$-invariant. If $\Lambda$ is a lattice in $G$, then the conclusion of Theorem~\textup{\ref{thm;orbit}} holds with $Y_x=\overline {Hx}$ and $\nu_x=\nu_{\overline{Hx}}$. 
\end{prop}

The following observations will be useful in the proof of the previous proposition. We denote by $A'_+=\set{a(t)\for t>0}$ the positive ray in $A'$.

\begin{lem}\label{lemma.fixed.point}
    Let $\Gamma$ be a subsemigroup of $P$ such that $\Gamma \cap K'A'_+U \neq \emptyset$. Then there exists $u \in U$ such that $u \Gamma u^{-1} \cap K'A'_+ \neq \emptyset$.
\end{lem}

\begin{proof}
By hypothesis there exists an element $\gamma_0 \in K'A'_+U \cap \Gamma$. Factorize $\gamma_0=p_0 u_0$ with $p_0 \in K'A'_+$ and $u_0 \in U$. Endow $\mathfrak{u}$ with some Euclidean structure. As in the proof of Lemma~\ref{lem;unipotent}, the linear map $\Ad(p_0^{-1})$ preserves the Lie algebra $\mathfrak{u}$ and any large power of it acts on $\mathfrak{u}$ as a contraction.
Moreover, since $U$ is connected and simply connected, as a consequence of the Baker--Campbell--Hausdorff formula (see e.g.~\cite[\S1.2]{corwin-greenleaf}), for every $u\in U$, the map $q_u\colon \mathfrak{u} \to \mathfrak{u}$ defined by $X \mapsto \log (\exp(X)u)$ is a polynomial map whose degree depends only on $U$ and whose coefficients depend continuously on $u$.

Using the same notation and reasoning as in the proof of Lemma~\ref{lem;unipotent}, we observe that for every $n \ge 1$, we have $\gamma_0^n=p_0^n u_0^{p_0^{n-1}} \dotsm u_0^{p_0}u_0$, with the term $u(\gamma_0^n)\df u_0^{p_0^{n-1}}\dotsm u_0^{p_0}u_0$ converging in $U$ as $n \to \infty$. From these facts, one deduces that there exists a ball $B$ in $\mathfrak{u}$ around $0 \in \mathfrak{u}$ such that for every $n \in \N$ large enough, the continuous map $f_n\colon \mathfrak{u} \to \mathfrak{u}$ defined by
\begin{align*}
f_n(X)=q_{u(\gamma_0^n)}(\Ad(p_0^{-n})X)=\log ( \exp(\Ad(p_0^{-n})X) u(\gamma_0^n))
\end{align*}
satisfies $f_n(B) \subset B$. It follows from the Brouwer fixed point theorem that $f_n$ has a fixed point $X \in \mathfrak{u}$. We claim that $u=\exp(X) \in U$ is the desired element. Indeed, since $\exp (\Ad(p_0^{-n})X)=p_0^{-n}\exp(X)p_0^n$, we have  $p_0^{-n}up_0^nu(\gamma_0^n)=u$ and hence $u \gamma_0^n u^{-1}= u p_0^n u(\gamma_0^n) u^{-1}=p_0^n \in K'A'_+$.
\end{proof}

Given $g\in P$, we write $g=k_ga_gu_g$ for its $K'A'U$-factorization.
\begin{lem}
\label{lem;dense}
For a subset $C\subset P$, let $U_C=\set{u_g\for g\in C}$ be the set of its $U$-parts. If the Zariski closure of $\Ad(C)$ contains $\Ad(U)$, then $\Ad(U_C)$ is Zariski dense in $\Ad(U)$. 
\end{lem}

\begin{proof}
Denote by $Q$ the Zariski closure of $\Ad(P)$, and observe that $\Ad(U)$ is contained in the unipotent radical $R_u(Q)$ of $Q$. Since $\Ad(K'A')$ is a linearly reductive subgroup of $Q$, %since A' is diagonalizable and K' commutes with A' and is compact
there is a Levi factor $L$ of $Q$ containing $\Ad(K'A')$ (see~\cite[Theorem~VIII.4.3]{hochschild}). Then we have $Q=L\ltimes R_u(Q)$ as algebraic groups. This implies
\begin{align*}
\Ad(U)\subset\Zcl(\Ad(C))\subset\Zcl(\Ad(K'A')\Ad(U_C))=\underbrace{\Zcl(\Ad(K'A'))}_{\subset L}\underbrace{\Zcl(\Ad(U_C))}_{\subset R_u(Q)}.
\end{align*}
We conclude that $\Ad(U)\subset \Zcl(\Ad(U_C))$, which is what we needed to show.
\end{proof}

\begin{proof}[Proof of Proposition~\textup{\ref{prop;more}}]
After choosing a maximal connected $\R$-split torus $A$ in $H$ containing $A'$, we see that the assumptions of Proposition~\ref{prop.examples2} are satisfied. Thus, $\mu$ is $H$-expanding.
Now let $\nu$ be an ergodic  $\mu$-stationary probability measure on $X=G/
\Lambda$. By Theorem~\ref{thm;rigidity}, $\nu$ is $\Gamma_\mu$-invariant and homogeneous, and the	connected component $N$ of $\Stab_G(\nu)$ is normalized by $H$.

In order to prove the statement about $H$-invariance, we can assume without loss of generality that $\Gamma_\mu$ contains an element in $K'A'_+$. Indeed, suppose that the conclusion is true for such measures; call them \emph{special}. Given an arbitrary measure $\mu$ as in the statement, by Lemma~\ref{lemma.fixed.point} we can find an element $u \in U$ such that $(\tau_u)_* \mu$ is special, where $\tau_u$ denotes conjugation by $u$. The properties in Definition~\ref{def.generated.by.expanding} are preserved by this conjugation. Then $u_* \nu$ is $(\tau_u)_* \mu$-ergodic and stationary and hence it is $H$-invariant. But since $u \in U\leqs H$, this implies that $\nu$ itself is $H$-invariant. 

So let us take $g_0=k_0a_0 \in \Gamma_\mu \cap K'A'_+$. Then, given an arbitrary $g \in \Gamma_\mu$ written as $g=k_ga_gu_g$ in its $K'A'U$ factorization, by considering a sequence $n_k$ such that $k_0^{n_k} \to e$ as $k \to \infty$, we get that the conjugates $g_0^{-n_k}g g_0^{n_k}$ converge to $k_g a_g$. This implies that $k_ga_g$ and thus also $u_g$ belongs to $\Gamma_\mu$. In other words, $\Gamma_\mu$ contains all of its $U$-parts.

We next claim that for any proper connected normal subgroup $S\leqs H$, there exists $g\in\Gamma_\mu$ whose $U$-part $u_g$ does not belong to $S$. To see this, by way of contradiction, let us suppose that all $U$-parts of elements of $\Gamma_\mu$ belong to some proper normal subgroup $S$. Using Lemma~\ref{lem;dense}, we deduce from this that $\Ad(U)\leqs\Ad(S)$, which entails that $U$ acts trivially in the adjoint representation of $H$ on $\mathfrak{h}/\mathfrak{s}$. On the other hand, the image of $a(1)$ in this representation has determinant one, so that it cannot expand all elements of $\mathfrak{h}/\mathfrak{s}$, contradicting $a(1)$-expansion of $U$.

Assuming that $H$ is not contained in $N$, we can apply the above with $S=(N\cap H)^\circ$.  Take $g=k_g a_g u_g \in \Gamma_\mu$ with $u_g\notin(N\cap H)^\circ$. By normality, also the $U$-parts of $g_0^{-n_k}g g_0^{n_k}$ do not belong to $(N\cap H)^\circ$. On the other hand, as observed above, these $U$-parts lie in $\Gamma_\mu\leqs H\cap \Stab_G(\nu)$ and converge to the identity.
This is impossible, since  $S$ is the connected component of $ H\cap \Stab_G(\nu)$ and hence an open subgroup of it. This contradiction shows that $H\leqs N$, and hence that any ergodic $\mu$-stationary probability measure $\nu$ is $H$-invariant.

Finally, applying the $H$-invariance statement to the homogeneous measure $\nu_x$ from Theorem~\ref{thm;orbit}, we see that the conclusions of that theorem hold with $Y_x=\overline{Hx}$.
\end{proof}

The following elementary but key equivariance property is the final ingredient required for the proof of Theorem~\ref{thm;birkhoff}.

\begin{lem}\label{lem;equivariance}
For $\mu^\N$-almost every $\omega=(g_i)_i \in P^\N$ and every $n \in \N$, we have
\begin{align*}
a_{\omega,n}u_\omega=k_{\omega,n}^{-1}u_{T^n \omega}g_{\omega,n},
\end{align*}
where $T\colon P^\N\to P^\N,(g_1,g_2,\dots)\mapsto (g_2,g_3,\dots)$ denotes the shift map.
\end{lem}

\begin{proof}
By Lemma~\ref{lem;unipotent}, there exists a set $\Omega$ of full $\mu^\N$-measure such that for every $\omega\in\Omega$, the sequence $u_{\omega,n}$ converges (to the limit $u_\omega$). 
Replacing $\Omega$ by $\bigcap_{i\ge 0}T^{-i}\Omega$ if necessary, we may assume that $T\Omega\subset\Omega$.
Let $\omega=(g_i)_i \in \Omega$ and $n \in \N$. Writing $g_i=k_ia_i u_i$ in its $K'A'U$ factorization, a straightforward computation shows that $u_{\omega,n}= a_1^{-1}k^{-1}_1 u_{T\omega,n-1 } k_1a_1u_1$. Passing to the limit as $n \to \infty$, we obtain $u_\omega=a_1^{-1}k^{-1}_1 u_{T\omega} g_1$. The lemma now follows by iterating the latter equality, using that $A'$ and $K'$ commute. 
\end{proof}

\begin{proof}[Proof of Theorem~\textup{\ref{thm;birkhoff}}]
Suppose the measure $\eta$ is generated by the probability measure $\mu$ supported on $P=K'A'U$ as in Definition~\ref{def.generated.by.expanding}.
 	By Theorem~\ref{thm;orbit} and Proposition~\ref{prop;more}, we know that for every $x\in X$,
 	 for $\mu^\N$-almost every $\omega=(g_i)_i\in P^\N$, the sequence of points
 	\begin{align*}
    (g_{\omega, n} x)_n
 	\end{align*}
 	is equidistributed with respect to $\nu=\nu_{\overline{Hx}}$. 
 	
   Replacing $K'$ by a subgroup, we may assume without loss of generality that 
 $\pi_{K'}(\Gamma_{\mu}) $ is dense in $K'$, where 
 $\pi_{K'}\colon P\to K'$ is the natural projection map. So the action of $\pi_{K'}(\Gamma_\mu)$ on $(K', m_{K'})$ by left translation is ergodic, where $m_{K'}$ is the Haar probability measure on $K'$. 
 By a version of Moore's ergodicity theorem (see~\cite[Theorem~III.2.5(i)]{bekka-mayer}) applied to the regular representation on the Hilbert space $L^2_0(X,\nu)$ of square integrable functions with mean zero,  the action of $\Gamma_\mu$ on $(X, \nu)$ is weakly mixing. 
 Therefore, the action of $\Gamma_\mu$ on $(X\times K'
 , \nu\times m_{K'})$ given by $ g(y, k)= (gy , \pi_{K'}(g)k ) $ is ergodic (cf.\ e.g.~\cite[Proposition~2.2]{schmidt84}). Thus it follows from~\cite[Corollary~5.5]{sw} that for almost every 
$\omega=(g_i)_i\in P^\N$, the sequence 
\begin{align*}
	(g_{\omega, n}x,  k_{\omega, n})_n
\end{align*}
is equidistributed with respect to $\nu \times m_{K'}$. Next, applying~\cite[Proposition~5.1]{sw}, this can be upgraded to almost sure equidistribution of
\begin{align}\label{equi0}
	(g_{\omega, n}x,  k_{\omega, n}, T^n\omega)_n
\end{align}
with respect to $\nu \times m_{K'}\times\mu^\N$, where $T\colon P^\N\to P^\N$ denotes the shift map. We caution here that when the support of $\mu$ is non-compact, the above equidistribution takes place in a non-locally compact space, so that the class of test functions to consider is that of bounded continuous functions. The proof of~\cite[Proposition~5.1]{sw}, however, only needs minor amending to accommodate this issue; see~\cite[Lemma~3.9]{prohaska-sert} and the short discussion before its proof. Applying the map $\omega=(g_1,g_2,\dots)\mapsto (u_\omega,g_1)$ to the equidistribution in \eqref{equi0}, we conclude that, for almost every $\omega=(g_i)_i\in P^\N$, the sequence 
\begin{align}\label{eq;equi}
	(g_{\omega, n}x,  k_{\omega, n}, u_{T^n\omega}, g_{n+1})_n
\end{align}
is equidistributed with respect to $\nu \times m_{K'} \times \tilde{\eta}$, where $\tilde{\eta}$ is a probability measure on $U\times P$ that projects to $\mu$ in the second coordinate. Again, some caution is needed at this step, since $\omega\mapsto u_\omega$ is not necessarily continuous. However, also this can be dealt with by considering Lusin sets and continuous extensions coming from Tietze's theorem as in the proof of~\cite[Proposition~5.2]{sw}.

The rest of the proof is the same as in~\cite[\S12]{sw}; we briefly reproduce it for the convenience of the reader. Given $f\in C_c(X)$, 
one considers the bounded continuous function $\varphi$ on $X\times K' \times U\times P$ defined by

\begin{align*}
\varphi(x, k, u, g)= \int_0^{\lambda(g)} f( a(t) k^{-1} ux ) \dd t,
\end{align*}
where $g=k_g a(\lambda(g)) u_g$ is the decomposition according to $P=K'A'U$. A direct calculation using the invariance of $\nu$ under $H$ shows that 
\begin{align}\label{eq;whole}
\int \varphi\dd(\nu \times m_{K'} \times \tilde{\eta})= \int_P \lambda(g) \dd \mu(g) \int _X f\dd\nu.
\end{align}

Suppose 
 $\omega=(g_i)_i$ is a generic point with respect to the equidistribution of \eqref{eq;equi} for which also Lemma~\ref{lem;equivariance} holds for every $n$.
Using only the last factor $P$ in the equidistribution, it follows that 
 \begin{align}\label{linear_growth}
\lim_{n \to \infty} \frac{\lambda(g_{\omega,n})}{n}=\lim_{n\to\infty}\frac1n\sum_{i=1}^n \lambda(g_i) = \int_P \lambda(g) \dd\mu(g)>0.
 \end{align}
 We thus obtain, by the equidistribution \eqref{eq;equi},
  
 \begin{align*}
 \int \varphi\dd(\nu \times m_{K'} \times \tilde{\eta})&=\lim_{n\to \infty}\frac1n\sum_{i=0}^{n-1} \varphi ( g_{\omega, i}x, k_{\omega, i}, u_{T^i\omega}, g_{i+1} )\\
 &=\lim_{n \to \infty} \frac1n \sum_{i=0}^{n-1} \int_0^{\lambda(g_{i+1})} f(a(t) k_{\omega,i}^{-1}u_{T^i \omega}g_{\omega,i}x)\dd t
 \\
 &=\lim_{n \to \infty} \frac1n \sum_{i=0}^{n-1} \int_0^{\lambda(g_{i+1})} f(a(t) a_{\omega,i}u_\omega x)\dd t\\
    &=\lim_{n\to \infty}\frac1n \sum_{i=0}^{n-1} \int_{\lambda(g_{\omega,i})}^{\lambda(g_{\omega,i+1})} f(a(t) u_\omega x)\dd t\\
    &= \lim_{n\to \infty}\frac{\lambda(g_{\omega,n})}{n}
    \frac{1}{\lambda(g_{\omega,n})}\int_0^{\lambda(g_{\omega,n})} f(a(t) u_\omega x)\dd t\\
    &= \int _P \lambda(g) \dd \mu(g) 
    \lim_{n\to \infty}\frac{1}{\lambda(g_{\omega,n})}\int_0^{\lambda(g_{\omega,n})} f(a(t) u_\omega x)\dd t,
\end{align*} 
where we used Lemma~\ref{lem;equivariance} in the third equality and that $\lambda(g_{\omega,i+1})=\lambda(g_{\omega,i})+\lambda(g_{i+1})$ in the fourth.
Together with \eqref{eq;whole}, this implies 
\begin{align}\label{eq.birkhoff.end}
\lim_{n\to \infty}\frac{1}{\lambda(g_{\omega,n})}\int_0^{\lambda(g_{\omega,n})} f(a(t) u_\omega x)\dd t = \int f \dd\nu.
\end{align}

Finally, notice that since the random variables $\lambda(g_{\omega,n})-\lambda(g_{\omega,n-1})=\lambda(g_n)$ are i.i.d.\ with a distribution that has a finite first moment, it follows from the law of large numbers that almost surely 
\begin{align}\label{eq.first.moment}
\lambda(g_{\omega,n})-\lambda(g_{\omega,n-1}) =  o(n).
\end{align}
Now \eqref{linear_growth}, \eqref{eq.birkhoff.end} and \eqref{eq.first.moment} together imply the Birkhoff genericity of $u_\omega x$ with respect to $(a(t))_{t>0}$ and $\nu$. 
\end{proof}

\section{Connections to Diophantine approximation on fractals}\label{sec;diophantine}
The goal of this section is to explain the connection between random walks and Diophantine approximation on affine fractals, prove a general result (Theorem~\ref{thm;dioph.insection}) which will imply Theorem~\ref{thm.dioph.intro} on Diophantine properties of Bedford--McMullen carpets, and mention some further directions.

\subsection{Weighted Diophantine approximation and Dani--Kleinbock flow}
To begin with, we recall basic notions in Diophantine approximation of matrices and the connection to homogeneous dynamics.

\subsubsection{Badly approximable matrices and Dirichlet improvability}
Let $m,n \in \N$ be positive integers, $\mathbf{r} =(r_1,\dots,r_m) \in (0,1]^m$ and $\mathbf{s} =(s_1,\dots,s_n) \in (0,1]^n$ be such that $\sum_{i=1}^m r_i=\sum_{j=1}^n s_j=1$ and $M \in \Mat_{m \times n}(\R)$ a matrix with rows $M_1,\dots,M_m$. Then $M$ is called \emph{$(\mathbf{r},\mathbf{s})$-badly approximable} or \emph{badly approximable for the weights $(\mathbf{r},\mathbf{s})$} if there exists a constant $C>0$ such that
\begin{align}\label{eq.dioph.mat}
    \max_{1\le i\le m}\abs{M_i\mathbf{q}-p_i}^{1/r_i}\cdot\max_{1\le j\le n}\abs{q_j}^{1/s_j} \ge C
\end{align}
for every $(\mathbf{p}, \mathbf{q}) \in \Z^m\times(\Z^n\setminus\set{0})$. Otherwise, $M$ is called \emph{$(\mathbf{r},\mathbf{s})$-well approximable}.

One can see by Dirichlet's principle, or by Blichfeldt and Minkowski's convex body results, that for every matrix $M \in \Mat_{m \times n}(\R)$, there exist infinitely many pairs $(\mathbf{p}, \mathbf{q}) \in \Z^m\times(\Z^n\setminus\set{0})$ such that the left-hand side of \eqref{eq.dioph.mat} is bounded above by $1$. As a consequence of a general form of Khintchine's theorem~\cite{schmidt.general.khintchine}, the set of $(\mathbf{r},\mathbf{s})$-badly approximable matrices is a Lebesgue null set. However, it has everywhere-full Hausdorff dimension; see \cite[Corollary~4.5]{kleinbock-weiss.Dirichlet} and \cite[\S5.4]{kw.modified.schmidt}.

Given weights $(\mathbf{r}, \mathbf{s})$, an equivalent way to express the aforementioned consequence of the Dirichlet principle is to say that for every matrix $M \in \Mat_{m \times n}(\R)$ and for every $t>0$, the following system of inequalities has a solution in $(\mathbf{p},\mathbf{q}) \in \Z^m \times (\Z^n\setminus\set{0})$:
\begin{align*}
\abs{M_i\mathbf{q}-p_i}\le e^{-tr_i} \quad \text{and} \quad \abs{q_j}\le e^{ts_j}\qquad(1\le i\le m,1\le j\le n).
\end{align*}
One says that the matrix $M  \in \Mat_{m \times n}(\R)$ is \emph{$(\mathbf{r},\mathbf{s})$-Dirichlet improvable} if there exists $\varepsilon \in (0,1)$ such that for every $t>0$ large enough, the following system of inequalities has a solution in $(\mathbf{p},\mathbf{q}) \in \Z^m \times (\Z^n\setminus\set{0})$:
\begin{align*}
\abs{M_i\mathbf{q}-p_i}\le \varepsilon e^{-tr_i} \quad \text{and} \quad \abs{q_j}\le \varepsilon e^{ts_j}\qquad(1\le i\le m,1\le j\le n).
\end{align*}

In the special case where the weights $(\mathbf{r},\mathbf{s})$ are given by $(\mathbf{m}, \mathbf{n})$---by which we mean that $r_i=1/m$ and $s_j=1/n$ for all $i,j$---the notion of Dirichlet improvability was introduced and studied by Davenport--Schmidt, who showed that the set of $(\mathbf{m},\mathbf{n})$-Dirichlet improvable matrices has zero Lebesgue measure~\cite{DS2} and that every $(\mathbf{m},\mathbf{n})$-badly approximable matrix is $(\mathbf{m},\mathbf{n})$-Dirichlet improvable~\cite{DS1}. The former result was generalized to arbitrary weights $(\mathbf{r},\mathbf{s})$ by Kleinbock--Weiss~\cite{kleinbock-weiss.Dirichlet}.

\subsubsection{Dani--Kleinbock flow}
Let $G=\PGL_d(\R)$, $\Lambda=\PGL_d(\Z)$, and set $X=G/\Lambda$. It is easy to see that $X$ can alternatively be written as $\SL_d(\R)/\SL_d(\Z)$, which can be identified with the space of unimodular lattices in $\R^d$ via $g\SL_d(\Z)\leftrightarrow g\Z^d$. For every $\varepsilon>0$, we define
\begin{align*}
K_\varepsilon\df\set*{g\Lambda \in X\for g\in\SL_d(\R),\, \max_{i=1,\dots,d} \abs{(g \mathbf{v})_i}\ge \varepsilon \text{ for every } v \in \Z^d \setminus \set{0}}.
\end{align*}
Viewing $X$ as space of unimodular lattices in $\R^d$, $K_\varepsilon$ is nothing but the subset of lattices all of whose nonzero vectors have length at least $\varepsilon$ in the supremum norm.
The collection of sets $K_\varepsilon$ is clearly decreasing in $\varepsilon$. For $\varepsilon<1$ the set $K_\varepsilon$ has non-empty interior, and for $\varepsilon>1$ one has $K_\varepsilon= \emptyset$, as can be seen by Minkowski's convex body theorem from geometry of numbers. Moreover, Mahler's compactness criterion states that the sets $K_\varepsilon \subset X$ for $\varepsilon>0$ are compact and that a subset of $X$ is relatively compact if and only if it is contained in one of the $K_\varepsilon$. 

Now let $d=m+n$ and denote by $x_0$ the identity coset in $X=G/\Lambda$. The Dani--Kleinbock correspondence principle---observed first by Dani~\cite{dani.correspond} and developed further, among others, by Kleinbock~\cite{kleinbock.duke} and later Kleinbock--Weiss~\cite{kleinbock-weiss.Dirichlet}---states that, loosely speaking, the Diophantine properties of a matrix $M\in \Mat_{m \times n}(\R)$ are encoded in the behavior of the trajectory of $u_Mx_0$ inside $X$ under suitable one-parameter diagonal subgroups of $G$, where $u_M\df
(\begin{smallmatrix}
 I_m & -M \\
 0 & I_n
\end{smallmatrix})$. We are going to use this principle in the form of the following proposition.
Given weights $(\mathbf{r}, \mathbf{s}) \in (0,1]^m \times (0,1]^n$ as before, let $a(t)$ denote the one-parameter subgroup of $G$ corresponding to  $a(1)=\diag(e^{r_1},\dots,e^{r_m},e^{-s_1},\dots,e^{-s_n})$.

\begin{prop}[Dani--Kleinbock correspondence]\label{prop.dani.correspond}
A matrix $M \in \Mat_{m \times n}(\R)$ is
\begin{itemize}
\item \textup{(\cite{kleinbock.duke})} $(\mathbf{r},\mathbf{s})$-badly approximable if and only if the forward-orbit $\set{a(t) u_M x_0 \for t \ge 0}$ is relatively compact in $X$, and
\item \textup{(\cite{kleinbock-weiss.Dirichlet})} $(\mathbf{r},\mathbf{s})$-Dirichlet improvable if and only if there exists $\varepsilon \in (0,1)$ such that $a(t) u_M x_0 \notin K_\varepsilon$ for every $t>0$ large enough.
\end{itemize}
\end{prop}

An obvious consequence of this proposition is that given weights $(\mathbf{r},\mathbf{s})$, if the forward orbit $\set{a(t) u_M x_0\for t\ge 0}$ associated to a matrix $M \in \Mat_{m \times n}(\R)$ is dense in $X$, then $M$ is $(\mathbf{r},\mathbf{s})$-well approximable and not $(\mathbf{r},\mathbf{s})$-Dirichlet improvable.

In fact, the ergodic theoretic approach that we adopt will allow us to establish the following finer Diophantine property. 

\begin{de}
Given weights $(\mathbf{r},\mathbf{s})$ and the associated one-parameter diagonal group $(a(t))_{t\in\R}$, a matrix $M \in \Mat_{m \times n}(\R)$ is said to be of \emph{$(\mathbf{r},\mathbf{s})$-generic type} if the forward-orbit $(a(t) u_M x_0)_{t> 0}$ equidistributes to the Haar measure $m_X$ on $X$. 
\end{de}

\subsection{Matrix sponges and self-affine measures}\label{subsec:sponges-aff-meas}
Here we briefly describe the iterated function system (IFS) construction of affine fractals and introduce the subfamily of affine fractals (matrix sponges) and self-affine measures whose Diophantine properties will be studied in the subsequent part.

\subsubsection{Affine fractals}\label{sub.affine_frac}
Let $\phi$ be an affine transformation of $\R^D$ given by $\phi(x)=Ax+b$ where $A \in \GL_D(\R)$ and $b \in \R^D$. It is called 
\emph{contracting} if the operator norm of its linear part $A$ with respect to the standard Euclidean structure of $\R^D$ satisfies $\norm{A}<1$. We shall refer to a finite tuple $(\phi_1,\dots,\phi_k)$ of contracting affine transformations $\phi_i$ of $\R^D$ as a \emph{contracting affine IFS}. Given such an IFS, there exists a unique non-empty compact subset $\mathcal{K}$ of $\R^D$ satisfying $\mathcal{K}=\bigcup_{i=1}^k \phi_i(\mathcal{K})$, referred to as the \emph{attractor} of the IFS $(\phi_1,\ldots,\phi_k)$. Putting less emphasis on the IFS, $\mathcal{K}$ is also called an \emph{affine fractal} or \emph{self-affine set}.  In the particular case where all the $\phi_i$ are similarities, the attractor $\mathcal{K}$ is also called a \emph{self-similar set}.

The \emph{coding map} $\pi$ associated to a contracting affine IFS is the map $\set{1,\dots,k}^\N \to \R^D$ defined by 
\begin{align}\label{eq.defn.coding.map}
\pi((i_1,i_2,\dots))\df\lim_{n \to \infty} \phi_{i_1} \circ \phi_{i_2} \circ \dots\circ \phi_{i_n}(x)
\end{align}
for some $x \in \R^D$; the limit is independent of $x$. The image of 
the coding map $\pi$ is precisely the affine fractal $\mathcal{K}$, and we have the following
equivariance property with respect to the shift map $T$ on $\set{1,\dots,k}^\N$:
\begin{align}\label{eq.coding.equivariance}
\pi((i_1,i_2,\dots))=\phi_{i_1} \pi(T(i_1,i_2,\dots))=\phi_{i_1}\pi((i_2,\dots)).    
\end{align}

Our results on random walks on homogeneous spaces also allow us to study a more general situation where the IFS is not required to be finite and where one can allow contraction to only take place on 
average. To describe this, let $I$ be a compact set and $I \to \GL_D(\R)\ltimes \R^D$, $i \mapsto \phi_i=(A_i,b_i)$ a continuous 
map, where $A_i$ denotes the linear part and $b_i$ the translation part of $\phi_i$. 
Let $\mu$ be a probability measure on $I$. We shall refer to the couple $(I,\mu)$ as a \emph{contracting-on-average affine IFS} if there exists $N \in \N$ such 
that
\begin{align}\label{eq.def.contracting.av}
\int \log \norm{A_{i_N}\dotsm A_{i_1}} \dd\mu^N(i_1,\dots,i_N)<0.
\end{align}
This definition does not depend on the choice of operator norm.

Using only boundedness of the translation parts, it is not hard to see that the limit $\lim_{n \to \infty} \phi_{i_1} \circ \dotsm \circ \phi_{i_n}(x)$ exists and does not depend on $x \in \R^D$ whenever the sequence $(\norm{A_{i_1}\dotsm A_{i_n}})_{n\ge 1}$ decays fast enough (e.g.\ exponentially). Under the contraction-on-average assumption, this holds for $\mu^\N$-almost every $(i_1,i_2, \dots)$, as one can see using submultiplicativity of the operator norm and Kingman's subadditive ergodic theorem. In this case, we thus obtain a measurable map $\pi\colon I^\N \to \R^D$ that we shall refer to as the \emph{coding map} of $(I,\mu)$. Note that the subset $\Omega$ of elements of $I^\N$ for which the previous limit exists satisfies $T \Omega \subset \Omega$ and on this set the coding map $\pi$ satisfies the equivariance relation \eqref{eq.coding.equivariance}. %To clarify this technical point, whenever we speak of the attractor $\mathcal{K}$ associated to $(I,\mu)$, we shall mean the set $\pi(\Omega) \subset \R^D$. However, we note that in fact the main object of interest will be the self-affine measure $\pi_* \mu^\N$ rather than the set $\mathcal{K}$ itself. In the case of a contracting affine IFS $(I,\mu)$, the coding map $\pi$ is defined everywhere on $I^\N$, it is continuous and its image $\mathcal{K}$ coincides with the support of $\pi_* \mu^\N$ for any $\mu$ having full support $I$.

Finally, we shall say that an IFS $(I,\mu)$ of affine maps of $\R^D$ is \emph{irreducible} if there does not exist a proper affine subspace $W$ of $\R^D$ such that $\phi_i (W)=W$ for $\mu$-almost every $i \in I$.

\subsubsection{Self-affine measures}\label{subsub.self-affine}
Given a contracting-on-average affine IFS $(I,\mu)$, the probability measure $\nu_\mu=\pi_* \mu^\N$ on $\R^D$ is called the associated \emph{self-affine measure} (or \emph{self-similar measure} if the IFS comprises only similarities). It is with respect to these self-affine measures that we will study the typical Diophantine behavior of vectors in $\R^D$ or more generally matrices in $\Mat_{m \times n}(\R)$. The measure $\nu_\mu$ is the unique stationary probability measure for the random walk on $\R^D$ given by the IFS; see~\cite{diaconis}. In the case of a finite IFS, i.e.\ when $I=\set{1,\dots,k}$, this just means that $\nu_\mu$ is the unique probability measure on $\R^D$ satisfying $\nu_\mu=\sum_{i=1}^k \mu(i)(\phi_i)_* \nu_\mu$.

For a finite contracting IFS consisting of similarities of $\R^D$, under a separation condition (see~\cite{hutchinson}), the Hausdorff measure on the attractor $\mathcal{K}$ is given by a self-similar measure which is also the unique measure on $\mathcal{K}$ whose pointwise dimension matches the Hausdorff dimension of the similarity fractal $\mathcal{K}$. For genuinely self-affine fractals, the situation is considerably more complicated (see e.g.\ ~\cite{barral-feng, kaenmaki-vilp, morris-sert-nobernoulli, morris-sert.nounique} and the references therein). On the other hand, for the Bedford--McMullen carpets introduced in \S\ref{sec;dioph.intro} and their higher-dimensional generalizations, there exists a unique ergodic shift-invariant probability measure on $\set{1,\dots,k}^\N$ whose pushforward $\nu$ by the coding map has full Hausdorff dimension~\cite{kenyon-peres}. Moreover, this measure $\nu$ is self-affine. In dimension $2$, it was already explicitly constructed and used by McMullen~\cite{mcmullen}, and is referred to as the McMullen measure in the literature.

\subsubsection{Matrix sponges}\label{subsub.sponges}
We now describe the family of affine fractals and self-affine measures that will be of interest to us. Let $\mathbf{r}=(r_1,\dots,r_m) \in (0,1]^m$ and $\mathbf{s}=(s_1,\dots,s_n) \in (0,1]^n$ be such that $\sum_{i=1}^mr_i=1=\sum_{j=1}^ns_j$. Consider the diagonalizable one-parameter groups $A'_\mathbf{r}\subset \GL_m(\R)$ and $A'_\mathbf{s}\subset\GL_n(\R)$ given by $\set{a_\mathbf{r}(t)\df\diag(e^{tr_1}, \dots, e^{tr_m})\for t \in \R}$ and $\set{a_\mathbf{s}(t)\df\diag(e^{ts_1}, \dots, e^{ts_n})\for t \in \R}$ respectively. Denote by $K_\mathbf{r}$ the compact group $C_{\GL_m(\R)}(A'_\mathbf{r}) \cap \Orth_m(\R)$ and similarly for $K_\mathbf{s}$ substituting $\mathbf{s}$ for $\mathbf{r}$ and $n$ for $m$.

We identify the real vector space $\Mat_{m \times n}(\R)$ with $\R^{mn}$ and consider affinities $\phi$ of $\Mat_{m \times n}(\R)$ of the type
\begin{align}\label{eq.describe.affinity}
M \mapsto A_1 M A_2 + B,
\end{align}
where $B \in \Mat_{m \times n}(\R)$, $A_1\in\GL_m(\R)$ and $A_2\in\GL_n(\R)$. We will refer to affinities of this form as \emph{matrix affinities} and use the notation $(A_1,A_2,B)$ to denote such a map. If a matrix affinity $\phi$ can be written as $\phi=(A_1,A_2,B)$ with $A_1 \in  a_\mathbf{r}(t) K_\mathbf{r}$ and $A_2 \in  a_\mathbf{s}(t) K_\mathbf{s}$ for some $t \in \R$, then we call it an \emph{$(\mathbf{r},\mathbf{s})$-matrix sponge affinity}.  
Given a contracting-on-average IFS $(I,\mu)$ of $(\mathbf{r},\mathbf{s})$-matrix sponge affinities, we call the associated attractor $\mathcal{K}$ an \emph{$(\mathbf{r},\mathbf{s})$-matrix sponge}. 

A cautionary remark is in order about our terminology.  In the literature, the terms ``carpet'' (in dimension $2$) or ``sponge'' (in general dimension) are used to describe self-affine fractals associated to IFS's whose linear parts are simultaneously diagonalizable with non-trivial (i.e.\ non-scalar) diagonals. However, the matrix sponge affinities that we just described also comprise many similarities of $\R^{mn}$. Similarities of $\R^{mn}$ of this form are called ``algebraic similarities'' by Simmons--Weiss~\cite[\S8.4]{sw}, which thus form a strict subclass of matrix sponge affinities. For example, specializing to $n=1$ we can record that the class of $(\mathbf{m},1)$-matrix sponges contains all self-similar fractals in $\R^m$ and the class of $(\mathbf{r},1)$-matrix sponges contains many examples of Bedford--McMullen carpets and their higher-dimensional analogues---the self-affine Sierpi\'{n}ski sponges---for suitably chosen weight vectors $\mathbf{r}$.
%even for  vector approximation, our setting does not comprise every diagonal transformation, because our setting puts a restriction on proportions. For example sponge IFS whose linear parts consist of $\diag(e^{-10},e^{-1})$ cannot be studied by our setting $\diag(e^{-10},e^{-6})$ can. 

\subsection{Relation with random walks and consequences}
Here we first adapt the constructions of Simmons--Weiss~\cite{sw} relating algebraic similarities with elements of $\PGL_d(\R)$ to the more general setting of matrix affinities. Then, we state and prove the main result of this section (Theorem~\ref{thm;dioph.insection}) on Diophantine properties of matrix sponges.

\subsubsection{Embedding matrix sponge affinities into \texorpdfstring{$\PGL_d(\R)$}{PGL_d(R)}}\label{subsub.embed}
Let $d=m+n$.
Given a matrix affinity $\phi=(A_1,A_2,B)$ of $\Mat_{m \times n}(\R)$, where $A_1 \in \GL_m(\R)$, $A_2 \in \GL_n(\R)$ and $B \in \Mat_{m \times n}(\R)$, we consider the element $\hat{A}_\phi$ of $\PGL_d(\R)$ corresponding to the matrix
\begin{align*}
\hat{A}_\phi=\begin{pmatrix}
A_1 & 0\\
0 & A_2^{-1}  
\end{pmatrix}.
\end{align*}

The following basic relation in $\PGL_d(\R)$, which is readily verified, plays a key role in transferring the results on random walks on homogeneous spaces to the study of Diophantine properties of matrix sponges:
For $M \in \Mat_{m \times n}(\R)$, we have
\begin{align}\label{eq.key.dioph}
\hat{A}_\phi u_M \hat{A}^{-1}_\phi u_B=u_{\phi(M)},
\end{align}
where, as before, $u_M=(\begin{smallmatrix}I_m & -M \\ 0 & I_n\end{smallmatrix})$. We set $g_\phi\df\hat{A}_\phi^{-1}u_B\in\PGL_d(\R)$. Given matrix affinities $\phi_1,\dots,\phi_n$, iterating \eqref{eq.key.dioph} yields
\begin{align}\label{eq.iterated.key}
g_{\phi_n}\dotsm g_{\phi_1}=\hat{A}^{-1}_{\phi_n}\dotsm \hat{A}^{-1}_{\phi_1}u_{\phi_1\dotsm\phi_n(0)}.
\end{align}

\subsubsection{Genericity of typical points on matrix sponges}

To state the following main result of this section, recall that given a contracting-on-average affine IFS $(I,\mu)$, we denote by $\pi$ the associated coding map and by $\nu_\mu$ the pushforward of the Bernoulli measure $\beta=\mu^\N$ by $\pi$.

\begin{thm}\label{thm;dioph.insection}
Let $(I,\mu)$ be an irreducible contracting-on-average IFS consisting of $(\mathbf{r},\mathbf{s})$-matrix sponge affinities. Then $\nu_\mu$-almost every point of $\R^{mn}$  is of $(\mathbf{r},\mathbf{s})$-generic type; in particular, $(\mathbf{r},\mathbf{s})$-well approximable and not $(\mathbf{r},\mathbf{s})$-Dirichlet improvable.
\end{thm}

In the classical case where $(\mathbf{r},\mathbf{s})=(\mathbf{m},\mathbf{n})$, this result corresponds to Simmons--Weiss'~\cite[Theorem~8.11]{sw}, which implies one of the main results of that article (\cite[Theorem~1.2]{sw}).
We are going to see in the proof that the contracting-on-average assumption in the theorem above amounts to asking that the $\mu$-average of the $t$-parameters associated to the $(\mathbf{r},\mathbf{s})$-matrix sponge affinities $\phi$ in the IFS is negative. This allows for easy checking of this condition.

\begin{rem}
The conclusion of Theorem~\ref{thm;dioph.insection} also holds for any measure $\tilde{\nu}_\mu$ obtained as pushforward of $\nu_\mu$ by an affine transformation of the linear space $\Mat_{m \times n}(\R)$ of the form $M \mapsto \alpha M \beta +\gamma$, where $\alpha \in \GL_{m}(\R)$ commutes with the diagonal group $A'_\textbf{r}$, $\beta \in \GL_{n}(\R)$ commutes with $A'_\textbf{s}$ and $\gamma \in \Mat_{m \times n}(\R)$. In particular, these Diophantine properties of $\nu_\mu$ are invariant under translation of $\nu_\mu$. %apply the theorem to the conjugated IFS, for appropriate conjugates of the compact groups $\Orth_m(\R)$ and $\Orth_n(\R)$ appearing in the definitions of $K_{\mathbf{r}}$ and $K_{\mathbf{s}}$
\end{rem}

We will deduce the theorem above by combining Theorem~\ref{thm;birkhoff}, Dani--Kleinbock correspondence and the introduced constructions. To ease notation, we will assume from now on that $I$ is already a set of matrix sponge affinities, with $\mu$ living thereon.

\begin{proof}[Proof of Theorem~\textup{\ref{thm;dioph.insection}}]
Recall that $\mathbf{r} =(r_1,\dots,r_m) \in (0,1]^m$ and $\mathbf{s} =(s_1,\dots,s_n) \in (0,1]^n$ are such that $\sum_{i=1}^m r_i=1=\sum_{j=1}^n s_j$, where $m$ and $n$ are positive integers. Let $d=m+n$ and set $G=H=\PGL_d(\R)$ and $\Lambda=\PGL_d(\Z)$. Moreover, we let $A'=\set{a(t)\for t\in\R}$ be the one-parameter diagonalizable subgroup of $G$ containing $a(1)=\diag(e^{r_1}, \dots, e^{r_m}, e^{-s_1}, \dots, e^{-s_n})$, and denote by $A'_+$ its positive ray $\set{a(t) \for t > 0 }$. Take $U$ to be the unipotent subgroup of $G$ given by the image of $\Mat_{m \times n}(\R)$ under the map $M \mapsto u_M$. It is $a(1)$-expanding (see Example~\ref{ex;main}).
In view of Dani--Kleinbock correspondence and Theorem~\ref{thm;birkhoff}, all we need to check is that the pushforward $\eta_0$ of the self-affine measure $\nu_\mu$ by the map $M \mapsto u_M$ is generated by $a(1)$-expanding random walks in the sense of Definition~\ref{def.generated.by.expanding}. 

We first define the probability measure $\mu_0$ on $G$. Given a matrix affinity $\phi=(A_1,A_2,B)$, recall the notation $g_\phi=\hat{A}_\phi^{-1}u_B \in \PGL_d(\R)$ introduced in \S \ref{subsub.embed}. We take 
\begin{align}\label{eq.def.mu}
    \mu_0\df c_* \mu,
\end{align}
the pushforward of $\mu$ by the map $c\colon\phi \mapsto g_\phi$. Then
it follows from our constructions that $\mu_0(P)=1$, where $P=K'A'U$ is defined as before Definition~\ref{def.generated.by.expanding}. Moreover, we claim that the contraction-on-average assumption implies that $\int_P \lambda(g)\dd\mu_0(g)>0$.  To see this, endow $\Mat_{m \times n}(\R) \cong \R^{mn}$ with the standard Euclidean structure and denote by $\norm{\cdot}$ the associated operator norm on $\End(\Mat_{m \times n}(\R))$. Given an $(\mathbf{r},\mathbf{s})$-matrix sponge affinity $\phi$, let us denote by  $A_\phi \in \End(\Mat_{m \times n}(\R))$ its linear part. By definition, we may write $\phi=(A_1,A_2,B)$ as in \eqref{eq.describe.affinity} with $A_1 \in a_\mathbf{r}(t)K_\mathbf{r}$ and $A_2 \in a_\mathbf{s}(t)K_\mathbf{s}$ for some $t\in\R$. Observe that by construction, the $t$-parameter is given by $t=-\lambda(g_\phi)$.
This implies that
\begin{align*}
\norm{A_\phi}\ge e^{\kappa t}=e^{-\kappa\lambda(g_\phi)},   
\end{align*}
where $\kappa\df\min_{i,j}(r_i+s_j)>0$. Plugging this inequality into the contraction-on-average property  \eqref{eq.def.contracting.av} and observing that $\lambda(g_{\phi_N\dotsm\phi_1})=\lambda(g_{\phi_N})+\dots+\lambda(g_{\phi_1})$ yields $\int_P \lambda(g)\dd\mu_0(g)=\int \lambda(g_\phi)\dd\mu(\phi)>0$, hence the claim.

We now show that the irreducibility assumption entails that $U\leqs\Zcl(\Gamma_{\mu_0})$. As in the proof of Proposition~\ref{prop;more}, we will first reduce to the case of \emph{special} measures $\mu_0$ for which $\Gamma_{\mu_0}$ contains an element of $K'A'_+$. Indeed, given a general $\mu_0$ as in \eqref{eq.def.mu}, using that $\int_P \lambda(g)\dd\mu_0(g)>0$ and Lemma~\ref{lemma.fixed.point}, it follows that there exists $u_0 \in U$ such that the pushforward by conjugation $(\tau_{u_0})_* \mu_0$ is special. The closed group generated by the support of $(\tau_{u_0})_* \mu_0$ is $u_0 \Gamma_{\mu_0} u_0^{-1}$ and if the Zariski closure of this group contains $U$, then that of $\Gamma_{\mu_0}$ also contains $U$. Moreover, this conjugation corresponds to conjugating the IFS by a translation, so that also irreducibility is preserved. So we now suppose that $\mu_0$ is special. 
Then as in the proof of Proposition~\ref{prop;more}, for every $g\in\Gamma_{\mu_0}$ written $g=k_ga_gu_g$ in its $K'A'U$-factorization, we know that also $k_ga_g$ and $u_g$ belong to $\Gamma_{\mu_0}$.
It follows that for every $g \in \Gamma_{\mu_0}$, the one-parameter unipotent subgroup of $U$ containing $u_g$ is contained in the Zariski closure of $\Gamma_{\mu_0}$. Now consider the connected unipotent group $V=\Zcl(\Gamma_{\mu_0}) \cap U$ and let $W_V$ be the corresponding subspace of $\R^{mn}$ under (the inverse of) the identification $M \mapsto u_M$. We claim that the subspace $W_V$ is invariant by the IFS of matrix sponge affinities. Indeed, by construction, for any $\phi=(A_1,A_2,B)$ in the IFS, the unipotent part $u_B$ of the associated element $g_\phi$ belongs to $V$ and hence $B \in W_V$. Moreover, for any $g \in \Gamma_{\mu_0}$, its $K'A'$-component $k_ga_g$ normalizes $V$. In view of \eqref{eq.key.dioph}, this translates to the statement that for any $\phi$ of the IFS, the linear part of $\phi$ leaves the subspace $W_V$ invariant. It follows that the subspace $W_V$ of $\R^{mn}$ is invariant by the IFS. Hence, by the irreducibility hypothesis, we have  $W_V=\R^{mn}$, or equivalently, $V=U$.

It remains to check that the measure $\eta_0$ coincides with the image of $\mu_0^\N$ under the map $\omega \mapsto u_\omega$ defined by Lemma~\ref{lem;unipotent}. To do this, let $\omega=(g_{\phi_1},g_{\phi_2},\dots)$. By definition of the coding map \eqref{eq.defn.coding.map} and the map $\omega \mapsto u_\omega$, it suffices to observe that for every $n\in \N$, factorizing $g_{\phi_n} \dotsm g_{\phi_1}$ as $k_{\omega,n}a_{\omega,n}u_{\omega,n}$ with $k_{\omega,n} \in K'$, $a_{\omega,n} \in A'$ and $u_{\omega,n} \in U$, we have $u_{\omega,n}=u_{\phi_1\dotsm \phi_n (0)}$; see \eqref{eq.iterated.key}. This finishes the proof.
\end{proof}

Finally, we state and prove the corollary of the previous theorem regarding the higher-dimensional analogues of Bedford--McMullen carpets, which was announced at the end of \S\ref{sec;dioph.intro}. These higher-dimensional fractals are constructed by the exact analogue in $\R^m$ of the procedure for Bedford--McMullen carpets described before Theorem~\ref{thm.dioph.intro}, now using pairwise distinct integers $a_1,\dots,a_m\ge 2$ and a division of $[0,1]^m$ into an $a_1\times\dots\times a_m$-grid. A fractal $\mathcal{K}$ obtained in this way is called a \emph{self-affine Sierpi\'{n}ski sponge}. 
Analogous to the McMullen measure on a Bedford--McMullen carpet, there exists a natural probability measure $\nu_{\mathcal{K}}$ on $\mathcal{K}$:  
Identifying $[0,1]^m$ with the $m$-torus and denoting by $T$ the toral endomorphism corresponding to the diagonal matrix $A=\diag(a_1,\dots,a_m)$, $\nu_{\mathcal{K}}$ is the unique
$T$-invariant ergodic probability measure on $\mathcal{K}$ of full Hausdorff dimension (see Kenyon--Peres~\cite{kenyon-peres}).

\begin{cor}\label{cor.sierpinski}
Let $m\ge 2$ and $a_1,\dots,a_m\ge 2$ be pairwise distinct integers satisfying 
\begin{align}\label{ensures_r_is_good}
\frac1m\sum_{j\neq i}\log a_j< \log a_i<\frac{2}{m-1}\sum_{j\neq i}\log a_j 
\end{align}
for $i=1,\dots,m$. Let $\mathcal{K}\subset\R^m$ be a self-affine Sierpi\'{n}ski sponge invariant under the toral endomorphism $T$ corresponding to the matrix $A=\diag(a_1,\dots,a_m)$ such that $\mathcal{K}$ is not contained in any affine hyperplane. Then for the choice of weights
\begin{align}\label{choice_of_r}
\mathbf{r}=\biggl(\frac{m\log a_i-\sum_{j\neq i}\log a_j}{\sum_j\log a_j}\biggr)_{1\le i\le m},
\end{align}
the set of $\mathbf{r}$-badly approximable vectors on $\mathcal{K}$ has measure zero with respect to $\nu_{\mathcal{K}}$.
\end{cor}
This corollary directly implies Theorem~\ref{thm.dioph.intro}.
\begin{proof}
We start by noting that $\mathcal{K}$ is the attractor of a finite contracting affine IFS $(\phi_1,\dots,\phi_k)$, where $\phi_i\colon x\mapsto A^{-1}x+b_i$ with translation vectors $b_i\in\prod_j\set{0,\frac{1}{a_j},\dots,\frac{a_j-1}{a_j}}$. If $I=\set{1,\dots,k}$ and $\pi\colon I^\N\to\R^m$ denotes the associated coding map, the proof of \cite[Theorem~1.2]{kenyon-peres} shows that $\nu_{\mathcal{K}}=\nu_\mu=\pi_*\mu^\N$ for some probability measure $\mu$ on $I$ of full support. Then the assumption that $\mathcal{K}$ is not contained in any affine hyperplane implies that the IFS $(I,\mu)$ is irreducible. We wish to arrange that the $\phi_i$ can be seen as $(\mathbf{r},1)$-matrix sponge affinities. By definition, this means that we have to write the linear part $A^{-1}=\diag(a_1^{-1},\dots,a_m^{-1})$ as $e^ta_{\mathbf{r}}(t)$ for some $t\in\R$, where $a_{\mathbf{r}}(t)=\diag(e^{tr_1},\dots,e^{tr_m})$. Solving the resulting system of equations under the constraint $r_1+\dots+r_m=1$ yields the weights specified by \eqref{choice_of_r}. The condition \eqref{ensures_r_is_good} ensures that $\mathbf{r}\in(0,1)^m$. Hence, Theorem~\ref{thm;dioph.insection} applies and gives the desired conclusion.
\end{proof}

We end our discussion of Diophantine approximation by mentioning that our approach has serious limitations when trying to tackle the general problem of understanding the measure-theoretic size of badly approximable vectors or matrices---weighted or not---in general self-affine fractals. Even seemingly tractable cases---e.g.\ $\mathbf{r}$-badly approximable vectors on an affine fractal for which $\mathbf{r}$ represents the average contraction ratio---require a further understanding of diagonal flows and, frustratingly, remain open.

%%% APPENDIX on epimorphic subgroups and subalgebras %%%%%%%%%%

\appendix

\section{Epimorphic subgroups and subalgebras}\label{app;epi}
In category theory, an epimorphism is by definition a morphism $f\colon A\to B$ satisfying the right cancellation property: $g\circ f=h\circ f$ implies $g=h$ for any two morphisms $g,h$ from $B$ to another object of the category. In categories where morphisms are maps with certain properties between underlying sets, the epimorphism property is equivalent to the question whether the values on the image of $f$ uniquely determine morphisms from $B$ to other objects. In this case, surjective morphisms are clearly epimorphisms. In many familiar categories, the converse, i.e.\ that only surjective morphisms can be epimorphisms, is also true. For example, this holds in the categories of $C^*$-algebras, groups, finite groups, all Lie algebras over a field $k$, and finite-dimensional Lie algebras over a field $k$ of positive characteristic; see~\cite{bergman,reid}. However, there are notable exceptions. These include the categories of finite-dimensional Lie algebras over a field of characteristic $0$ and that of algebraic groups, which are our main interest. The corresponding lines of study were initiated by Bergman~\cite{bergman} and Bien--Borel~\cite{bien-borel1,bien-borel2}, respectively, who proved the following.
\begin{prop}\leavevmode
\begin{enumerate}[label=\textup{(\roman*)}]
    \item \textup{(}\cite[Corollary~3.2]{bergman}\textup{)} Let $\mathfrak f\subset\mathfrak g$ be finite-dimensional Lie algebras over a field $k$. Then the inclusion $\mathfrak f\hookrightarrow\mathfrak g$ is an epimorphism if and only if in every finite-dimensional representation of $\mathfrak g$, the subspaces annihilated by $\mathfrak f$ and $\mathfrak g$ coincide.
    \item \textup{(}\cite[Theorem~1]{bien-borel1}\textup{)} Let $\mathbf{G}$ be a Zariski connected linear algebraic group over an algebraically closed field $k$, and $\mathbf{F}\leqs \mathbf{G}$ an algebraic subgroup. Then the inclusion $\mathbf{F}\hookrightarrow \mathbf{G}$ is an epimorphism if and only if in every finite-dimensional algebraic representation of $\mathbf{G}$, the subspaces of $\mathbf{F}$- and $\mathbf{G}$-fixed vectors coincide.
\end{enumerate}
\end{prop}

We take this representation-theoretic characterization as the defining property of an epimorphic subgroup of a semisimple real Lie group.
\begin{de}\label{de;epimorphic}\leavevmode
\begin{enumerate}[label=(\roman*)]
    \item Let $\mathfrak f$ be a subalgebra of a finite-dimensional real Lie algebra $\mathfrak g$. We say that $\mathfrak f$ is \emph{epimorphic} in $\mathfrak g$ if for any finite-dimensional real
    representation of $\mathfrak g$, the subspaces annihilated by $\mathfrak f$ and $\mathfrak g$ coincide.
    \item Let $G$ be a connected semisimple real Lie group. A subgroup $F$ of $G$ is said to be \emph{epimorphic} in $G$ if for every finite-dimensional representation of $G$, the vectors fixed by $F$ are also fixed by $G$.
\end{enumerate}
\end{de}
In the literature, it has been common to only introduce and study the concept of epimorphic subgroups for algebraic groups. Let us therefore check that our definition coincides with the usual one when the groups involved are algebraic.
\begin{prop}\label{prop;defs_coincide}
Let $G$ be a Zariski connected semisimple real algebraic group and $F$ a Lie subgroup of $G$ such that $F^\circ$ is Zariski dense in $F$. Suppose that $F$ is epimorphic in $G$ in the category of real algebraic groups, meaning that in every finite-dimensional real algebraic representation of $G$, the vectors fixed by $F$ are also fixed by $G$. Then $F^\circ$ is epimorphic in $G^\circ$ in the sense of Definition~\textup{\ref{de;epimorphic}}.
\end{prop}
To be precise, by $G$ being a real algebraic group we mean that $G=\mathbf{G}(\R)$ is the group of real points of an underlying complex algebraic group $\mathbf{G}$ defined over $\R$, and a real algebraic representation is the restriction to real points of an algebraic representation of $\mathbf{G}$ defined over $\R$. Moreover,  $F^\circ$ and $G^\circ$ denote the connected components of $F$ and $G$, respectively, in the Lie group topology.
It is easy to see that the converse of the proposition is also true. Finally, we remark that $F$ is epimorphic in $G$ in the category of real algebraic groups if and only if $F$ is epimorphic in $\mathbf{G}$ in the category of complex algebraic groups.

The idea of the proof of the proposition above is to pass to the Lie algebra level, where all representations are algebraic thanks to semisimplicity. The following two lemmas enable this step.
\begin{lem}
    \label{lem;proof} 
    Let $G$ be a connected semisimple Lie group and $F$ a closed subgroup of $G$. If $\mathfrak f=\Lie(F)$ is an epimorphic subalgebra of $\mathfrak g=\Lie(G)$, then $F$ is epimorphic in $G$. 
\end{lem}
\begin{proof}
A  representation of $G$ naturally induces a representation of its Lie algebra. 
A vector that is $F$-fixed on the Lie group level is then $\mathfrak f$-annihilated on the Lie algebra level. Therefore, such vectors are annihilated by $\mathfrak g$ and hence fixed by $G$, since $G$ is connected.
\end{proof}

\begin{lem}\label{lem;converse}
Let $F$ and $G$ be as in Proposition~\textup{\ref{prop;defs_coincide}}. 
Then  $\mathfrak f=\Lie(F)$ is an epimorphic subalgebra of $\mathfrak g=\Lie(G)$.
\end{lem}
\begin{proof}
If $\mathfrak f$ is not an  epimorphic subalgebra of $\mathfrak g$, 
then using complete reducibility of $\mathfrak g$-representations, we can find a non-trivial irreducible representation
 $\rho\colon \mathfrak g\to \mathfrak {gl}(V)$ such that the subspace 
 \begin{align*}
 V_0=V^{\mathfrak f}=\set{v\in V\for \rho(f)v=0\text{ for all }f\in\mathfrak f}
 \end{align*}
 is nonzero. Let $\mathfrak{g}_\C$ and $V_\C$ be the complexifications of $\mathfrak g$ and $V$, respectively. It follows from the discussion in~\cite[\S8]{onishchik} (Theorem~1 and Corollary~1) that either (1) $\mathfrak{g}_\C$ acts irreducibly on $V_\C$, or (2)
  $V$ has a complex structure  and $\mathfrak{g}$ acts by $\C$-linear transformations. In both cases, we thus obtain an irreducible complex representation of $\mathfrak{g}_\C$ (either on $V_\C$ or on $V$), which we denote by $\rho_\C$. We also set $k=\R$ in the first case and $k=\C$ in the second, and record that since $\mathfrak g$ acts $k$-linearly, the subspace $V_0$ is $k$-invariant.

 We claim that there exists $n\in \N$ such that the tensor product representation
$\rho ^{\otimes_{k} n} $ of $\mathfrak g$ lifts to a real algebraic representation of $G$. 
Assuming the claim and using that $F^\circ$ is Zariski dense in $F$, we find that $V_0^{\otimes_ {k}n}$ is a nonzero $F$-fixed  subspace of $V^{\otimes_{k}n}$.
Since $F$ is an epimorphic subgroup of $G$ in the algebraic category, the space $V_0^{\otimes_ {k}n}$  is $G$-fixed. It follows that $\mathfrak g$ annihilates $V_0^{\otimes_ {k}n}$, hence
$\mathfrak g$ annihilates $V_0$. This contradicts the assumption that $(\rho,V)$ is a non-trivial irreducible representation, and thus establishes the statement of the lemma.

It remains to prove the claim. Let $\mathbf{G}$ be a Zariski connected semisimple complex algebraic group defined over $\R$ such that $G=\mathbf{G}(\R)$. Then $\mathfrak g_\C$ is the Lie algebra of $\mathbf{G}$. By~\cite[Corollary~A.4.11]{conrad} there is a simply connected algebraic cover $\tilde{\mathbf{G}}$ of $\mathbf{G}$ defined over $\R$.

In case (1), since the representation $\rho_\C\colon\mathfrak g_\C\to\mathfrak{gl}(V_\C)$ is algebraic by semisimplicity, it lifts to an irreducible algebraic representation $\tilde{\mathbf{G}} \to \GL(V_\C)$ defined over $\R$ (with respect to the real structure on $V_\C$ given by $V$). The kernel $\mathbf{N}$ of the covering
map $\tilde{\mathbf{G}}\to \mathbf{G}$ is finite and central. By Schur's lemma and irreducibility, $\mathbf{N}$ thus acts on $V_\C$ by scalar multiplication by roots of unity. Therefore, there exists $n\in \N$ such that $\mathbf{N}$ acts trivially on $V_\C^{\otimes_\C n}$. Since the representation of $\tilde{\mathbf{G}}$ on $V_\C^{\otimes_\C n}$ is defined over $\R$, we deduce that it induces a real algebraic representation of $G$ on $V^{\otimes_{k}n}=V^{\otimes _\R n}$. 

In case (2), $\rho_\C\colon\mathfrak g_\C\to\mathfrak{gl}(V)$ lifts to an irreducible algebraic representation $\tilde{\mathbf{G}}\to\GL(V)$. By the same argument as in the first case, for some $n\in\N$ the kernel $\mathbf{N}$ of the covering map acts trivially on $V^{\otimes_\C n}$. Hence, the action of $\tilde{\mathbf{G}}$ on $V^{\otimes_{k}n}=V^{\otimes _{\C} n}$ factors through an algebraic representation of $\mathbf{G}$. By restriction of scalars, we can view $\mathbf{G}$ and $\GL(V^{\otimes _{\C} n})$ as groups of real points of algebraic groups defined over $\R$. Composing the map $G\to \mathbf{G}$ with the representation of $\mathbf{G}$ on $V^{\otimes _{\C} n}$ we obtain the desired lift of $\rho^{\otimes_{k}n}$.
\end{proof}

\begin{proof}[Proof of Proposition~\textup{\ref{prop;defs_coincide}}]

By Lemma~\ref{lem;converse}, 
 $\mathfrak f=\Lie(F)$ is an epimorphic subalgebra of $\mathfrak g=\Lie(G)$.  Then Lemma~\ref{lem;proof} implies that $F^\circ$ is epimorphic in $G^\circ$ in the sense of Definition~\ref{de;epimorphic}(ii).
\end{proof}

\bibliographystyle{plain}
\bibliography{refs}

\end{document}